\documentclass[11pt]{amsart}

\usepackage{amsmath,amssymb,amscd}
\usepackage[mathscr]{eucal}
\newcommand{\Spec}{{\mathrm{Spec}\, }}

\newcommand{\bbar}[1]{{\overline{#1}}}

\newcommand{\cA}{{\mathcal A}}
\newcommand{\cB}{{\mathscr B}}
\newcommand{\cC}{{\mathscr C}}

\newcommand{\cE}{{\mathcal E}}
\newcommand{\cF}{{\mathcal F}}
\newcommand{\cH}{{\mathcal H}}
\newcommand{\cM}{{\mathcal M}}

\newcommand{\cO}{{\mathcal O}}

\newcommand{\cS}{{\mathscr S}}
\newcommand{\cT}{{\mathscr T}}
\newcommand{\cX}{{\mathcal X}}

\renewcommand{\sp}[1]{{\mathrm{Spec}}}

\newcommand{\scr}[1]{\mathcal {#1}}

\renewcommand{\l}{\lambda}
\renewcommand{\o}{\omega}
\newcommand{\sg}{\mathrm{Sing}}
\newcommand{\ma}{\mathrm{Max}}
\renewcommand{\max}{\mathrm{max}}

\theoremstyle{plain}
\newtheorem{thm}{Theorem}[section]
\newtheorem{lem}[thm]{Lemma}
\newtheorem{pro}[thm]{Proposition}

\theoremstyle{definition}
\newtheorem{Def}[thm]{Definition}
\newtheorem{rem}[thm]{Remark}
\newtheorem{exa}[thm]{Example}
\newtheorem{voi}[thm]{}
\begin{document}

 % April 21, 2009
\bigskip 
\title[Algorithmic equiresolution of families of singularities]{Algorithmic equiresolution of possibly non-reduced families of singularities}

\author{Augusto Nobile}

\address{Louisiana State University \\
Department of Mathematics \\
Baton Rouge, LA 70803, USA}

\subjclass[2000]{14B05, 14E15, 14F05, 14D99}

\keywords{Resolution algorithm, embedded variety, coherent ideal, basic object}

%\date{\today}
\email{nobile@math.lsu.edu}
%\maketitle
%\section{Introduction}
%\label{S:intro}

\begin{abstract}
This paper studies the concept of algorithmic equiresolution of a family of embedded varieties or ideals, which means a simultaneous resolution of such a family compatible with a given (suitable) algorithm of resolution in characteristic zero.  The paper's approach is more indirect: it primarily considers the more general case of families of basic objects (or marked ideals). 
A definition of algorithmic equiresolution is proposed, which applies to families whose parameter space $T$ may be non-reduced, e.g., the spectrum of a suitable artinian ring. Other definitions of algorithmic equiresolution are also discussed. These are geometrically very natural,  but the parameter space $T$ of the family must be assumed regular. It is proven that when $T$ is regular, all the  proposed definitions are equivalent. 
\end{abstract}
\maketitle
\section*{Introduction}
\label{S:intro}

%\maketitle

After Hironaka solved the problem of resolution of singularities of algebraic varieties (working over fields of characteristic zero) attempts were made to provide a more constructive proof (algorithmic or canonical resolutions). These efforts were successful and, 
at present, theories of algorithmic resolution of singularities  are well established (see, e.g., \cite{BM}, \cite{BEV}, \cite{Cu}, \cite{EH}, \cite{EH}, \cite{Kol}, \cite{W}). Once algorithmic methods to resolve singularities of algebraic varieties are available, a natural question is to study simultaneous resolution (or {\it equiresolution}) of families of algebraic varieties, in a way compatible with a given resolution algorithm. The problem of algorithmic resolution, at least in what seems to be the most crucial case, namely that of an embedded variety (i.e., a closed subvariety of a regular ambient one), is closely related to that of {\it principalization} of a sheaf of ideals on a regular variety. Indeed, with a suitable algorithm for principalization of ideals it is possible to obtain one for desingularization of embedded varieties. (See, e.g, \cite{BEV} or \cite{BMF}.)

In \cite{EN} some basic results on algorithmic equiresolution are obtained in the case of families of ideals or of embedded schemes, parametrized by smooth (or at least reduced) varieties. (See  also \cite{BEV} and \cite{BMM}). Namely,  two notions of equiresolution are proposed, one requiring that the centers for the transformations leading to the algorithmic resolution of the family 
 be smooth over the parameter variety $T$ (condition AE), the other requiring the local constancy of an invariant associated to each fiber (condition $\tau$). It is proved that (under suitable compactness  assumptions, i.e., that certain morphisms be proper) both notions are equivalent.  An application of these results discussed in \cite{EN} is  the construction of an ``equisolvable''stratification of $H_{red}$, where $H$ is the Hilbert scheme parametrizing subschemes of a  scheme $W$, smooth over a field, having a fixed Hilbert polynomial $P$. For a number of reasons the restriction that the scheme  parametrizing the family   must be reduced is not very satisfactory (for instance, $H$ above might not be reduced). In this paper a definition of equiresolution that makes sense for families parametrized by non necessarily reduced schemes $T$ is proposed. We call it {\it condition (E)}. Other notions, called conditions ($A$), ($F$), and ($C$) are also presented. Condition ($A$) corresponds to condition (AE) of \cite{EN}, the other formalize variants of the idea that the algorithmic resolution of the family should induce the algorithmic resolution of the fibers. 
 
 The main objective of this paper  is to prove that when the parameter space $T$ is regular all these conditions are equivalent. Assuming  the properness of certain projections one proves that these are also equivalent to condition ($\tau$)   of \cite{EN}. 
 
 Condition (E) is defined on the basis of the work done in the article \cite{NE}. In it, the crucial case of ``infinitesimal'' families (of deformations), i.e., parametrized by $\Spec A$, $A$ a (suitable) artinian ring, is studied. In this situation there is a single fiber, and a notion of algorithmic equiresolution is introduced, which attempts to make precise the intuitive idea that the different steps in the algorithmic resolution of the fiber ``nicely'' spread over the infinitesimal parameter space $\Spec A$. Essentially, condition ($E$) requires that for all $t \in T$ (the parameter space) the naturally induced infinitesimal family over $\Spec (A_{n,t})$, where $A_{n,t} = {\cO}_{T,t}/M_{T,t}^{n+1}$, $M_{T,t}$ being the maximal ideal of $\cO_{T,t}$, be algorithmically equisolvable, in the sense of \cite{NE}. 
 
 In this paper, as well as  in \cite{NE}, we work on a settting more restricted than that of \cite{EN}, where a general resolution algorithm satisfying certain properties (the so-called {\it good algorithms}) was used. Here (or in \cite{NE}) we work with a specific algorithm, namely that of \cite{EV} or \cite{BEV} (also discussed in \cite{Cu}). Most likely other algorithms (like those of \cite{BM}, \cite{W}, etc.) could be used  to obtain similar results. Probably, in the future, when more experience is gained, a presentation at least as general as that of \cite{EN} will be available. 
 
 Also, as in \cite{NE}, most of the paper is devoted to the study of families of {\it basic objects} (called {\it presentations} in \cite{BM} and {\it marked ideals} in \cite{W}). Results in this  technically simpler case  imply similar ones in the more interesting case of families of ideals or embedded varieties. 
 
 The article consists of seven sections. In section \ref{S:AL} we review the algorithm (for basic objects over fields of characteristic zero) of \cite{BEV} of \cite{EV}, which  we call the V-algorithm, and a variation thereof (the W-algorithm, which borrows some techniques from \cite{W}). The W-algorithm was used in \cite{NE}.  We prove in detail the fact that the V- and the W-algorithms are the same. 
In section \ref{S:FB} we introduce families of basic objects and study some of their basic properties.
 In section \ref{S:BE}, working with families of basic objects parametrized by a regular scheme, 
we introduce the aforementioned conditions ($A$), ($F$), ($C$) and ($\tau$) and prove the equivalences among them already cited.  In section \ref{S:CE}, after reviewing necessary results from 
\cite{NE}, we introduce condition ($E$). The equivalence of conditions ($A$) and ($E$) is proved in sections \ref{S:S} and \ref{S:E}.  In order not to break the main line of reasoning, in section  \ref{S:S} we gather several results that are used in the proof, which is presented in section \ref{S:E}.  In section \ref{S:IV}, we explain how the work already done easily leads to similar results
 for families of ideals (or of triples $(W,I,E)$, where $W$ is a smooth variety, $I$ a sheaf of ideals on $W$, and $E$ a finite sequence of smooth hypersurfaces of $W$ with normal crossings), and families of embedded schemes and varieties.
 
 Trying to keep this article reasonably brief, and because this is still work in progress, we do not attempt to discuss applications similar to those found in \cite{EN}. We hope to return to these questions in the future.
 
  It is my pleasure to thank O. Villamayor and S. Encinas for useful discussions and their encouragement.

\section{Algorithms}\label{S:AL}
\begin{voi}
\label{V:a1.0}
 In general, we shall use the notation and terminology of \cite{H}. We describe next a few  exceptions. If $W$ is a scheme, a $W$-ideal will mean a coherent sheaf of 
${\cO}_W$-ideals.  If $I$ is a $W$-ideal, the symbol $\mathcal V (I)$ will denote the closed subscheme of $W$ defined by $I$. As usual, $V(I)$ will denote the closed subset of the underlying topological space of $W$ of zeroes of $I$. If $Y$ is a closed subscheme of a scheme $W$, the symbol $I(Y)$ denotes the $W$-ideal defining $Y$. An algebraic variety over a field $k$ will be a reduced algebraic $k$-scheme. If $W$ is a reduced scheme, a never-zero $W$-ideal is a $W$-ideal $I$ such that the stalk $I_x$ is not zero for all $x \in W$, in general $I$ is a never-zero ideal of $W$ if $I{\cO}_{W'}$ is never-zero, with $W'=W_{red}$. 

The term \emph{local ring} will mean 
 \emph{noetherian local ring}. In  general, the maximal ideal, or radical, of a local ring $R$ will be denoted by $r(R)$. Often, we write $(R,M)$ to denote the local ring $R$ with maximal ideal $M$. The order of an ideal $J$ in the local ring $(A,M)$ is the largest integer $s$ such that $J\subseteq M^s$. 
 
 If $W$ is a noetherian  scheme, $I$ is a $W$-ideal and $x \in W$, then $\nu _x(I)$ denotes the order of the ideal 
 $I_x$ of ${\cO}_{W,x}$.

The symbols ${\bf N}$, ${\bf Z}$ and ${\bf Q}$ will denote the natural, integral and rational numbers respectively. 
\end{voi}

\begin{voi}
\label{V:a2.0} 
We shall work with the following collection $\cS$ of schemes (see \cite{BEV}, section 8, or \cite{EN} (1.1) (c)). Let $\cS$ be the class of regular, equidimensional $k$-schemes $W$, where $k$ is a (variable) field of characteristic zero, satisfying the following  condition: the scheme $W$ admits a finite covering by open sets of the form $\Spec (R)$, where $R$ is a  noetherian, regular $k$-algebra, with the property that ${\rm Der}_k(R)$ is a projective $R$-module of rank $n = \dim W$; moreover, for any maximal ideal $M$ of $R$, $\dim (R_M)=n$ and $R/M$ is an algebraic extension of $k$. 
\end{voi}

\begin{voi}
\label{V:a3.0} A {\it basic object in $\cS$}, or a {\it $\cS$-basic object}, is a four-tuple $B=(W,I,b,E)$, where $W \in \cS$, $I$ is a never-zero $W$-ideal, $b >0$ is an integer and $E=(H_1, \ldots, H_m)$ is  a sequence of distinct regular hypersurfaces of $W$ (i.e., each $H_i$ is a regular Weil divisor of $W$) with normal crossings (\cite{BEV}, 2.1). 

The smooth $k$-scheme $W$ is the {\it underlying scheme} of $B$, denoted by $us(B)$. The {\it dimension of $B$} is the dimension of the scheme $us(B)$.
\end{voi}

\begin{voi}
\label{V:a4.0} The singular set $\sg (B)$ of the basic object of \ref{V:a3.0} is 
$\{ x \in W: \nu _{x}(I) \ge b\}$. This  is a closed set of $W$. Indeed, one may introduce an operation $\Delta ^{i}$ on $W$-ideals, $i \ge 1$, so that $\sg (B)=V({\Delta}^{b-1}(I))$. Concerning $\Delta=\Delta ^1$, if $w$ is a closed point of $ W$ and $x_1, \ldots, x_r$ is a regular system of parameters of $R={\cO}_{W,w}$ and $D_i$ is the derivation associated to $x_i$ (the ``partial derivative'' with respect to $x_i$), then ${\Delta}(I)_w$ is the ideal of $R$ generated by $I_w \cup \{D_i f : f \in I_w, i=1,\ldots,r\}$, ${\Delta}^i$ is defined by iteration. One defines $\Delta(I)$ more intrinsically by means of suitable Fitting ideals of $\Omega_{{\nu(I)}/k}$, $k$ the base field. (See \cite{BEV}, section 13, for more details.)  When we want to stress the fact that  $W$ is a scheme over a field $k$ we shall write ${\Delta}^{i}(I/k))$ rather than ${\Delta}^{i}(I))$. 
\end{voi}

\begin{voi}
\label{V:a5.0} 
{\it Pemissible transformations}. If $B=(W,I,b,E)$, $E=(H_1, \ldots, H_m)$ is a basic object, a  {\it permissible center} for $B$ is a closed subscheme $C \subseteq \sg(B) \subset W$ having normal crossings with the hypersurfaces $H_i$, $i=1, \ldots, m$. In particular, $C$ is regular, see 
\cite{BEV}, Definition 2.1. We define the {\it transform} of $B$ with center $C$ as the  basic object 
$B_1=(W_1,I_1,b,E_1))$ where $W_1$ is the blowing-up of $W$ with center $C$ (so, we have a natural morphism 
$W_1 \to W$), $E'=(H'_1, \ldots, H'_m, H_{m+1})$, where $H'_i$ is the strict transform of $H_i$ ($i=1, \ldots m$) and $H_{m+1}$ is the exceptional divisor. Finally $I_1$ is the {\it controlled transform} of $I$, that is $I_1:=I(H_{m+1})^{-b}I{\cO}_{W_1}$ (cf. \cite{BEV}, section 3). The process of replacing $B$ by such new basic object $B_1$ is is called the (permissible) {\it transformation} of $B$ with center $C$, often  denoted by $B \leftarrow B_1$. We'll also write $B_1:=\cT(B,C)$. A sequence of basic objects 
$B_0 \leftarrow \cdots \leftarrow B_r$, where each arrow $B_j \leftarrow B_{j+1}$ is a permissible transformation, is called a {\it permissible sequence} (of basic objects and transformations) 

If $W \leftarrow W_1$ is as above, we define the {\it proper transform} $\bar{I_1}$ of $I$ to $W_1$ as the $W_1$-ideal 
${\cE}^{-a} I {\cO}_{W_1}$, where ${\cE}$ defines the exceptional divisor and the exponent $a$ is as large as possible. This integer is constant along each irreducible component of the center $C$ used, but in general not globally constant. Given a permissible sequence of basic objects as above, define (inductively) ${\bar I}_{i+1}$ to be the proper transform of $\bar{I_i}$, for all $i$. We also say that $\bar{I_i}$ is the {\it proper transform} of $I$ to $B_i$.

In the notation above, if $f:W' \to  W$ is a smooth morphism, we define the pull-back of $B$ to $W'$ as the basic object $B'=(W',J,b,E')$, with $J=I{\cO}_{W'}$ and 
$E'=(L_1, \ldots, L_m)$, where $L_i= f_{-1}(H_i), \, i=1, \ldots, m$.
\end{voi}
\begin{voi}
\label{V:a6.0}
A {\it resolution} of the basic object $B$ is a permissible sequence 
$B_0 \leftarrow \cdots \leftarrow B_r$ such that $\sg (B_r)=\emptyset$. 

An algorithm of resolution (for  basic objects in $\cS$) is a rule that associates to each positive integer $d$  a totally ordered set $\Lambda^{(d)}$, with a minimum element $0_d$,  and for any given basic object (in $\cS$)  $B_{0}=(W_{0},I_{0},b_{0},E_{0})$ of dimension $d$ (with 
 $\sg(B_0) \not=\emptyset$), functions as follows. In all cases, their value is $0_d$ outside the singular set. Otherwise, we have an upper semicontinuous function 
$g_0: W_0 \to {\Lambda}^{d}$, (taking finitely many values), such that 
$C_0 = {\rm Max}(g_0) = \{w \in W_0: g_0(w) ~ {\rm is ~ maximum }\}$ is a permissible center. If $B_1=\cT (B_0,C_0)$ has $\sg (B_1) \not= {\emptyset}$ and $W_1=us(B_1)$,  we have sa function 
$g_1: W_1 \to {\Lambda}^{d}$, such that $C_1 = {\rm Max}(g_0)$ is a well determined $B_1$-permissible center, which we blow-up, and so on. Eventually we get in this way a permissible sequence 
$$ (1) \quad B_0\leftarrow B_1 \leftarrow \cdots \leftarrow B_r$$ 
	
We require that this be a resolution, i.e., $\sg (B_r)=\emptyset$. Moreover, this should be stable under etale base change, meaning that if $B'$ is the basic object obtained by pull-back under an etale map $W' \to W$, then the pull-back of the sequence (1) may be identified to the resolution sequence of $B'$ (and the new resolution functions are induced by the original ones). In a similar way it is required compatibility with respect to change of the base field $k$. This is called the {\it algorithmic resolution sequence of $B$}. Notice that if $\sg(B_0)=\emptyset$, no resolution function (or resolution center) is attached to $B_0$, in this case 
the algorithmic resolution sequence of $B_0$ is $B_0$ itself.
\end{voi}

\begin{voi}
\label{V:a7.0}
In this paper we shall work with a specific algorithm, namely that of \cite{BEV} or \cite{EV}. In order to review the basic construction, we must recall first some auxiliary notions.

(i)  {\it Monomial objects.} A basic  object
$B=(W,I,b,E)$, $E=(H_1, \ldots, H_m)$ (in $\cS$) is 
 {\it monomial} if for each $w \in W$ we have:
$I_w = {I(H_1)} ^{\alpha _1 (w)} \ldots {I(H_m)} ^{\alpha _m (w)}$, and each function
$\alpha _i: W \to {\bf Z}$
is constant on each irreducible component of  $H_i$ and zero outside $H_i$. If $B$ (in $\cS$) is monomial, we define functions ${\Gamma _i}$, $i=1,2,3$, with domain $W$, which take the value $0_d$ if $w \notin S=\sg (B)$ and otherwise are as follows.

If $w \in S$, $\Gamma _1 (w)$ is the smallest integer $p$ such that there are indices $i_1, \ldots, i_p$ such that
$$(1) \quad \alpha_{i_1} (w) + \cdots + {\alpha_{i_p}(w)} \geq b $$
Consider, for $w \in S$,  the set $P'(w)$ of sequences  $i_1, \ldots, i_p$ satisfying $(1)$ above, and let  $\Gamma _2 (w)$ be the maximum of the rational numbers
$({\alpha _{i_1} (w)} + \cdots + {\alpha _{i_p}(w)})/b$, for
$(i_1, \ldots, i_p) \in P'(w)$.

If $w \in S$, let $P(w)$ be the set of all sequences 
$(i_1, \ldots, i_p,0,0, \ldots)$ such that  
$({\alpha _{i_1} (w)} + \cdots + {\alpha _{i_p}(w)})/b = \Gamma _2 (w)$, and define 
$\Gamma _3 (w)$ to be the maximum of the set $P(w)$, when we use the lexicographical order.

Finally, one defines a function $\Gamma$ (or $\Gamma _B $) from $W$ to
${\bf Z} \times {\bf Q} \times {\bf Z}^{\bf N}$
 by the formula 
$ \Gamma (w) = (- \Gamma _1 (w), \Gamma _2 (w), \Gamma _3 (w))$. When the target is lexicographically ordered the function $\Gamma$ is upper semicontinuous.

Let max$(\Gamma_3) = (i_1, \ldots, i_p,0,0, \ldots)$ and take
$C= H_{i_1} \cap \cdots \cap H_{i_p}$. Then, it turns out that $C$ is a permissible center for the basic object $B$ and that, if $B_1={\cT}(B,C)$,   
${\mathrm {max}}( \Gamma _{B_1}) < {\mathrm {max}}( \Gamma _{B})$. 
  Iterating this process, after a finite number of steps we reach a situation where the singular locus is empty. For details see \cite{BEV}, section 20.

\smallskip
(ii) {\it The functions $t_r$}. If 
$$(1) \quad B_0 \leftarrow \cdots \leftarrow B_r$$
 is a sequence of basic objects (in $\cS$) and permissible transformations (where we write 
$B_j = (W_j, I_j, b, E_j)$, for all $j$), we define, for $x \in \sg (B_r)$, w-$\mathrm{ord}_r(x)$, or simply $\o _r (x)$, by the formula 
$\omega _r (x) :=   \nu _x ({\bar {I_r}}) / b$ (where ${\bar {I_r}}$ denotes  the proper transform of $I_0$ to $W_r$ (\ref{V:a5.0}). It can be proved that if 
in our sequence the center of each transformation is contained in 
${\ma} ({\omega}_j)$ (the set of points where ${\omega}_j$ reaches its maximum 
${\max} ({\omega}_j)$),  then 
${\max} ({\o}_{j-1}) \geq {\max} ({\o}_j)$, for $j < r$. 
 
 The functions $t_r$  are defined by induction on the length $r$ of a 
$\o$-permissible sequence as (1) above.  If $r=0$, for $x \in \sg (B_0)$ we write 
$t_0(x) = (\omega _0 (x), n_0(x))$, where $n_0(x)$ is the number of hypersurfaces in $E_0$ containing $x$. Assume that $t_j= (\omega _j, n_j)$ was defined (on $\sg (B_j)$) for $j < r$ and that in our sequence (1) is $t$- permissible. This means that each center $C_i$ used in the blowing-ups is contained in the subset of $\sg (B_i)$ where $t_i$ reaches its maximum value (in particular then (1) is $\o$-permissible). 
Let $s$ be the smallest index such that 
${\max}(\o _s) = {\max}(\o _r)$ and $E^{-}_r$ the collection of the hypersurfaces in $E_r$ which are strict transforms of those in $E_s$. Then, for 
$x \in \sg(B_r)$ we set: $t_r(x) = (\omega _r(x), n_r(x))$, where $n_r(x)$ is the number of hypersurfaces in $E^{-}_r$ containing $x$. A $B_r$-center which is contained in 
Max($t_r$) will be called $t$-permissible. It can be proved that in a $t$-permissible sequence the sequence ${\mathrm {max}} (t_j)$ is non-increasing.

\end{voi}

\begin{voi}
\label{V:a8}
{\it Equivalence} We shall recall the well-known notion of {\it equivalence} of basic objects, essentially due to Hironaka. Aside from permissible transformations, we shall consider other operations that may be applied to a basic object $B=(W,I,b,E), \, E=(H_1, \ldots, H_m)$.

\smallskip

(a) {\it Extensions}.  If $B$ is as above and $n$ a positive integer, let   $ W[n]:= W \times {\mathbf A}_k^n $ and    $(W[n],I[n], b,E[n])=B[n]$ the basic object  naturally
induced (by pull-back) by $B$ by means of  the natural projection ${W \times {\mathbf A}_k^n} \to W$. 
 The object  $W[1]$ is called 
  the {\it extension} of $B$ (terminology from \cite{EV}).
  
  \smallskip
  
  (b) {\it Open restrictions}. If $U$ is an open subset of $W$, the pull-back $B_{|U}$ of $B$ via the inclusion $U \subseteq W $ is  the {\it (open) restriction} of $B$ to $U$.
  
  \smallskip
  
  Now we say that basic objects 
   $B=(W , I, b, E)$ and $B'=(W , J, c, E)$  in $\cS$ are {\emph {equivalent}} if whenever 
   $$(1) \qquad B=B_0 \leftarrow \cdots  \leftarrow B_s ~,$$
   $$(2) \qquad B'=B'_0 \leftarrow \cdots  \leftarrow B'_s$$
 are sequences, where each arrow stands for either a permissible transformation, an extension or an open restriction (where   if  the $i$-th arrow of (1) corresponds to a permissible transformation, 
  then the $i$-th arrow of (2)  corresponds to a permissible transformation with the same center, and vice-versa), then we have: $\sg (B_s) = \sg(B'_s)$. 
\end{voi}

 \begin{voi}
\label{V:a9} 
 
Extensions  play an essential role in the proof of  an important result of Hironaka, whose ingenuous  method of proof is sometimes called ``Hironaka's trick,''  see \cite{BEV}, section 21 or \cite{BMF} 6.1. It says that if    $B=(W , I, b, E)$ and $B'=(W , J, c, E)$  are equivalent basic objects, then
$\mu _x (I) /b = \mu _x (J) /c$, for all $x \in \sg(B) = \sg (B')$. 

This result has the following corollary. Assume $B$ and $B'$ are equivalent basic objects, let 
 $$(1) \qquad B=B_0 \leftarrow \cdots  \leftarrow B_s ~,$$
   $$(2) \qquad B'=B'_0 \leftarrow \cdots  \leftarrow B'_s$$
   be $t$-sequences, where in each case we have used the same permissible centers. Notice that, writing 
   $B_i=(W_i,I_i,b,E_i), \, B'_i=(W_i,J_i,c,E_i)$, $E_i=(H_1,\ldots,H_{m},H_{m+1}, \ldots, H_{m+i})$, where $H_1, \ldots, H_m$ are the strict transforms of the hypersurfaces in $E_0$ (and similarly for $B'_i$), we have equalities 
   $$(3) \quad I_i={\bar I}_i {I(H_{m+1})^{a_1}}  \ldots     I(H_{m+i})^{a_i}      $$
   $$(4) \quad J_i={\bar J}_i {I(H_{m+1})^{a'_1}}  \ldots     I(H_{m+i})^{a'_i}      $$
   where the exponents are locally constant. We let $t_1, \ldots  t_s$ (resp.  
   $t'_1, \ldots  t'_s$) be the $t$-functions of $B$ (resp. of $B'$). Then, in the notation just introduced,:
   
   {\noindent {\it (a) $t_i=t'_i$, $i=1, \ldots, s$}},
   
     {\noindent {\it (b) $a_i/b=a'_i/c$, $i=1, \ldots, s$}}.
\end{voi}
  
  An important construction  (introduced in \cite{W})  that leads to equivalent objects  is that of the homogenized ideal, which we review next.
 
\begin{voi}
\label{V:a10.0}
{\it Homogeneized ideals}. 
If $W$ is a scheme, a $W$-{\it weighted ideal} is a pair $(I,b)$, where $I \subset {\cO}_W$ is $W$-ideal  and $b$ is a non-negative integer. 

Let $(I,b)$ be a weighted $W$-ideal. Its associated homogenized ideal is the the $W$-ideal 
$$(1) \quad H (I,b)= I + \Delta (I) T(I) + \cdots +  \Delta ^i(I) T(I) ^i + \cdots + \Delta ^{b-1}(I) T(I)^{b-1}$$ 
where we have written $T(I):= {\Delta}^{b-1}(I)$. 
If $(I,b)$ is a weighted ideal on a smooth scheme $W$, then 
${\Delta}^{b-1}(I) = {\Delta}^{b-1}({H}(I,b))$ (see \cite{W}, 2.9). If $b$ is clear from the context we simply write ${H}(I)$.

If $B=(W , I,b, E)$ is a  basic object in $\cS$, the basic object  ${H}(B):=(W, H (I,b), b,E)$ is the {\it homogenized basic object associated to $B$}. 

If 
 $B=(W , I,b, E)$ is a basic object, then $H (B)$ is equivalent to $B$.  The verification in case the arrow in (1) (or (2)) of \ref{V:a8} is a permissible transformation is done in  
  \cite{W}, 2.9.2.  Concerning possible extensions  use the (easily verified) fact that, in the notation of \ref{V:a8}, 
 $H(I{\cO}_{W[1]},b)=H(I,b){\cO}_{W[1]}$, whence $(H(B))(e)=H(B(e))$. The verification for open restrictions is immediate.
 \end{voi}

\begin{voi}
\label{V:a11.0}
Now we are in position to describe our main resolution algorithm. For the time being, it will be called the W-algorithm. This is discussed in 
\cite{NE}, \cite{BEV} or \cite{EV}, we include it for completeness and to fix the notation. 

For each integer $d \geq 1$ we must indicate a totally ordered set $\Lambda^{(d)}$  and, for any given basic object $B_{0}=(W_{0},I_{0},b_{0},E_{0})$ over $k$  of dimension $d$,  the corresponding resolution functions $g_j$.
 
This process will be defined inductively on the dimension of $B_0$, as follows. In the sequel, 
$\cS _1:= {\bf Q}\times {\bf Z}$ and $\cS _2:= {\bf Z}\times {\bf Q}\times {\bf Z}^{\bf N}$, in all cases lexicographically ordered.

($\alpha$) If ${\rm dim} (B_0) = 1$, let $\Lambda ^{(1)}= \{0\}\cup \cS _1 \cup \cS _2 \cup  \{\infty \}$, where if $a \in \cS _2 $ and $b \in \cS _1 $ then $a >b$, $0$ is the smallest element of the set  and $\infty $ is the largest one. Then we define $g_0(x)=0$ if $x \notin \sg(B_0)$ and  for 
$w \in \sg (B_0)$, $g_0(w)= t_0(w)$. If $g_i$ is defined for $i < s$, determining a permissible sequence 
$B_0\leftarrow B_1 \leftarrow \cdots \leftarrow B_s$ we define, for $w \in 
\sg(B_s)$, $g_s(w)=t_s(w)$ if $\o _s(w) > 0$ and $g_s(w)=\Gamma _s (w)$ if $\o _s(w)  0$; while $g_s(x)=0$ if $x \notin \sg(B_s)$.

In the induction step we need the following  auxiliary construction.

\smallskip
 
($\beta$) {\noindent {\it  Inductive step}}. Assume that we have an algorithm of resolution  defined for basic objects of dimension $< d$. 
Consider a $t$-permissible sequence of basic objects and transformations 
$$ (1) \quad B_0 \leftarrow B_1 \leftarrow \cdots \leftarrow B_s $$ 
Let $w \in \ma (t_s)$ and suppose that, near $w$,  $\dim \ma(t_s) \le d-2$. 
   Then,  there is an open neighborhood $U$ of $w$ (in $W_s=us(B_s)$), a hypersurface $Z_s$ on $U$, containing $w$, and a basic object 
${B_s}^*=(Z_s, {I_s}^*, {b_s}^*,{E_s}^*)$, having the following properties:

\smallskip
 
 (i) $\sg ({B_s}^*) = {\rm Max} \, ({t_s} _{|U})$. 
 
 \smallskip

(ii) The algorithmic resolution sequence corresponding (by the induction hypothesis) to ${B_s}^*$:
 $$ (2) \quad {B_s}^* \leftarrow ({B_s}^*)_{1} \leftarrow \cdots \leftarrow 
 ({B_s}^*)_{p}$$ 
(determined, say, by resolution functions $\widetilde{g}_i$) 
 induces a $t$-permissible sequence 
 $$ (3) \quad  {\widetilde B_s} \leftarrow  {\widetilde B_{s+1}} \leftarrow 
 {\widetilde B_{s+p}}             $$ 
 (obtained by using the same centers $C_i=\ma (\widetilde{g}_i)$, and denoting by  ${\widetilde B_s}$ the restriction of $B_s$ to $U$).
 
 \smallskip
 
 (iii) If 
 ${\rm max} (t _s)={\rm max} (t _{s+j})$ ($j=1, \ldots,  p$) then,  
  for all such indices $j$,  
 $us(({{B_{s}}^*})_j)$ gets identified to $Z_{s+j}$, the strict transform of $Z_s$ to $us({\widetilde B}_{s+j})$ and 
$\sg(({{B_{s}}^*})_j)= {\rm Max} (\widetilde{t_{r+j}})$ (where $\widetilde{t_j}$ are the $t$-functions of the sequence (3)). 

\smallskip

(iv) Under the assumption of (iii) for all $j=0, \ldots , p$, if 
$w_j \in {\rm Max}(\widetilde{t_{s+j}})$
 is in the pre-image of $w$ (under the morphism $us (\widetilde B_{s+j}) \to us (\widetilde B_{s})$ arising from (3)), the resolution function $\widetilde{g_j}$ of ${B_s}^*$ defines a function (still denoted by $\widetilde{g_j}$) on a neighborhood (in ${\rm Max}(\widetilde{t_{s+j}})$) of $w_j$.  Neither the neighborhood $U$ nor the hypersurface $Z_s$ are uniquely determined by the  process, but the value $\widetilde{g_j}(w_j)$ is independent of the choices made. 

 In \ref{V:a12.0} we shall explain how to make these constructions.

($\gamma$) Now, assuming the resolution functions given for dimension $< d$,  we'll define resolution functions $g_j$ for objects of dimension $d$ as follows. In this case,  
the totally ordered set of values will be: 
 $\Lambda ^{(d)}=\{0_d\} \cup (\cS _1 \times {\Lambda}^{(d-1)})\cup \cS _2 \cup \{ \infty  _d\}$; where 
 $\cS _1 \times {\Lambda}^{(d-1)}$ is lexicographically ordered, any element of $\cS _2$ is larger than any element of $\cS _1 \times {\Lambda}^{(d-1)}$, 
 $\infty _d$ is the largest element and $0_d$ the smallest one. In all cases we'll write $g_i(x)=0_d$ if $x \notin \sg(B_i)$.  For other values, consider first a single basic object $B_0$. Given $x \in \sg(B_0)$, let $M(1)$ denote the union of the one-codimensional components of $\ma (t_0):=M$. Necessarily we have $\o _0(x) > 0$ and there 
  are three cases.  
(a) $x \in M(1)$. Then, set $g_0(x)=(t_0(x), {\infty}_{d-1})$. 
(b) $x \in M \setminus M(1)$. Take a neighborhood $U$ of $x$ (in $W$) such that the basic object $B_0^*$ above  is defined 
 and the function 
 $\widetilde {g_0}:Z_0 \to {\Lambda}^{(d-1)}$ as in (iv) above (with $s = j=0$) corresponding to $B_0^*$. 
 Then set   $g_0(x)=(t_0(x),\widetilde {g_0}(x))$. This value is independent of the choices made. (c) $x \notin M$
 Then set $g_0(x)=(t_0(x),\infty _{d-1})$.

Assume now  that resolutions functions $g_i$, $i=0, \ldots, j-1$ have been defined, determining  centers $C_i = \ma \,{(g_i)}, i=0, \ldots, j-1$, leading to a permissible sequence $B_0 \leftarrow \cdots \leftarrow B_j$,  $
 B_i=(W_i,I_i,b,E_i)$, $i=0, \ldots, j$, $j > 0$. We assume that if $B_{j-1}$ is not a monomial object, then this is a $t$-sequence.
  
  There are two basic cases: (a) $\max (\o _j) =0$, (b) $\max (\o _j) > 0$. 
  
   In case (a), $B_j$ is monomial. For $x \in \sg(B_j)$ let $\Gamma _j$ be its $\Gamma$-function and set $g_j(x):=\Gamma _j (x)$. In case (b), let $M_1(j)$ denote the union of the one-codimensional components of $M(j):= \ma (t_j)$ and $H$ the exceptional divisor of the blowing-up (with center $C_{i-1}$) $W_{j-1} \leftarrow W_j$.  For $x \in \sg (B_j)$ there are three  sub-cases: 
   
   ($b_1$) $x \in M_1(t_j) \cap H$. Then we set  $g_j(x)=(t_j(x), {\infty}_{d-1})$                                                         
     
   ($b_2$) $x \in (M(j) \setminus M_1(j))\cap H$ (the {\it inductive} situation). Consider the smallest index $s$ such that $t_s(x_s)=t_j(x)$, where $x_s$ is the image of $x$ in $\sg (B_s)$ induced by the sequence above. Using 
  the construction of ($\beta$), applied to $x_s \in W_s$, we obtain resolution functions of $B^*_s$, $\widetilde{g_{0}}, \widetilde{g_{1}}, \dots$. So, it makes sense to take $\widetilde{g_{j-s}}(x)$, and it can be proved that this value is well-defined. We set 
 $g_j(x)=(t_j(x),\widetilde{g_{j-s}}(x)) \in \cS _1 \times {\Lambda}^{(d-1)}$. 
 
   ($b_3$)  $ x \notin H$. Then, if $x'$ is the image of $x$ in $W_{j-1}$, set $g_j(x)=g_{j-1}(x')$

 With this definition, if $\max (g_j) >0$ (and hence 
 $B_j$ is not monomial)  then the center $C_j={\mathrm {Max}}(g_j)$ is contained in ${\mathrm {Max}}(t_j)$.

 It can be proved that the sequence $\{  {\mathrm {max}}(g_j)       \}$ is strictly decreasing, which  leads to a resolution of $B$ (\cite{BEV} or \cite{EV}).

 \end{voi} 
\begin{voi}
\label{V:a12.0}
We shall explain better some details of this process, specially the crucial inductive step ($\beta$) of \ref{V:a11.0}. For this , we must review some other concepts. We follow the terminology of \cite{NE}.

\smallskip

($\alpha$) {\it Adapted hypersurfaces, nice objects}. Let   
$B=(W, I,b,E)$ a $\cS$-basic object, 
 ($E=(H_1, \ldots, H_m)$ ).  We say that  a hypersurface $Z \subset W$ is transversal to one of the divisors $H_i$ at $w \in H_i \cap Z$ if there is a regular system of parameters 
$a_1, \ldots, a_n$ 
in $\cO _{W,w}$ such that $I(Z)_{w}=(a_1)$ and $I(H_i)_w = (a_2)$; $Z$ is transversal to $H_i$ if it is so at each common point. Finally, $Z$ is transversal to $E$ if it is transversal to each hypersurface in $E$.

A hypersurface $Z \subset W$ 
is {\it adapted} to $B$ (or  $Z$ is {\it $B$-adapted}) if
 the following conditions hold:
\begin{itemize}
\item[(A1)] $I(Z) \subseteq \Delta ^{b-1} (I)$ (an inclusion of sheaves of
${\cO}_W$-ideals),
\item[(A2)] $Z$ is transversal to $E$ (in particular, $Z$ is regular). 
\end{itemize}
If, moreover, it satisfies:
\begin{itemize}
\item[(A3)] Whenever $D$ (resp. $D'$) is an irreducible component of $Z$ (resp. of
$V({\Delta}^{n-1}I)$) then $D \not= D'$ 
\end{itemize}
we say that $Z$ is {\it $B$-inductive}.

\smallskip
We shall say that $B$ (a basic object in $\cS$) is {\it nice} if either
Sing($B$) is empty or $B$ admits an adapted hypersurface. 

Notice that if $B$ (as above) is a nice basic object, then for all $x \in \sg(B)$ we have $\nu _x (I)=b$. A basic object having this property is called $good$ (terminology from \cite{EV}).

\smallskip

($\beta$) {\it Inductive objects}.  
If  $B$ is a nice basic object 
  we define a $W$-ideal, called the {\it coefficient ideal} and denoted by
${\cC}(I)$, as follows:
 $${\cC}(I) : = \sum_{i=0}^{b-1} \, [{\Delta}^i(I)]^{b!/b-i}  $$
 
 If $Z$ is a $B$-inductive hypersurface, then 
 the {\it coefficient ideal relative to $Z$}, or the 
$Z$-{\it coefficient ideal}, denoted by
${\cC}(I,Z)$, is the restriction of  
${\cC}(I)$ to $Z$. This is a never-zero $Z$-ideal.\\
The basic object 
$B_Z := (Z , {\cC}(I,Z), b!, E_Z)$, where  $E_Z:= ( {H_1 \cap Z},
\ldots, {H_m \cap Z })$, is called the {\it inductive object of $B$, relative to the inductive hypersurface $Z$}.

A fundamental property of the inductive object is: $\sg (B) = \sg (B_Z)$, and similarly when we consider  sequences of permissible transformations  based on these objects, using the same centers.

\smallskip

$(\gamma)$ Consider a $t$-permissible sequence of basic objects 
$B_0 \leftarrow \cdots \leftarrow B_r$ (where $(B_j=(W_j,I_j, b,E_j))$ and a point $w \in \sg (B_r)$. Then,  there is a nice object 
$B_r  ''=(U,I_r '',b_r '',E_r '')$
  ($U$  a suitable neighborhood of $w$ in $W_r$) admitting an adapted hypersurface $Z$. To construct it, introduce first a $W_r$-ideal $J_r$ as follows. 
  Let max($t_r$)=($b_r/b, \bar n$), 
$\bbar{I_r}$ the proper transform of $I_0$ to $W_r$.  If  
$b_r \geq b$  let $J_r={\bbar I_r}$.
 If  $b_r < b$, write  
$E_r=(H_1, \ldots, H_m, \ldots, H_{m+r}$) (where $H_1, \ldots, H_m$ are the strict transforms of the hypersurfaces that appear in $E_0$). Then there is an expression 
$I_r=  I(H_{m+1})^{a_1}\ldots I(H_{m+r})^{a_r} {\bar{I_r}} $. Set 
$\scr{C}_{r}=I(H_{m+1})^{a_1}\ldots I(H_{m+r})^{a_r}$  and  
$J_r = {\bbar{I_r}}^{b-b_r}+\scr{C}\,^{b_r}  $.

Now, returning to $B_r ''$, $U$ is a neighborhood of $w$  such  
$\mathrm{Max} ({t_r}_{|U})  = \mathrm{Max} ({{\o}_r}_{|U}) \cap H^*_1 \cap \cdots \cap  H^*_{\bbar{n}} \cap U$, $E_r ''=E_r^{+}= : (E_r \setminus E_r ^{-})_{|U}$, $I_r''=(J_r + I(H^*_1)^{b'} + \cdots + I(H^*_{\bbar n})^{b'})_{|U}$ 
 and $b_r ''=b_r$ if $b_r \ge b$ while 
$b_r ''=b_r(b-b_r)$ if $b_r < b$.

Of course, more precisely we should  rather write $({I_r}_{|U})''$ and $({B_r}_{|U})''$.

Some properties of this object $B_r''$ are (assuming $U=W_r$ to simplify the notation): (i) ${\mathrm{Sing}}(B_r'')={\mathrm{Max}}(t_r)$, (ii) If $C \subset {\mathrm{Max}}(t_r)$ is a center , $(B_r'')_1= \cT(B_r '',C)$,     $B_{r+1}=\cT(B_r,C)$, and ${\mathrm{max}}(t_r) = {\mathrm{max}}(t_{r+1})$, then 
$(B_r'')_1=(B_{r+1})''$.
\end{voi}

\begin{voi}
\label{V:a13.0} 
We complete the discussion of the W-resolution algorithm of \ref{V:a11.0}. In the notation of \ref{V:a11.0}, we take  as the open set $U$ a neighborhood of $w$ over which the  nice object $B_s''$ of \ref{V:a12.0} ($\gamma$) is defined, hence its (again nice) associated homogenized object   $H ({B_s}'')$
 (see \ref{V:a11.0}) admits an inductive hypersurface $Z_s$ containing $w$, defined on $U$. This will be the $Z_s$ of 
 \ref{V:a11.0}.  Our object $B^*_s$ of \ref{V:a11.0} will be  $({H}(B''_s))_{Z_s}$ . 
 
 In \cite{NE}, 8.5, it is verified that properties (i)-(iv) of \ref{V:a11.0} are valid.
\end{voi}
 \begin{voi}
 \label{V:a14.0} {\it The V-algorithm} The algorithm discussed in \cite{EV} or \cite{BEV}, which will be referred to as the V-algorithm, is very similar to the W-algorithm. One proceeds as in \ref{V:a11.0}, the only difference is that in the inductive step the auxiliary object $B^*_s$ is 
 $(B_s'')_{Z_s}$ rather than $({H}(B''_s))_{Z_s}$ (notation of \ref{V:a13.0}). This looks simpler, however with this approach it is more difficult to check that the process is independent of the choice of the adapted hypersurfaces $Z_s$ we choose. In the mentioned references, this is done by re-developing the theory in the broader context of {\it generalized basic objects}. A key role is played by Hironaka's trick. Using instead the W-algorithm, one may work entirely within the class of basic objects, the key element being a glueing lemma discovered by Wlodarczyk (\cite{W}), which involves suitable etale neighborhoods. This approach has the additional advantage that it can be used to generalize, to same extent, the theory to the situation where we work over an artinian ring rather than a field. See 
 \cite{NE} or, for a review, Section \ref{S:CE}.
 
 It seems that the V- and W- algorithms are ``essentially'' the same. But we may be more precise: they are exactly the same, in the sense that the resolution functions in either case coincide. We shall check this fact, but we need the following remarks.
\end{voi}

\begin{voi}
\label{V:mucheq} 
Here we discuss two  instances of equivalent basic objects (\ref{V:a8}). We shall freely quote results from \cite{BEV},    \cite{EV} and  \cite{NE}, and omit some verifications which are rather tedious but straightforward consequences of the definitions.

\smallskip
\noindent $(a)$ Let $B=(W,I,b,E)$ and ${\bbar B}=(W,J,b,E)$ be {\it equivalent} nice  basic objects, admitting  a common adapted hypersurface $Z$. Then, $B_Z$ and ${\bbar B}_{Z}$ are equivalent. 

Indeed, $\sg (B_Z)= \sg (B) = \sg (\bbar B) = \sg ({\bbar B}_Z)$. Now consider transformations ${B \leftarrow B_1}$, 
${\bbar B} \leftarrow {\bbar B}_1$, both with a common permissible center $C$ and the corresponding transformations (with center $C$) 
$B_Z \leftarrow (B_Z)_1$ and  ${\bbar B}_Z \leftarrow ({\bbar B}_Z)_1$. Let $Z_1$ be the strict transform of $Z$ to $W_1:= us(B_1)=us({\bbar B}_1)$. Then 
$\sg ((B_Z)_1)=\sg ((B_1)_{Z_1}  )= \sg (B_1) = \sg ({\bbar B}_1) = \sg (({\bbar B})_{Z_1}) = \sg ( ({\bbar B}_Z )_1 )$.  Take, instead,  
extensions. Let $Z(e):= {Z \times {\bf A}^{1}_{k}}$. Then there are identifications 
$ (B(e))_{Z(e)} = (B_Z)(e)$ and $({\bbar B} (e) _{Z(e)}=({\bbar B}_Z)(e) $. Since, $\sg(B(e) = \sg ({\bbar B}(e)$ (by equivalence), as above we get an equality $\sg((B_Z)(e))    = \sg (({\bbar B})_Z)(e)$.  The verification for open restrictions is easy. Repeating, we see that $B_Z$ and ${\bbar B}_Z$ are indeed equivalent.

\smallskip

$(b)$ If $B$, ${\bbar B}$ are basic objects such that their associated objects $B''$ and 
$(\bbar B)''$ of \ref{V:a12.0} ($\gamma$) are globally defined, then $B''$ and $(\bbar B)''$ are equivalent.

 The proof is similar to that of part (a), using the facts that $\sg (B'')= \ma (t)$,  the $t$-function is invariant under equivalence (\ref{V:a9}) and the following observation.
 
 If $B=B_0 \leftarrow \cdots \leftarrow B_s$ is a $t$-permissible sequence (with centers $C_i$, $B_i=(W_i,I_i,b,E_i)$), then it induces a $t$-permissible sequence of extensions $B(e)=B_0 (e)\leftarrow \cdots \leftarrow B_s (e)$ (with centers $C'_i = q_i ^{-1}(C_i)$,  $q_i:us(B_i(e))=W_i \times {\bf A}^{1} \to W_i $ being the projection). Assuming, to simplify the notation, that $B''$ globally defined, we have $(B_s(e))''=(B_s '')(e)$.
\end{voi}
 
 In the next statement, when we say ``the V-resolution functions of a basic object'' we mean the resolution functions we get when we apply to it the V-resolution algorithm, similarly for the W-resolution functions.

\begin{pro}
\label{P:igua}
Let $B=(W,I,b,E)$  be a $\cS$-basic object, 
$\{g_j  \}, ~ j =1, \ldots  ,r$ and  
$\{G_j  \}, ~ j =1, \ldots  ,r'$  be the V-  and W-resolution functions of $B$ respectively. Then $r=r'$ and 
$g_j = G_j$ for all $j$.

\end{pro}

\begin{proof} In view of \ref{V:a9}, the only possible difference in the definition of the $V$- and $W$-resolution
sequences is the the ``inductive situation'', i.e., in the notation of \ref{V:a11.0} ($\gamma$), case ($b_2$). We prove this case, proceeding by induction on the dimension $d$ of $B$. If 
$d=\dim (B) =1$ situation ($b_2$) is not present, so the functions $g_i$ and $G_i$ agree in
this one-dimensional case.

For the inductive step, assume by induction that $g_i=G_i$ for $i < j $ (in particular, for $j=0$
there is no hypothesis). Thus, using the algorithmic centers 
$C_i =\ma (g_j)=\ma (G_j)$,
$i=0, \ldots, {j-1}$,
in both cases we get the same partial resolution $B=B_0 \leftarrow \cdots 
\leftarrow B_j$.
We want to prove: $g_j=G_j$. As mentioned, the only relevant case is
situation ($b_2$) of \ref{V:a11.0} ($\gamma$) (the inductive situation). Let $x \in \ma (t_j)$ (which has
codimension $>1$). Note that the index $s$ involved is the same, whether
we are defining $g_j$ or $G_j$. Let $x_s$ be the image of $x$ in $W_s =
us(B_s)$. Take an open neighborhood $U$ of $x_s$ where the nice basic
object
$(B_{|U})''$ is defined, admitting an inductive hypersurface $Z_s$. In the
definition of $g_j(x)$ (resp. $G_j(x)$) we use the inductive object
$  {({B_s}_{|U})''}_{Z_s}   $  
(resp.
$ {  {H ( {(B_s ''})}_{|U}}$).  Let
 $\widetilde {g_i}$ and $\widetilde {g'_j}$ be the V-resolution functions of
$    {({B_s}_{|U})''}_{Z_s}    $  and
$ ( {H }(  {({B_s}_{|U})''} )_{Z_s}$ respectively, and $\widetilde {G_i}$ the
W-resolution functions of
$  {({B_s}_{|U})''}$, $i=s, \ldots$. By \ref{V:mucheq} ,
 $ {({B_s}_{|U})''}_{Z_s}   $ and $          ({  {H }( ({B_s}_{|U})'')) }_{Z_s} $ are
equivalent. Now, a property of the V-resolution algorithm is that equivalent
basic objects have the same $V$-resolution functions (\cite{BEV}, 12.5),  so
$\widetilde{g_i}=\widetilde{g' _i}$, for all $i$.   By induction on the
dimension, 
$\widetilde{g_i}=\widetilde{G_i}$,  $i=s, \ldots$. So,
$g_j(x)=(t_j(x), \widetilde {g_j} (x)) =
(t_j(x), \widetilde {G_j} (x))=G_j(x)$, showing that $g_j=G_j$.

Now it is clear that  $r=r'$.

\end{proof}

Throughout the remainder of this article, when we mention {\it the resolution  algorithm} we mean either the V-algorithm or the W-algorithm. By \ref{P:igua} they coincide.

\section{Families of basic objects}
\label{S:FB}

In this section, all the considered schemes are in the class $\cS$ (see \ref{V:a2.0}). 

\begin{voi}
\label{V:ne0.1}
Let $\pi:W\to T$ be a morphism of schemes. If $\pi$ is clear from the context, we shall write  $W^{(t)}$ to denote the fiber 
$\pi ^{-1}(t)$ of $\pi$ at $t \in T$, i.e., the $k(t)$-scheme  obtained from $\pi$ by the base change 
$\Spec (k(t)) \to T$ (the unique morphism whose image is $t$). Similar notation is used for geometric fibers. If $t \in T$, $\bbar t$ denotes the geometric point determined by the algebraic closure of the residue field $k(t)$. 

We have a canonical morphism 
$j_t:W^{(t)} \to W$, inducing a homeomorphism of topological spaces $W^{(t)} \thickapprox {\pi}^{-1}(t) \subset W$. If $t$ is closed, $j_t$ is a closed embedding.  Often, if $Z$ is a closed subscheme of $W$, we want to view $Z \cap W^{(t)}$ as a scheme. This is done by 
{\it defining} $Z \cap W^{(t)}$ as ${j_t}^{-1}(Z)$ (a closed subscheme of $W^{(t)}$). If $t$ is closed, via the closed 
embedding $j_t$ this becomes the usual scheme-theoretic intersection of $W^{(t)}$ and $Z$ (as closed subschemes of $W$). Note that always the 
 fiber $W^{(t)}$ may be identified to a certain {\it closed } fiber. Namely, let 
 $T \{ t \} := \Spec ({\cO}_{T,t}), ~  W\{ t  \}= W  \times _{T} T\{ t \}$ (where 
$T \{ t \} \to T$ is the natural morphism sending the closed point $t'$ of $T \{ t \}$ to $t$) and $f_t:W \{ t \} \to T \{ t \}$ the second projection. Then, $W^{(t)}$ can be identified to ${f_t}^{-1}(t')$ (the only closed fiber of $f_t$). We also have a canonical morphism 
$j \{ t \}: W \{ t \} \to W$ (the first projection). Note that there is a canonical isomorphism between 
$Z \cap W^{(t)} := j_t^{-1(Z)}$ and $j \{ t \}^{-1} (Z) \cap f _t ^{-1} (t') $ (a scheme-theoretic intersection of closed subschemes of $W \{ t \}$). So, using these identifications, in many arguments there is no loss of generality in assuming that in an intersection $Z \cap W^{(t)}$ as above the point $t \in T$ is closed.
\end{voi}
 
 \begin{voi}
 \label{V:ne0}
  Let $\pi:W\to T$ be a smooth morphism of schemes and  $w$ a point of $W$, $t=\pi(w)$,  $R={\cO}_{W,w}$,
$R'={\cO}_{W^{(t)},w}$ (the local ring of the  fiber $W^{(t)}$ at $w$, which is regular). 

(a) A system of elements $a_1, \ldots, a_n$ in $R$ is called a {\it regular system of parameters of $R$, relative to $\pi$} (or a {\it $T$-regular system of parameters}, or an {\it $A$-regular system of parameters}, if $T=\Spec (A)$)  if the induced elements  $a_1^{(t)}, \ldots,  a_n^{(t)}$ in $R'$ form a regular system of parameters, in the usual sense, of the regular local ring ${\cO}_{W^{(t)},w}$. Elements $a_1, \ldots, a_r$ of $R$ are
{\it part of a $T$-regular system of parameters}, or 
{\it a partial  $T$-regular system of parameters}, if they are contained in an  $T$-regular system of parameters $a_1, \ldots, a_n$, $r \leq n$, of $R$. Then necessarily 
$a_1, \ldots, a_r$ is a regular sequence in the local ring ${\cO}_{W,w}$ (this follows from \cite{NE}, 10.2).

(b) A $T$-hypersurface of $W$ is an effective Cartier divisor $H$ of $W$ such that,  for all $w \in H$, the ideal $I(H)_w$ is generated by an element $a$ of ${\cO}_{W,w}$ where $\{  a  \}$ is  a partial $T$-regular system of parameters. Such a scheme $H$ is smooth over $T$ (see 
\cite{NE}, 11.2).

(c) Let $E=(H_1, \ldots, H_m)$ be a finite sequence of relative effective Cartier divisors of $W$ ({\it relative}  means that the divisor is flat over $T$). We say that $E$ has (or $H_1,\ldots,H_m$ have) {\it normal crossings, relative to $\pi$} (or to $T$, if $\pi$ is clear) if,   
 for any  point $w \in H_1 \cup \cdots \cup H_m$, there 
is a $T$-regular system of parameters  
$a_1, \ldots, a_n$ of ${\cO}_{W,w}$  such that, 
  for each
$H_i$ containing $w$, the ideal  $I(H_i)_w$ is generated by a suitable element $a_{j}$, $j \in \{1, \ldots, n \}$. In particular, each $H_i$ is a $T$-hypersurface.

(d) A closed subscheme $C$ of $W$ is said to have normal crossings with $E$ (relative to $T$)  if 
 for any  point $w \in C$, there 
is a $T$-regular system of parameters  
$a_1, \ldots, a_n$ of ${\cO}_{W,w}$ (see \ref{V:ne2}) such that 
 the stalk $I(C)_w$ is generated by ($a_1, \ldots, a_r) {\cO}_{W,w}$ (for some $r \leq n$) and for each
$H_i$ containing $w$, the ideal  $I(H_i)_w$ is generated by a suitable element $a_{j}$, $j \in \{1, \ldots, n \}$. In particular $C$ is locally defined by a regular sequence (see \cite {NE}, Lemma 11.2).  The possibility that $E = \emptyset$   is not excluded.  

It can be proved  that the induced projection $C \to T$ is smooth, and  the
blowing-up $W_1$ of $W$ with center $C$ is also $T$-smooth (this follows from  \cite{NE}, 11.2).   

\end{voi}

Now we recall a basic result on blowing-ups which involves regular sequences. A proof can be seen in \cite{NE}, Proposition 11.7.
 \begin{pro}
 \label{P:bu}
 Let $f:X \to T$ be a morphism of schemes, $C \subset X$ a closed subscheme, flat over $T$, $J=I(C)$, such that $J_x \subset {\cO}_{X,x}$ is generated by a regular sequence , for all $x \in C$. Let $T' \to T$ be a morphism, 
 $X'=X \times _T T'$, $p:X' \to X$  the natural projection, $C'=C \times _T T'=p^{-1}(C)$, $X_1 \to X$ and $X_1' \to X_1$ the blowing-ups with centers $C$ and $C'$ respectively. Then, there is a natural isomorphism $X_1' = X_1 \times _X X'$ 
 \end{pro}

\begin{Def}
\label{D:fa}
A {\it family of basic objects} in $\cS$ (or an $\cS$-family of basic objects) parametrized by a scheme $T$  is a four-tuple  
$$(1) \quad \cB = (\pi : W \to T,I,b,E)$$  
where $\pi$ is a smooth, surjective  morphism of schemes, $I$ is a $W$-ideal, $b$ is an integer and $E=(H_1, \ldots, H_m)$ is a finite sequence of relative effective Cartier divisors of $W$ having normal crossings  (relative to $T$) such that for all $t \in T$ the fiber $I^{(t)}:=I{\cO}_{W^{(t)}}$ is a never-zero $W^{(t)}$ ideal and $E^{(t)}=(H_1^{(t)}, \ldots, H_m^{(t)})$ is a sequence of distinct regular  hypersurfaces of  $W^{(t)}$ (necessarily having  normal crossings). 

For all $t \in T$ we may define the fiber at $t$: 
${[\cB]}^{(t)}:=(W^{(t)},I^{(t)},b, E^{(t)})$ (as well as the geometric fiber $[\cB]^{(\bbar t)}$), which is a basic object in $\cS$.  Often    we shall   write $B^{(t)}$ to          denote the fiber ${[\cB]}^{(t)}$. 

A family of basic objects parametrized by $T$ will be also called 
a {\it basic object over $T$}, or a $T$-{\it basic object} or, if $T = \Spec A$ ($A$ a ring) an $A$-{\it basic object}.

If $\cB$ is as above, $W$ is the {\it underlying scheme } of the family $\cB$, we denote it by $us(\cB)$. The dimension of $\cB$, dim ($\cB$), is the dimension of the scheme $us(\cB)$.
\end{Def}

\begin{voi}
\label{V:ne1} In case $T$ is regular (and hence also $W$ is so, because $\pi$ is smooth) we may rephrase Definition \ref{D:fa} as follows.  
 A family of basic objects in $\cS$  parametrized by a regular scheme $T$  is a four-tuple $\cB = ({\pi : W \to T},I,b,E)$  
where $\pi$ is a surjective smooth morphism of schemes, $I$ is a never-zero $W$-ideal, $b$ is an integer and $E=(H_1, \ldots, H_m)$ is a finite sequence of hypersurfaces (of the regular scheme $W$, see \cite{BEV}, Definition 2.1)  having normal crossings  such that for all $t \in T$ the fiber $I^{(t)}:=I{\cO}_{W^{(t)}}$ is a never-zero $W^{(t)}$ ideal, and $E^{(t)}=(H_1^{(t)}, \ldots, H_m^{(t)})$ is a sequence of distinct regular  hypersurfaces of  $W^{(t)}$ with normal crossings. 

In this case, the four-tuple $B=(W ,I,b,E)$ is a basic object (in the sense of \ref{V:a3.0}), called the {\it basic object associated to $\cB$} and denoted by $aso(\cB)$.

\end{voi}

\begin{voi}
 \label{V:loc}
{\it Localization}. Let $\cB$ be a family (as in \ref{D:fa}), $t \in T$, $R= {\cO}_{T,t}$, $S=\Spec (R) $ and $f:S \to T$ the natural morphism sending the closed point $s$ of $S$ into $t$. By pull-back, $\cB$ induces an $S$-basic object $\cB _S$, called the {\it localization} of $\cB$ at $t$. Note that we may identify ${[\cB]}^{(t)}={[\cB _S]}^{(s)}$. This allows us to assume in many discussions  on fibers, without loss of generality, that the fiber is taken at a closed point of $T$.   
\end{voi}

\begin{voi}
 \label{V:ne1.1} 
(a) With the notation of \ref{D:fa}, 
  if $C \subset W$ is the irreducible closed subscheme of $W$ defined by the $W$-ideal
$J \subset {\cO}_W$, we shall say that the order of $I$ along
$C$  is $\geq m$,
written
$\nu (I,C) \geq m$, if $I \subseteq J^m$. 
 If $m$ is the largest interger that works, we write $\nu (I,C) = m$. For $C$ reducible,  
$\nu(I,C):=\max \{\nu (I,D): D$ is an irreducible component of $C\}$. 
\smallskip

(b) A {\it permissible center} of the $T$-basic object (1) (or a $\cB$-center, or just a $T$-center, if $\cB$ is clear) is a closed subscheme $C$ of $W$ which  has normal crossings with $E$ (relative to $T$) and such that for every  irreducible component $D$ of $C$,
$\nu(I,D) = \nu(I^{(t)},D^{(t)}) \ge b$ for all $t \in T $. 

\smallskip
For $C$ irreducible as above, the inequality $\nu(I^{(t)},C^{(t)}) \ge \nu(I,C) $ is always valid. If $C$ satisfies  simply the equality 
$\nu(I,C) \ge b$, we might call call $C$ a {\it weakly permissible center}. We shall not use this concept in this article. 
\end{voi}

\begin{voi}
\label{V:ne6}

If $\cB=(\pi:W \to T,I,b,E)$ is an $T$-basic object,  $C$ a $T$-permissible center for $\cB$ and $q: W_1 \to W$ is the blowing-up of $W$ with center $C$, then (generalizing \ref{V:a5.0}) we may consider several $W_1$-sheaves induced by ${I}$, namely: (i) $I'_1={I}{\cO}_{W_1}$ (the {\it total transform} of $I$ to $W_1$), (ii) $I_1:= {\mathcal E}^{-b} I'_1$, where ${\mathcal E}$ defines the exceptional divisor of $\pi$ (the {\it controlled transform} of $I$ to $W_1$), 
 (iii) $\overline {{I}_1} := {\mathcal E}^{-a}{I'}_1$, with $a$ as large as possible (the {\it proper transform} of ${I}$.) If $C$ is not connected, the exponent $a$ is constant along $p^{-1}(C')$, for each connected component $C'$ of $C$, but not necessarily globally constant.
 
 The four-tuple  
${\cB}_1 := ({\pi}_1:W_1 \to S,{I}_1,b,E_1)$, where ${\pi}_1=\pi q$, $I_1$ is the controlled transform of $I$ and 
$E_1=(H'_1, \ldots,H'_{m},H'_{m+1})$, with $H'_{m+1}$ the exceptional divisor and $H'_{i}$ the strict transform of $H_i$ (defined by 
$(I(H'_{m+1}))^{-1}I(H_i){\cO}_{W_1}$), $i=1, \ldots,m$,  is a new $T$-basic object,  called the {\it transform} of the $T$-basic object ${\cB}$ with center $C$. The process of replacing a basic object $\cB$ by its transform ${\cB}_1$ (with a $T$-permissible center $C$, as above) will be called the {\it transformation} of $\cB$ with center $C$, indicated by 
$\cB \leftarrow \cB_1$, if $C$ is clear. Sometimes we let $\cT (\cB,C)$  denote the transform of $\cB$ with center $C$.

 One may verify that the transform of a $T$-basic object with a $\cB$-permissible center over $T$ induces  the transform of the  fiber $B^{(t)}$ with center $C^{(t)}:=C \cap W^{(t)}$, for all $t \in T$.
 Moreover, working with proper transforms, 
${\overline {{I}_1}}^{(t)}=
 \overline {{I^{(t)}}_1}$, for all $t \in T$. 

A sequence of $T$-basic objects  $\cB _0 \leftarrow \cdots \leftarrow \cB _r$ is $T$- {\it permissible} if    
    $\cB _{i+1}=\cT (\cB _i,C_i)$, where $C_i$ is a $\cB _i$-center, $i=0, \ldots, r-1$. 

\smallskip

Such a   sequence is called an 
{\it equiresolution} of ${\cB}_0$ if $\sg ([{\cB}_r]^{(t)})= \emptyset$, for all $t \in T$.

\end{voi}

\begin{voi}
\label{V:ne2}
Consider a family of basic objects $\cB = (\pi : W \to T,I,b,E)$,  
with $T$ (and hence $W$) regular. Let $B =(W,I,b,E)$  be the basic object associated to the family $\cB$ and  
 $$(1) \quad B=B_0 \leftarrow \cdots \leftarrow B_r$$
(we write $B_i=(W_1,I_i,b,E_i)$) the algorithmic resolution of $B$, obtained via resolution functions $g_0, \ldots, g_{r-1}$, which gives us resolution centers $C_i := {\rm {Max}}(g_i) \subset W_i $, $i=0, \ldots, {r-1}$. Let ${q_i:W_i \to W_0}=W$ be the composition of the induced blowing-up morphisms 
$f_i: W_i \to W_{i-1}$,   $\pi _i:=\pi {q_i}:W_i \to T$ and   $p_i:C_i \to T$  the restriction of $\pi _i$ to $C_i$, $i=1, \ldots, r$. For any point (or geometric point) $t$ of $T$
 let 
 $$(2)_t \quad  \quad B^{(t)}=(B^{(t)})_0 \leftarrow \cdots \leftarrow (B^{(t)})_{r_t}   $$
 be the algoritmic resolution of the fiber $B^{(t)}$, obtained via resolution functions $g^{(t)}_0, \ldots, g^{(t)}_{{r_t}-1}$, which determine  algorithmic resolution centers $C^{(t)}_i={\mathrm {Max}}(g^{(t)}_i)$, $i=0, \ldots, {r_t}-1$.
 
Throughout the remainder of this section we shall use the notation just introduced.
\end{voi} 

\begin{voi}
  \label{V:R}
Given a family $\cB$ as above, if $x \in \sg(B)$ then $x \in \sg(B^{(t)})$ (with $t=\pi(x), \, B=aso(\cB)$). The converse is not always true. For instance, we have the following example (see also Examples \ref{E:san1} and \ref{E:san2}). 

\smallskip

\noindent
{\it Example}. Consider the $T$-basic object  $B=(W \to T,I,3,\emptyset)$, where $W \to T$ corresponds to the inclusion of polynomial rings (over a field $k$) 
$k[t] \subset k[t,x]$, and $I$ to the principal ideal $(tx+x^3)$. If $B^{(0)}$ is the closed fiber, then 
$\sg(B)=\emptyset$ but $\sg(B^{(0)})=V(x) \not= \emptyset$. 

\smallskip

So, we introduce:

 {\it Condition (R)}.    A family $\cB$, as in \ref{V:ne2}, satisfies condition ($R$) if, for all $t \in T$, $x \in W^{(t)}$, we have 
$x \in \sg(B) $ if and only if $x \in \sg(B^{(t)})$.

This means: if $\nu _x (I) < b$ then $\nu _x (I^{(t)})<b$, for all $x \in W^{(t)}$, $t \in T$. 
\end{voi}

\begin{voi}
\label{V:prep}
 Here is a situation where $T$-regular systems of parameters (\ref{V:ne0})  appear naturally. With the assumptions and notation of \ref{V:ne2}, consider the sequence (1) and assume the induced projection $p_0:C=C_0 \to T$ is smooth. Since $C$ is regular, given $w \in C$  we may find a regular system of parameters $a_1, \ldots , a_d $  of ${\cO}_{W,w}$  such that, for some $n \le d$, the elements 
 $a_1, \ldots, a_n$
 induce a regular system of parameters of 
 ${\cO}_{{W^{(t)}},w}={\cO}_{W,w}/r({\cO}_{T,t}){\cO}_{W,w}$ (with $t=\pi(w)$); moreover $a_1, \ldots,a_r$, for some $r<n$, generate $I(C)_w \subset {\cO}_{W,w}$ 
 and $I(H)_w=(a_i)$ for each hypersuface $H$ in $E$ containing $w$, for a suitable $i$.  Thus, $a_1, \ldots,a_r$ are part of a $T$-regular system of parameters and $C$ has normal crossings with $E$, relative to $T$.  
 
The fact that $a_1, \ldots,a_r$ is a regular sequence implies: if 
 ${\cT}(B,C)=B_1=(W_1,I_1,b,E_1)$, then $\pi _1:W_1 \to T$ is again smooth (see \cite{NE}, Proposition 10.5). It is easily checked that 
 ${\cB}_1=  (\pi _1:W_1 \to T,I_1,b,E_1)$ is a $T$-basic object. If the induced projection $ C_1 \to T$ is again smooth, we may repeat the process. Thus, if $C_i \to T$ are smooth, $i=0, \ldots n$, from the sequence (1) we obtain  we obtain $T$-basic objects 
${\cB}_i=(\pi _i :W_i \to T, I_i, b, E_i)$, $i=0, \ldots, n.$

This observation and  
 \ref{V:R} 
 motivate the following notion.
\end{voi}

\begin{Def}
\label{D:ae}
{\it Conditions ($A_n$)}. Here we use the notation of \ref{V:ne2} and \ref{V:prep} and  we assume $T$ is regular. 
 A family of basic objects $\cB$, whose associated basic object $B$  has the sequence (1) of \ref{V:ne2} as algorithmic resolution, satisfies condition $A_n$, $0 \le n < r$ if, for $i=0, \ldots, n$,  (i) the induced morphism $p_i:C_i \to T$   is smooth and surjective and 
(ii)  $\cB _i$ satisfies condition $(R)$, $i=0, \ldots, n$.
\end{Def}

\begin{rem}
\label{R:us1}
(a) If ${\cB}=(W \to T, I, b, E)$ is a family of basic objects (as in \ref{V:ne2}), then its restriction to a suitable open dense subset of $T$ is 
 such that all the projections $C_i \to T$ (notation of \ref{V:ne2}) will be smooth. This is a consequence of  the Generic Smoothness Theorem (\cite{H}, page 272). Using it,  we find an open $U$ in $T$ such that over $U$ all the necessary projections are smooth.
 
(b)  If ${\cB}$ is a family as in (a), then its restriction to a suitable open dense subset of $T$
 will satisfy condition ($R$). To check this fact we may assume that $T$ is irreducible. We may find 
an affine open set $U$ of $T$, $U= \Spec (D)$, with $D$ a suitable integral domain, and a covering of 
${\pi}^{-1}(U)\setminus {\sg B} \subset W$ 
by affine open sets 
$W'_i = \Spec (G_i)$, $i=1, \dots, q$, with the following properties: 
($i$) $G_i$ is a finitely generated $D$-algebra, say $G_i=D[x_1, \ldots, x_n]$, such that the generators $x_1, \ldots,x_n$ induce on 
${\cO}_{W,z}$ a $T$-regular system of parameters , for each closed point $z \in W'_i$ (we use the fact that $\pi$ is smooth), 
($ii$) for each $i$, the restriction of $I$ to $W'_i$ corresponds to an ideal 
$(f_1, \ldots, f_m)G_i$ so that 
$f_1=\sum a_{\alpha} x_1^{\alpha _1}\cdots {x_n}^{\alpha _n}$,
$a_{\alpha} \in D$ for all $ \alpha$,  where for some coefficient 
$a_{\beta (i)} \not= 0$, $\beta (i)=(\beta _1, \ldots, \beta _n)$ we have 
$\beta _1 + \cdots + \beta _n < b$.

Let $V_i=U \setminus V(a_{\beta (i)})$ and $V=\bigcup_{i=1}^{i=q} V_i$. Then $V$ is a dense open set in $T$ 
 satisfying the requirements.

(c) As a consequence of (a) and (b), there is an   open dense subset $G$ of $T$ such that the restriction of $\cB$ to $G$ satisfies condition $A_{r-1} $ (and hence ($A_n$), for all $0 \le n <r$).

\end{rem}

\begin{voi}
\label{V:ne4} 
Sometimes, given a family $\cB$  we want to compare values of the resolution functions $g_i$ of $B$ (the basic object associated to $\cB$)  and those of the resolution functions $g^{(t)}_i$ of the fiber $B^{(t)}$ at $t \in T$. A problem is that if $x \in W^{(t)} \subset W$ then  $g_i(x) \in \Lambda^{(d)}$ while 
$g^{(t)}_i(x) \in \Lambda^{(d_t)}$ (where $d$ and $d_t$ are the dimensions of $W$ and $W^{(t)}$ respectively). We may circumvent this difficulty as follows (see \cite{EN}, 1.13 (3)).

If $B_0=(W_0 \to \Spec \,( k),I_0,b,E_0)$ is a basic object in $\cS$, consider the basic object $B_0[n]=(W_0[n],I_0[n],b,E_0[n])$ of \ref{V:a8}. Recall that  $W_0[n]:=W_0 \times_{k}{\bf A}^{n}_{k}$, let $p:W_0[n] \to W_0$ be the first projection. 
Note that the fiber $p^{-1}(x)$, $x \in W_0$, is isomorphic to 
${\bf A}^{n}_{k(x)}$. Take the algorithmic resolution    
$$(1) \quad B_0 \leftarrow B_1 \leftarrow \cdots \leftarrow B_s$$
 of $B_0$, obtained by means of resolution functions 
 $h_0, \ldots , h_{s-1}$.   This resolution induces by pull-back via $p$ a resolution  
$$(2) \quad B_0[n] \leftarrow (B_0[n])_1 \leftarrow \cdots \leftarrow (B_0[n])_s$$
of $B_0[n]$. A feature of the V-algorithm  we are using is that (2) is the algorithmic resolution of $B_0[n]$ (having resolution functions 
$h_{n,0}, \ldots,h_{n,s-1} $). Note that $us((B_0[n])_i)=W_i\times_{k}{\bf A}^{n}$, $W_i=us(B_i)$. Let 
$p_i:W_i\times_{k}{\bf A}^{n} \to W_i$ be the first projection. 

Now, if $x \in W_i$, write $h_i[n](x):=h_{n,i}(\tilde x)$, where $\tilde x$ is the generic point of the fiber 
${p_i}^{-1}(x) \cong {\bf A}^{n}_{k(x)}$. Thus, by replacing the value $h_i(x)$ by $h_i[n](x)$, we have a natural way to view, by means of this ``shifting'' process, the values of the resolution functions of $B_0$ as elements of 
${\Lambda}^{(n+d)}$ ($d=\dim \, W_0$). 

So, to compare  values of $g_i$ and $g^{(t)}_i$ we use instead $g_i$ and $g^{(t)}_i[d-d_t]$ respectively. Often, to simplify, when the meaning is clear from the context we simply write $g^{(t)}_i(x)$ to denote $g^{(t)}_i[d-d_t](x)$.
\end{voi}

\begin{voi}
\label{V:np} 
 The algorithm we are using (i.e., the V-algorithm) has the following property. We use the notation and assumptions of \ref{V:ne2}, so \ref{V:ne2} (1) is the algorithmic resolution of $B$.  Then if $x \in \sg (B_{i+1})$ and $f_i(x) \notin C_i$, we have $g_{i+1}(x) = g_i(f_i(x))$. 
From this the following statement is proved. Let $0={\alpha}_{i,1} < \cdots < {\alpha}_{i, {m_i}}$ be the set of values of the resolution function $g_i$ and 
 $S_{i,m_j}=\{x \in \sg(B_i): g_i(x)={\alpha}_{i,m{_j}}\}$ (which is a regular locally closed subscheme of $W_i$). Then, if $x \in S_{i,j}$, there is an index $s \ge i$, a point $x' \in \sg (B_s)$, neighborhoods $U$ of $x$ (in $S_{ij}$) and $V$ of $x'$ (in $C_s$) respectively, such that the natural map $W_s \to W_i$ induces an isomorphism 
 $V \stackrel{\sim}{\rightarrow} U$.
 
 From this observation it follows that if all the centers $C_i$ (in \ref{V:ne2} (1)) are smooth over $T$, then all the projections $S_{i,j}\to T$ are smooth.
 
 \end{voi}

\begin{voi}
\label{V:ct} 
 
 We review some results from  \cite{EN}, Section 6. 

(a) Consider a family $\cB$, as in \ref{V:ne2}, with $T$ regular, let $B=B_0=(W_0,I_0,b,E_0):=aso({\cB})$. Let $t $ be a closed point of $T$, consider the algorithmic resolutions (1) and $(2)_t$ of \ref{V:ne2}. We know that the fiber $B^{(t)}$ is a new basic object. Alternately, we express this fact by saying that $\cB$ is {\it 0-compatible with the algorithm at $t$} (or just {\it 0-compatible at t}). 
 If, moreover,  (i) $C_0$ is transversal to $W_0^{(t)}$ and (ii) $C_0^{(t)}=C_0 \cap W_0^{(t)}$, we say that $\cB$ is 1-compatible with the algorithm at $t$. 
In this case, the blowing-up $W^{(t)}_1$ of $W^{(t)}_0$ with center $C_0^{(t)}$ may be identified to the fiber at $t$ of the blowing up $W_1$ of $W_0$ with center $C$. Indeed, by (i) at each point $x \in C_0 \cap {W_0}^{(t)}$ the projection $C_0 \to T$ is smooth, hence $I(C_0)_x$ is generated by a regular sequence, so we may use Proposition \ref{P:bu}. If  $C_1$ is transversal to $W_1^{(t)}$ and  $C_1^{(t)}=C_1 \cap W_1^{(t)}$, we say that $\cB$ is 2-compatible with the algorithm at $t$.   Iterating, we define the notion ``$\cB$ is $s$-compatible with the algorithm at $t$, for  any index $s$ such that $0 \le s \le r$''. We say that $\cB$ is $s$-compatible with the algorithm (or just $s$-compatible, if this is clear) if for all $t \in T$ the localization of $\cB$ at $t$ is $s$-compatible with the algorithm at the closed point. 

If $\cB$ is $s$-compatible with the algorithm at $t$ and, in addition and after identifications already discussed, $\sg (B_j) \cap W_i^{(t)}=\sg(B_i^{(t)}), \, j=0, \ldots, s$, we say that $\cB$ is {\it strongly $s$-compatible} with the algorithm at $t$. If this holds for every $t \in T$ we say that $\cB$ is {\it strongly $s$-compatible}. 

(b) The proof of Theorem (6.4) of \cite{EN} shows that if $\cB$ (as above) is $q$-compatible with the algorithm at a closed point $t \in T$ and $x \in W^{(t)}_r \cap \sg(B_r)$ ($ W^{(t)}_r$ is identifiable to a closed subscheme of $W_r$), then 
$g_q (x)\le g^{(t)} _q (x)$ with equality if and only if  
$S_{q,j}\cap W_q^{(t)}$ is transversal at $x$ (see \ref{V:np}). If $\cB$ is strongly $r$-compatible then the same formulas hold, more generally, for every $x \in W^{(t)}_r$ (see \ref{R:Ay}) 

From this (using the observations of \ref{V:loc}) similar results follow at each point $t \in T$, closed or not, provided $\cB$ is compatible (or strongly compatible) with the algorithm. 
\end{voi}

\begin{rem}
\label{R:Ay} The statement of Theorem 6.4 in \cite{EN} is not correct. The given proof shows that the conclusion is valid only if we assume that $x \in \sg(B_r) \cap W_r^{(t)}$.  The examples that follow, due to S. Encinas, show that  
the claimed equality may fail if $x \notin \sg(B_r)$. If we add the condition that $\sg(B_r) \cap W_r^{(t)} = \sg(B_r^{(t)})$ (that is, strong compatibility) then the equality is valid as stated (since 
$g_r(x)=0 \Leftrightarrow x \notin \sg(B_r) $, similarly at the fiber at $t$). 

In case $b=1$ in our basic object, the condition $\sg (B_r) \cap W_r^{(t)}=\sg (B_r^{(t)})$ is automatic. Indeed, then $x \notin \sg(B_r) \Leftrightarrow J_x={\cO}_{W,x} \Leftrightarrow 
J_x^{(t)} = {\cO}_{W_r^{(t)}}    \Leftrightarrow        x \notin \sg (B_r^{(t)})$. So, if $b=1$, Theorem 6.4 of \cite{EN} is correct as stated. This is the only case where that theorem is applied in the article \cite{EN}. Indeed, it is used for objects of the form $(W,J,1,E)$, which correspond to families of ideals (see section \ref{S:IV}). So, this necessary change in the statement of Theorem 6.4 does not affect the results of \cite{EN}.
\end{rem}

\begin{exa}
\label{E:san1} Let $\cB=(W \to T, J, 2, \emptyset)$, where $W \to T$ is induced, by taking spectra,  from the inclusion of polynomial rings 
$k[t] \subset k[t,x,y]$
 (where, say, $k$ is the complex numbers), $I=(x^2 + t y^2)$. Then, with $0$ the origin $(t)$ of $T={\bf{A }}^{1}$, 
$B^{(0)}= (\Spec (k[x,y],(x^2),2, \emptyset)$, $\sg (B^{(0)}=V(x)$, $\sg (B)=V(x,y)$. Any closed point $z=(t,x,y)=(0,u,0)$ with $u \not= 0$ is in $\sg(B^{(0)})$ but not in  $\sg(B)$. Hence, $g_0(z) =0$ but $g^{(t)}_0(z) > 0$, although $\cB$ is (vacuously) 0-compatible with the algorithm.
\end{exa}
  
\begin{exa}
\label{E:san2}
  Let $\cB=(W \to T, J, 2, \emptyset)$, where $W \to T$ is as in \ref{E:san1} and $J$ corresponds to the ideal $(x^2y^2+tx^4)$. Here, 
$B^{(0)}= (\Spec (k[x,y],(x^2y^2),2, \emptyset) $. The $T$-basic object $\cB$ is 1-compatible. Indeed, 
$\max (g_0)= \max(g^{(t)}_0) =( 4/2,0,1,0, \infty)$ and we have $C_0=\ma (g_0)=V(x,y) \subset \Spec (k[t,x,y] = W$, while  
$C^{(t)}_0 = \ma(g^{(t)}_0)=V(x,y)$. So, 
$C^{(t)}_0 = C_0 \cap W^{(t)}_0$, the intersection being transversal. 

Transform with center $C_0$. Working in the relevant affine open of  $W_1$ (the blowing-up of $W$) and using, to simplify, still $t, x, y$
 to denote the coordinates, we have the controlled transforms 
$J_1 = y^2(x^2+t y^2)$ and $J^{(0)}_1=(y^2x^2)$; here $y$ defines the exceptional divisor. Then 
$\sg(J_1,2)=V(y)$, $\sg(J^{(t)_1})=V(x) \cup V(y)$. If 
$z$ (lying over $t \in T$) is in $  V(x) \setminus V(y)$ then $g^{(t)}(z)=(2,0, \infty) > 0$ but $g_1(z)=0$. 
\end{exa}

\begin{voi}
 \label{V:ct1}
Here we discuss  situations that insure compatibility with the algorithm. Consider a family $\cB$ as in \ref{V:ct}, whose notation we retain. 

(a) As mentioned, $\cB$ is always $0$-compatible. Assume now 
$\cB$ satisfies condition $A_0$. Then $\cB$ satisfies $(R)$ and hence, for all $t \in T$, 
$\sg(B) \cap W ^{(t)}= \sg (B^{(t)})$ and $\cB$  is strongly 0-compatible. Also, $C_0$ is smooth over $T$, which is equivalent to saying that, for all $t \in T$, $C_0$ is transversal to the fiber $W^{(t)}$. By the strong $0$-compatibility   that we have, \ref{V:ct} (b) implies that 
$g_0(x)=g^{(t)}_0$, for each $x \in W^{(0)}$. It easily follows that  
$\ma (g_0) \cap W_0^{(t)}= \ma(g_0^{(t)})$, i.e., $C_0 \cap W_0^{(t)}=C_0^{(t)}$. Thus (if $r > 1$) we have 1-compatibility. Assume now  that ($A_1$) holds. i.e., $C_1$ is smooth over $T$, and (in the notation of \ref{V:prep}) $\cB _1$ satisfies condition $R$. Then we also have strong 1-compatibility.  
 Proceeding as above, we see that
$g_1(t)=g^{(t)}_1(x)$ for each $x \in W_1^{(t)}$ and $C_1 \cap W_1^{(t)}=C_1^{(t)}$; moreover if $r >2$ we have 2-compatibility. If    $(A_2) $ (and hence $(R)$) holds, we have strong 2-compatibility. Repeating this process we see  that if $\cB$ satisfies condition $(A_{q})$ then,
 for  $i=0, \ldots, q$, 
 we have strong $i$-compatibility.  Hence,  
$g_i(x)=g^{(t)}_i(x)$ for each $x$ lying over $t$
 and   $C_i \cap W_i^{(t)}=C_i^{(t)}$,  $i=0, \ldots, q$. 

(b) Similarly, assume $\cB$ is such that 
$g_0(x)=g^{(t)}_0(x)$  for each $x \in W^{(t)}_0$. By \ref{V:ct} (b), this implies that $C_0$ is transversal to $ W^{(t)}_0$ 
 and hence that $C_0$ is  smooth over $T$. As in (a), we see that $C^{(t)}_0 =C_0 \cap W^{(t)}_0 $, hence $\cB$ is 1-compatible. If 
$g_1(x)=g^{(t)}_1(x)$  for each $x \in W^{(t)}_1$, again $C_1$ is smooth over $T$ and, as in (a), we have 2-compatibility. Repeating, and using th identifications of \ref{V:ct} (a), it makes sense to write 
$g_i(x)=g^{(t)}_i(x)$  for each $x \in W^{(t)}_i$,  $i < r$ and, if this is valid, then $C_i$ is smooth over $T$ for all $i$. 

(c) The discussion of (a) together with  \ref{R:us1} (c) show that a $T$-basic object $\cB$  as above (with $T$ irreducible) is always compatible with the algorithm at the generic point of $T$.  

\end{voi}

\begin{Def}
\label{D:f} {\it Conditions ($F_n$)}. A family $\cB$ (as in \ref{V:ne2}) satisfies condition ($F_n$) (where  $0 \le n <r$) if 
for all 
$x \in {W_i}^{(t)}$, 
 we have 
$g_i(x)=g^{(t)}_i(x)$, $0 \le i \le n$. Here we use the conventions of \ref{V:ne4} (i.e. we view $ g^{(t)}_i$ as a function with codomain $\Lambda ^{(d)}$,  $d= \dim (\cB)$, for all $i$) and the remarks of \ref{V:ct1} (b). 
\end{Def}

The discussion of \ref{V:ct1} may be summarized as follows.

\begin{pro}
\label{P:AnFn} Conditions $(A_n)$ and $(F_n)$ are equivalent.
  \end{pro}

\begin{voi}
\label{V:ne6.1} 
Here is an important example of a $T$-permissible sequence (see \ref{V:ne6}). We work with the notation and assumptions of \ref{V:ne2}. In particular $T$ is smooth, $B$ is the basic object associated to $\cB$, (1) of \ref{V:ne2} is  the algorithmic resolution of $B$. Assume 
$\cB$ satisfies condition $(A_{q-1})$, in particular 
 the projections $p_i:C_i \to T$ are smooth, for $0 \le i < q < r$. Let $\cB _i = (\pi _i : W_i \to T, I_i,b, E_i)$. We claim that each $C_i$ is a $\cB _i$-center, $i=0, \ldots, q-1$. Thus using these centers we obtain a $T$-permissible sequence 
$\cB _0 \leftarrow \cdots \leftarrow \cB _q$. If $q=r$, this is called the {\it T-sequence associated to the resolution sequence of $B$}. 

To verify this assertion, all we must check is that $\nu(I_i,D)=\nu(I_i^{(t)},D^{(t)})\ge b$, for each irreducible component $D$ of $C_i$, $i=0, \ldots, q-1$, $t \in T$. 
 If $i=0$, this follows from the the fact that $g_0(x)=(\o _0(x),n_0(x))=(\nu _x (I) /b,n_0(x))$ (similarly for  $g_0^{(t)}(x)$) and that 
$g_0(x)=g_0^{(t)}(x)$ for all $x \in W_0 ^{(t)}$, $t=\pi (x)$. Indeed, this is true because $(A_n)$ (with $n=q-1$) holds, and hence (by \ref{P:AnFn}) $(F_n)$ is also valid. For $i > 0$, one uses again the equality $g_i(x)=g_i^{(t)}(x)$ (a consequence of $(F_n)$ and Proposition 
5.3 of \cite{NE}, which relates the transforms $I_i$ and ${\bar{I_i}}$.
\end{voi}

\begin{rem}
\label{R:C}
The assumptions and notation are those  of 
\ref{V:ne2}. For all $t \in T$, we always have $p_0^{-1}(t) = C_0 \cap W^{(t)}$ (scheme-theoretic intersection), where $p_0:{C_0 \to T}$ is the induced projection.  Suppose that, in addition,  
   ${C_0 \cap W^{(t)}}=C^{(t)}_0$ (the zero-th  center of the algorithmic resolution of the fiber $B^{(t)}$), for all $t \in T$.
    We claim that this implies that  $C_0$ is flat over $T$.
    
    To verify this statement, by the fact that both $C_0$ and $T$ are regular, by \cite{M}, page 179  this will be true if for every $t \in T$, the dimensions of  the generic fiber of the projection $p_0:C_0 \to T$  and of    $p_0^{-1}(t)$    agree.  Now, $p_0^{-1}(t)=C_0 \cap W^{(t)}$, by our assumption this equals $C^{(t)}_0$. Moreover, by the regularity of the center $C^{(t)}_0$, this intersection is transversal and so, by \ref{V:np}, $g^{(t)}_0(w) = g_0(w)   $, for any point $w \in C^{(0)}_0$ (where 
  $g^{(t)}_0$ and $  g_0   $ are zeroth-resolution functions of $ B^{(t)}$ and $B$ respectively). Now, since by \ref{R:us1} (a) the restriction of $\cB$
to a suitable non-empty open set of $T$ satisfies condition ($A_{r-1}$), and hence, by \ref{P:AnFn}, condition ($F_{r-1}$), the zeroth-resolution function $g^{(u)}_0$ of the generic fiber $W^{(u)}$ satisfies 
$g^{(u)}_0(z)=g_0(z)$, for any $z \in  C_0 \cap W^{(u)}$ ($u$ is the generic point of $T$). Since $g_0$ is constant along $C_0 = {\rm {Max}}(g_0)$, $g^{(t)}(w)=g^{(u)}(z)$. By a property of the V-algorithm, the fiber $p_0^{-1}(t)=C^{(t)}_0$ and the generic fiber $p_0^{-1}(u)=C_0 \cap W^{(u)}=C^{(u)}_0$ are equidimensional, of the same dimension. Thus, the projection 
$C_0 \to T$ is flat at any point of $p_0^{-1}(t)$. Since $t$ is arbitrary,  $p_0: C_0 \to T$ is flat.

Let $\cB _1={\cT}(\cB,C_0)$ be as in \ref{V:prep} (hence $B_1$ of \ref{V:ne} (1) is the associated basic object to the family $\cB_1$).   
Then, for all $t \in T$, by  \ref{P:bu} (in case $T' \subset T$ is the inclusion of the point $t \in T$)  there is an identification of $(B^{(t)})_1$ and  $[{\cB} _1]^{(t)}$ (the fiber of $\cB _1$ at $t$). In particular, $W^{(t)}_1=us((B^{(t)})_1$) may be identified to $\pi _1 ^{-1}(t)$  (a subscheme of $W_1$) and $p_2 ^{-1}(t) = C_1 \cap W^{(t)} _1$. Assume 
  $C_1 \cap W^{(t)} _1 = (C^{(t)}) _1$, the first resolution center in the sequence  $(2)_t$ in \ref{V:ne}.
  Reasoning as above, we conclude that $C_1 \to T$ is flat. So, if $\cB _2 = \cT(\cB _1, C_1)$, then $({B}^{(t)})_2$ can be identified to $(B_2)^{(t)}$, for all $t \in T$,  and so on. 
  
  With these identifications the following definition  makes sense:
  \end{rem}
  
\begin{Def}
\label{D:cn} {\it Conditions ($C_n$)}.  A family $\cB$ (as in \ref{V:ne2}) satisfies condition ($C_n$) (with $0 \le n <r$) if, for all $t \in T$, in the sequence $(1)$ of \ref{V:ne2}  we have  $C_i \cap W^{(t)}_i=C^{(t)}_i$, $i=0, \ldots, n $, where $ C^{(t)}_i$ is the $i$-th center used in the resolution $(2)_t$ of \ref{V:ne2} (identified to a subscheme of $us({\cB}_i)$ as in \ref{R:C}). Here the intersection is in the scheme-theoretic sense, hence equal to $p_i^{-1}(t)$. 
\end{Def}

\section{Some notions of algorithmic equiresolution}

\label{S:BE}
\begin{voi}
\label{V:ne}  
Throughout this section we make the following assumptions:

(i)  all the schemes are in the class $\cS$ (see \ref{V:a2.0}). 

(ii) {\it The algorithm} will mean the V-algorithm of resolution for basic objects in $\cS$ (or, equivalently, the W-algorithm, see \ref{V:a13.0}, \ref{V:a14.0} and \ref{P:igua}). 

(iii) All the families of basic objects that we shall consider will be parametrized by a regular scheme.

We work throughout with a family 
${\cB}=(W \to T, I, b, E)$, $T$ regular, with associated object $B$.  We shall use the algorithmic resolutions $(1)$ and $(2)_t$ of   \ref{V:ne2}.

We intend to introduce several possible notions of algorithmic equiresolution on such a family $\cB$.  
\end{voi}

\begin{Def}
\label{D:A}
{\it Condition ($A$)}.  We say that $\cB$ satisfies condition ($A$) if:  
 ($i$) $\cB$ satisfies condition $A_{r-1}$  (see  \ref{D:ae}, $r$ is as in (1) of \ref{V:ne2}) and 
  ($ii$) $\sg (B^{(t)}_r) = \emptyset$, for all $t \in T $.

Part ($ii$) is equivalent to the assertion that $\cB _r$ satisfies condition ($R$) (see \ref{V:R}).
\end{Def}

\begin{voi}
\label{V:ne3}
 Part $(ii)$ of the previous definition does not follow from $(i)$. For instance, take 
$\cB = (W \to T, (tx+x^3),3, \emptyset)$, with $W= \Spec ({\bf C}[t,x])  $, $T=\Spec ({\bf C}[t])$.  Since $\sg (aso(\cB)) =\emptyset$, part ($i$) is vacuously satisfied here, but $\sg (B^{(0)}) \not= \emptyset$.

\begin{Def}
\label{D:fe} {\it Condition ($F$)}. A family $\cB$  satisfies condition ($F$) if: ($i$) ${\cB}$ satisfies condition ($F_{r-1}$) 
(see \ref{D:f})    and ($ii$) for all $t \in T$, in $(1)$ and $(2)_t$ of \ref{V:ne2} we have $r=r_t$.
\end{Def}

\begin{Def}
\label{D:c} {\it Condition ($C$)}.  A family $\cB$  satisfies condition ($C$) if: ($i$) it satisfies condition $C_{r-1}$ (see \ref{D:cn}) and ($ii$)  $r=r_t$ for all $t \in T$ (notation of \ref{V:ne2}). 
\end{Def}

\begin{pro}
 \label{P:AFC} Conditions ($A$), ($F$) and ($C$) are equivalent.
\end{pro}

\begin{proof}
 We already saw (in \ref{P:AnFn}) that $(A_{r-1}) \Leftrightarrow (F_{r-1})$ (notation as in \ref{V:ne2}). So, if ($A$) holds, part (i) of condition ($A$) implies that ($F_{r-1})$ is valid. Since $x \in \sg(B^{(t)})$ if and only if $g^{(t)}_{0}(x) > 0$, this implies that the length $r_t$ of the algorithmic resolution of the each fiber $B^{(t)}$ must satisfy $r_t \ge r$, for all $t \in T$. But (ii) in condition ($A$) implies that $r_t$ must be exactly $r$, for all $t \in T$. Thus, part ($ii$) in Condition ($F$) also holds.

Conversely ($ii$) in condition ($F$) clearly implies that  $\sg (B_r^{(t)})=\emptyset$, i.e., ($ii$) in condition ($A$). This proves the equivalence of ($A$) and ($F$)

\smallskip

{\it The implication $(F) \Rightarrow (C)$}. Use the notation of \ref{V:np}. 
 By \ref{V:ct} we see from our hypothesis that $C_0 =S_{0,m_0} \subset W_0$ is smooth over $T$. Moreover, for any $t \in T$, the intersection $C_0 \cap W^{(t)}$ is transversal, hence this is regular scheme (in particular reduced). Since  $C^{(t)}$ is also regular, to check the equality of schemes $C_0 \cap W^{(t)} = C^{(t)}$ it suffices to check the equality of underlying sets, and this easily follows from our assumption ($F$). Since $C_0$ is smooth over $T$, the basic object $B_1$ induces, as in 
\ref{V:prep},
 the family ${\cB}_1=\cT(\cB,C_0)$, and we may use the same argument as before to get $C_1 \cap W^{(t)}_1 = C^{(t)}_1$ and the smoothness of $C_1 $ over $T$. Iterating, we see that condition ($C$) holds.

\smallskip

 {\it The implication $(C) \Rightarrow (A)$}. We prove first that the morphism $p_0:C_0 \to T$ is smooth. Since both $T$ and $C_0$ are schemes over a field of characteristic zero it suffices to show that: (a) $p_0$ is flat and (b) the fibers are regular schemes. Since both $C$ and $T$ are regular schemes, to verify the flatness it suffices to verify that all the components of non-empty closed fibers have the same dimension (see \cite{M}, page  179).
 But for  a closed point $t$ of $ T$, $p_0^{-1}(t)=C_0 \cap W_0^{(t)}$ (scheme-theoretic intersection), and by assumption this is equal to $C^{(t)}_0$ (the zero center of the algorithmic resolution $(2)_t$ of \ref{V:ne2}). We know that this is regular, so (b) above is verified. Moreover, the regularity of the intersection forces $C_0$ and $W_0^{(t)}$ to meet transversally at each common point. Thus, by \ref{V:np}, for all $x \in C^{(t)}_0$, $g^{(t)}_0(x)=g_0(x)$, which must be equal to max ($g_0$). But it is known that the dimension of  the component of 
 $C^{(t)}_0$ containing $x$ is determined by the value $g^{(t)}_0(x)$. Thus, the equality just gotten implies that all the components of fibres $p_0^{-1}(t)$ have the same dimension, i.e., (a) above is valid. So, $p_0$ is smooth. 
 
  As in 
 \ref{V:np}, the basic object $B_1={\cT}(B,C_0)$ (in (1) of \ref{V:ne2}) determines a family 
 ${\cB}_1=(W_1 \stackrel{\pi _1}{\rightarrow} T,I_1,b,E_1)$, with $\pi _1$ the composition 
 $W_1 \to W \stackrel{\pi}{\rightarrow} T$. We may repeat the argument above (using the equality 
 $p_1^{-1}(t)=C_1 \cap W_1^{(t)}=C^{(t)}_1$, with $p_1$ the restriction of $\pi _1$ to $C_1$) to obtain the smoothness  
  of $p_1$. Iterating, by the assumed validity of condition ($C$), we get: $p_i$ is smooth, 
  $i=0, \ldots, r-1$. 

Now we check that $\cB$ satisfies condition $(R)$. Let $x \in W \setminus \sg (B)$, $t= \pi (x)$. We must show: 
$x \in  W^{(t)} \setminus \sg (B^{(t)})$. We may assume $t$ is a closed point. Let 
$x \in \sg (B^{(t)})$, i.e., $x \in S_{i \alpha} ^{(t)}$, $\alpha > 0$ (notation of \ref{V:np}). There is an index $j$ and a point $x_j \in W_j$ such that, locally, $S_{i \alpha} ^{(t)}$ is isomorphic to the algorithmic center $C_j ^{(t)} \subset W_j$ (near $x_j$). But, were 
$x \notin \sg(B)$ then $x_j $ is not in $\sg (B_j)$ and hence $x_j \in C_j \subseteq \sg (B_j)$. Thus  
$C_j \cap W_j ^{(t)} \not= C_j^{(t)}$, contradicting condition ($C$). Thus,   by \ref{V:ne6.1},        $C_0$ is a permissible $T$-center. In a similar way, we prove the $\cB _1 = \cT (\cB, C_0)$ satisfies condition $(R)$. Iterating, we see that 
${\cB}_i$ satisfies condition ($R$), $0 \le i < r$. So, this discussion, together to the previous one, shows that condition ($A_{r-1}$), i.e., part ($i$) of (A), holds.
 Part ($ii$) of condition ($A$) is clear.
 Thus, ($C$) $\Rightarrow$ ($A$) is proved.
\end{proof}

In the next condition we need certain auxiliary functions (see \cite{EN}, 2.2). Given a family $\cB$ (as in  \ref{V:ne}), if $t \in T$ let $c_i(t)$ denote the number of irreducible components of $C^{(\bbar t)}_i$ (the $i$-th  center of the algorithmic resolution of the geometric fiber $B^{(\bbar t)}$). Write 
 $${\tau}_{\cB}(t) := ({\rm max}(g^{(\bbar t)}_0), c_0(t), \ldots, {\rm max}(g^{(\bbar t)}_{r_t}), c_{r_t}(t), \infty, \infty, \ldots) $$
where $\infty$ is the maximum of $\Lambda ^{(d_t)}, ~ d_t=\dim \, W^{(t)}$.
 \end{voi}

\begin{Def}
\label{D:tau} {\it Condition $\tau$}. A family of basic objects $\cB$ parametrized by $T \in \cS$ (not necessarily regular) satisfies condition $\tau$ if the function ${\tau}_{\cB}$ is locally constant. 

\smallskip
Since $\tau _{\cB}$ is upper-semicontinuous, the requirement in condition $\tau$ simply means that $\tau _{\cB}$, restricted to each connected component of $T$, is constant.   
\end{Def}

\begin{rem}
\label{R:nume}
 If $q:C \to T$ is a smooth, proper morphism of noetherian schemes, with $T$ integral, then $c(t)$, the number of irreducible components of the geometric fiber at $t \in T$, is independent of $t$. A proof of this result is found in \cite{EN}, (2.5) (iii).
\end{rem}

\begin{pro}
 \label{P:tau} Let $\cB$ be a family, assume all the projections $C_i \to T$ are proper morphisms. Then $\cB $ satisfies condition (A) if and only if it satisfies condition $(\tau)$.

\end{pro}

\begin{proof}
{\it The implication $(A) \Rightarrow (\tau)$}. We know that ($A$) implies condition ($F$). Hence  $r_t = r$ for all $t \in T$ and the sequence $({\mathrm max}(g^{(t)}_0), \ldots, {\mathrm max}(g^{(t)}_r))$ is independent of $t \in T$. Now we have to check that the sequence 
 $(c_0(t), \ldots, c_r(t))$ is constant. 
  Since by condition ($C$) (valid, because it is equivalent to (A)) we have that $p_i^{-1}(t)=C_i^{(t)}$ for all $t \in T$. Then \ref{R:nume}  implies that $c_i(t)$ is independent of $t$. This proves the desired implication.

 {\it The implication $(\tau) \Rightarrow (A)$}. Since $(A) \Leftrightarrow (F)$, we prove that $(\tau) \Leftrightarrow (C)$. First we shall verify the following facts about the family $\cB _0$ (we use the notation of \ref{V:ne2}).
 \begin{itemize}
 \item [($a_0$)] max($g_0$)=max($g^{(t)}_0$), for all $t$ in $T$. 
 
 \item [($b_0$)] $p_0:C_0 \to T$ is smooth.
 
 \item [($c_0$)]$C_0$ intersects $W^{(t)}$ transversally, and $C_0 \cap W_0^{(t)}=C^{(t)}_0$, for all $t \in T$.
 \end{itemize}
 
 {\it Proof of $(a_0)$}. Let $z$ be the generic point of $T$ and $\alpha$ the common value of ${\mathrm {max}}(g^{(t)}_0)$, $t \in T$. 
 Since $\cB$ is compatible with the algorithm at the generic point $z$ (\ref{V:ct1}), by \ref{V:ct} (b) we have 
$\alpha = g_0(z)=\max \{ g_0(w): w \in W^{(z)} \} \le \max ({g_0})$. On the other hand, if $\max (g_0)$ is reached at $x \in E$ (with $t=\pi (x)$) we have: 
 $\max (g_0) = g_0(x) \le g^{(t)}_0 \le \max (g^{(t)}_0)= \alpha$. These inequalities imply $\max (g_0)= \alpha$, i.e., $(a_0)$.

 {\it Proof of $(b_0)$} By \ref{V:ne2}, for $x \in C_0$ we have $g_0(x) \le g^{(t)}_0(x)$. Since $C_0$ is the locus of maximum value of $g_0$, by $(a_0)$ both members of this inequality must be $= \alpha$, so 
$g_0(x) = g^{(t)}_0(x)$.
  Again by \ref{V:ct} (b),  $p_0$ is smooth at $x$. Since $x$ is arbitrary, $p_0 : C_0 \to T$ is smooth.

 {\it Proof of $(c_0)$} The transversality follows from the smoothness of $p_0$. Hence, 
 $C_0 \cap W^{(t)}$ is regular, a fortiori reduced. Also $C^{(t)}_0$
 is regular, hence reduced. Consequently, to establish 
the equality stated in $(c_0)$ we may proceed set-theoretically. The inclusion 
$C_0 \cap W_0^{(t)} \subseteq C^{(t)}_0$ easily follows from $(a_0)$. Note that 
$C_0 \cap W_0^{(t)} = p_0 ^{-1}(t)$. To show that we must have an equality, we may work with the geometric fiber, i.e., to show that 
$$ (1) \quad p_0^{-1}(\bbar{t})=C^{\bbar{t}}_0$$
with $\bbar{t}$ the geometric point of $T$ at $t$. Now, since for each point $x \in C_0$ we have 
$g_0(x)= g^{(t)}_0(x) = \alpha$, it follows that each irreducible (or connected) component of each side has the same dimension (see 
\cite{EN}, (1.2)(vi) and (1.15)). So, to show the equality (1) it suffices to show that the number $c'$ of components of 
$ p_0^{-1}(\bbar{t})$ equals the number $c$ of components of  $C^{\bbar{t}}_0$. Now, following \ref{R:us1} (a), restrict to an open set $U \subseteq T$ where $(\cB _0)_{|U}$ satisfies Condition (A). Note that the generic point $z$ is in $U$. Hence, by ($C$) (equivalent to (A)), $C^{(z)}_0=C_0 \cap W^{(z)}=p_0^{-1}(z)$. By \ref{R:nume}, by the properness of $p_0$, $c'=c_0(z)$. So, $c=c_0(t)=c_0(z)=c'$, as needed to prove $(c_0)$.
 
 Now consider ${\cB}_1= \cT({\cB}_0,C_0)$. The fiber $B_1(t)$ is identified to 
 $\cT(B^{(t)}_0, C_0 \cap W^{(t)})$ and hence, by $(c_0)$, to $\cT(B^{(t)}_0, C^{(t)}_0)$, i.e. the term $(B^{(t)})_1$  in the sequence $(2)_t$ of \ref{V:ne2}. So, we may state properties $(a_1)$, $(b_1)$ and $(c_1)$, analagous to $(a_0)$, $(b_0)$ and $(c_0)$, and prove them exactly as before. Iterating this procedure,  eventually we get that, in particular,  $(c_j)$ is valid for all $j$, i.e., (i) of condition ($C$). Part (ii) of ($C$) follows from the fact that $\max (g_i ^{(t)})$ is constant.  
 \end{proof}

 \begin{rem}
 \label{R:nosa} Equiresolution should be regarded as a form or equisingularity. Accordingly, one should expect that when a family (say, of basic objects) is equisolvable, then the different fibers would be ``equally complicated'', e.g., have ``the same'' algorithmic resolution. But in this sense, in general, Condition (A), or the ones equivalent to it, are not very satisfactory. If we fix a closed point $t=0 \in T$ and consider the fiber $B^{(0)}=(W^{(0)},I^{(0)},b,E^{(0)})$ of the basic family $\cB =(W \to T, I, b, E)$ at $0$, then we get an open neighborhood $U$ of $W^{(0)}$ in $W$ such that for all $t \in T$ (near $0$) $(B^{(t)})_{|U}$ has the ``same'' resolution as $B^{(0)}$. But, it might be the case that for all $t \not= 0$ in $T$ the whole fiber $B^{(t)}$ might be very different to  $B^{(0)}$. See the example that follows, where $B^{(0)}$ is non-singular (hence already resolved) but $B^{(t)}$ is not, for $t \not= 0$. 

Problems like those of this example cannot occur if all the projections $p_i :C_i \to T$ (notation of \ref{V:ne2}) are proper, as the equivalence (A) and ($\tau$) in this case indicates. So, following the present approach, it seems that the requirement 
  ``$p_i$ is proper for all $i$'' is important to insure satisfactory results. Of course, this requirement is satisfied under the assumption that $W$ is proper over $T$. 
 \end{rem}
 
 \begin{exa}
 \label{E:bad} In this example condition (A) is satisfied and  
  $B^{(0)}$ is non-singular (already resolved) while $B^{(t)}$ is not, for $t \not= 0$. In this case the projection $p_0:C_0 \to T$ is not proper. 
   Here, $k$ is a characteristic zero field, $R=k[x,y,t]$, a polynomial ring in three variables, $O$ the origin of ${\bf A}^{3}=\Spec (R)$ (so, $\{O\} = V(x,y,t)$), $W={\bf A}^{3} - \{O\}$, $T=\Spec (k[t])$, $\pi:W \to T$ the morphism induced by the inclusion $k[t] \subset R$, 
 $\cB = (\pi:W \to T, I,2,\emptyset)$, where $I$ is the restriction to $W$ of the sheaf of ideals (on ${\bf A}^{3}$) 
 defined by the ideal $(x^2-y^2)R$, $B=(W,I,2,\emptyset)$. Here, $\sg (B)=V(x,y)$ (the $t$-axis, restricted to $W$), write $C=\sg(B)$. If ${B}_1=\cT({B},C)$, 
 then $B \leftarrow {B}_1$ is the algorithmic resolution of 
 $B$. Clearly $C$ is smooth over $T$, hence $\cB$ satisfies condition (A). But $\sg (B^{(0)})=\emptyset$ (so, $B^{(0)}$ is resolved, where $0$ is the origin of ${\bf A}^{3}_{k(t)}$) while $\sg(B^{(t)})\not= \emptyset$, fot $t \not= 0$. 
  \end{exa}

\section{Condition ($E$)}

\label{S:CE} 
 
\begin{voi}
\label{V:ce2}

   We want to introduce another equiresolution condition, condition ($E$), which makes sense even for families whose parameter space is not reduced. Our approach uses the theory developed in \cite{NE}.

For the remainder of this article, the symbol ${\mathcal A}$ denotes the collection of artinian local rings $(A,M)$ such that the residue field $k=A/{M}$ has characteristic zero. Such a ring is necessarily a complete $k$-algebra.

A family of basic objects of the form 
$$(1) \quad \cB = (W {\stackrel{\pi}{\rightarrow}} S,I,b,E),~S=\Spec (A),~A \in \cA$$ 
  will be called an {\it infinitesimal family}.
Denote by  
 $B^{(0)}=(W^{(0)},I^{(0)},b,E^{(0)})$  its only fiber (so, $W^{(0)}$ is a smooth $k$-scheme, $k=A/M$). 

The main observation is that, given an infinitesimal   family,  sometimes  the algorithmic resolution of the fiber ``reasonably spreads'' to    a resolution of the family.
   When this is possible, we say that the family is algorithmically equisolvable.   Condition ($E$) requires that all the infinitesimal families naturally induced at the  points of $T$ be algorithmically equisolvable. We shall summarize  results of \cite{NE} that we need in order to make this idea precise.  Although in \cite{NE} one works exclusively with infinitesimal families, some of its results can be presented in greater generality. This fact will be useful in section 7.
\end{voi}

\begin{voi}
\label{V:ne7} We briefly indicate how several concepts discussed in section \ref{S:AL} for basic objects (over fields) can be adapted to $T$-basic objects, $T \in \cS$.
 
 \smallskip
 
 (a) In \ref{V:a4.0} we defined the $W$-ideal $\Delta(I)$ (or $\Delta (I/k)$),  for a basic object $B=(W,I,b,E)$, with $W$ a scheme (in $\cS$) over the field $k$. 
 Given a family $\cB$ as in \ref{D:fa}, we may similarly introduce a {\it relative}  sheaf ${\Delta}(I/T)$. Namely, we set    
 ${\Delta}(I/T):=  I+ {\mathcal F}_{d-1}({\Omega}_{{\nu}(I)/T}) $, where $\cF$ denotes {\it Fitting ideal}, $d$ is the relative dimension of $W$ over $T$, and ${\Omega}_{{\nu}(I)/T}$ is the sheaf of relative differentials. This is a $W$-ideal. By iteration we define the ideals 
${\Delta}^{i}(I/T)$. If $t$ is a closed point of $T$, $W^{(t)}$ the fiber of $W \to T$ at $t$, then ${\Delta}^{i}(I/T){\cO}_{W^{(t)}}$ can be identified to 
${\Delta}^{i}(I^{(t)}/k(t))$ (see \ref{V:a4.0}).

\smallskip

(b) {\it  W-equivalence.} Two $T$-basic objects $\cB=(W \to S, I, b, E)$ and $\cB'=(W \to S, J, c, E')$ are {\emph {pre-equivalent}} if the following conditions hold: $C \subset W$ is a $\cB$-permissible center if and only if it is a $\cB'$-permissible center, if $\cB_1=\cT(\cB,C)$ and $\cB'_1=
\cT(\cB',C)$ respectively then $C_1\subset us(\cB _1))=us(\cB _1 '))$ is a $\cB_1$-permissible center if and only if it is a $\cB'_1$ permissible center, and so on. We say that $B$ and $B'$ are {\emph {W-equivalent}} if they are pre-equivalent and, in addition, the  fibers 
$B^{(t)}$ and $B'^{(t)}$ are also pre-equivalent, for all $t \in T$ (see \cite{NE}, 4.8). 

\end{voi}

\begin{voi}
\label{V:ne8.1}
We continue discussing how to  adapt notions introduced in \ref{V:a10.0} and  \ref{V:a12.0} to  families of basic objects  
${\cB}=(\pi:W \to S,I,b,E)$, $S \in \cS$.
 
\smallskip

(a) We get a {\it relative} homonegeized ideal ${\cH}(I/S,b)$ (or just ${\cH}(I/S) $, if $b$ is clear) by using formula (1) in \ref{V:a10.0}, but with $\Delta^{i} (I)$ substituted by the relative $W$-ideal $\Delta^{i} (I/S)$, and $T= \Delta^{b-1} (I/S)$. 
 Thus we obtain a homogeneized family ${\cH}({\cB})$ associated to $\cB$.  The families $\cB$ and ${\cH}({\cB})$ are W-equivalent. 

\smallskip

(b) We define $S$-hypersurface, $S$-transversality to $E$, etc., as in \ref{V:a12.0} ($\alpha$) but now  using (partial) $S$-regular systems of parameters, in the sense of \ref{V:ne0}, instead of regular systems. We define ``a $S$-hypersurface is adapted to $\cB$ (or $\cB$-adapted, or just relatively adapted or $S$-adapted, if this is clear) as in \ref{V:a12.0}, replacing (A1) by ``$I(Z) \subseteq {\Delta}^{b-1}(I/S)$'', (A2) by ``$Z$ is $S$-transversal to $E$''. If, moreover, $Z$ satisfies the analogue of (A3) of \ref{V:a12.0}, it is called {\it inductive}. A family is {\it nice} if it admits a $S$-adapted hypersurface. The {\it relative coefficient ideal} ${\mathcal{C}}(I/S)$ is defined as in \ref{V:a12.0} ($\beta$), but now using the sheaves $\Delta^{i} (I/S)$ rather than $\Delta ^{i}(I)$. 

In a similar way we introduce the inductive $A$-object ${\cB}_{Z}=(Z \to S,{\mathcal C}(I/S,Z),b!,E_Z)$, as in \ref{V:a12.0} ($\beta$) (where now  $\cB$ is nice, $Z$ is a $S$-inductive hypersurface,  $E_Z$ consists of the intersections of each $H$ in $E$ with $Z$ and the morphism $Z \to S$ is induced by $\pi$).

\end{voi}

\begin{voi}
\label{V:ne8} In \cite{NE} other notions  studied in Section \ref{S:AL} in the case where we work over a field are extended  
 to the case of infinitesimal families. Actually, as we shall see, those results hold in a more general setting, for instance if  $\cB $ is an $A$-basic object,  $A$ a local ring such that 
$S=\Spec (A) \in \cS$. We let $0$ denote the closed point of $S$ and $B^{(0)}$ the 
closed (or {\it special}) fiber of $\cB$.  In this subsection, all the local rings we consider will be such that their spectra are in $\cS$ of  \ref{V:a2.0}. 

 (a) {\it Monomial objects}.  Let $A$ and $S$ be as above, we use the notation of \ref{V:a7.0} (i)).  An $A$-basic object     ${\cB}=(W \to S, I,b,E)$, where $E=(H_1, \dots, H_m)$, is {\it premonomial} if its  closed 
fiber $B^{(0)}$ is monomial. Let ${\Gamma}=(\Gamma_1,\Gamma_2,\Gamma_3) := {\Gamma}_{B^{(0)}} $. We say that  $B$ is {\it monomial} if it is premonomial and, letting
$(i_1, \ldots, i_p,0,0, \ldots) = {\rm max}~ (\Gamma_3)$, then
$C:= H_{i_1} \cap \cdots \cap H_{i_p}$ is a $B$-permissible center. This is called the {\it canonical center } of the monomial $A$-basic object ${\cB}$.

(b) {\it t-permissible centers}. Let  $A $ a local ring, consider  a sequence  $A$-permissible transformations of $A$-basic objects  
$$\quad (1) \quad  {\cB}_0 \leftarrow {\cB}_1 \leftarrow \cdots \leftarrow {\cB}_r$$
where we write  ${\cB}_j=(W_j \to S, I_j, b, E)$, for all $j$.
 We shall define, by induction on the length $r$, what it means that (1) is $t$-permissible, or simply a $t$-sequence.

If $r=0$ (i.e. there is just one basic object), the sequence (reduced to one object) is, by definition, $t$-permissible.

Next, assume the notion of $t$-permissible center is defined, by induction, if the sequence has length $\leq r$, in such a way that it  induces a sequence  of special fibers which is $t$-permissible, in the sense of  \ref{V:a7.0}(ii). We declare a sequence of length $r+1$ to be $t$-permissible if the following conditions $\alpha$ and $\beta$ hold. ($\alpha$) The $r$-truncation of (1) is $t$-permissible . Then
 by looking at special fibers we have functions 
 $\o _i$ and  $t_i$, $i=1, \ldots, r$ 
  satisfying 
${\rm max} ~(t_i) \geq {\rm max}~(t_{i+1}) $, $i=0, \ldots , r$.  
 Let $s$ be the largest index such that 
${\rm max}\,({\o}_s) = {\rm max} \,({\o}_r)$. Let $E_r^{-}$ consist of the hypersurfaces in $E_r$ which are transforms of those in $E_s$ and $(b_r/b,\bbar{n})= {\mathrm  max}(t_r)$. Then we demand: ($\beta$) any component $C$ of the center $C_r$ used to obtain $B_{r+1}$ satisfies:  
 $\nu (\bbar{I_r},C) =  \nu (\bbar{I^{(0)}_r},C^{(0)})= b_r$ and  for each closed point $y \in C$,  
 the number of hypersurfaces in $E^{-}_r$ containing $y$ is equal to $\bbar{n}$. 

If (1) is 
  a $t$-permissible sequence, each center $C_j$ used in it is said to be $t$-permissible for ${\cB}_j$, and by 
  the function $t _j$ of (1) we mean the function $t _j$ of the corresponding sequence of special fibers. 

  If (1) is a $t$-sequence, then the induced sequence of closed fibers is a $t$-sequence, in the sense of \ref{V:a7.0}. The $t$-functions of (1) are, by definition, the $t$-functions of the induced sequence of fibers (\ref{V:a7.0}).
    
  (c)  The sequence (1)   is called $\rho$-permissible if there is an integer $s \ge 0$ such that: 
(a) ${\cB}_0 \leftarrow \cdots \leftarrow {\cB}_s$  is $t$-permissible, (b) ${\cB}_j$ is monomial if $s \le j$ and,  for all such $j$, ${\cB}_j \leftarrow {\cB}_{j+1}$ is the transformation with the canonical center of ${\cB}_j$. In particular, it could be $s =r$, in this case the sequence is $t$-permissible.
\end{voi}

 \begin{voi}
 \label{V:ce3}

If $\cB =(W \to S,I,b,E)$ is a family of basic objects, $S=\Spec (A)$, $A \in \cA$,
$w \in W$ is a closed point, and 
 $a_1, \ldots, a_n$ are elements of ${\cO}_{W,w}$ inducing a regular system of parameters of ${\cO}_{W^{(0)},w}$, then the  completion $R^*$ of $\cO _{W,w}$ with respect to the ideal $(a_1, \ldots, a_n)$ is isomorphic to a power series ring  $R^*=A'[[x_1, \ldots,x_n]]$ (with $A'={\cO}_{W,w}/(a_1,\ldots,a_n)$  and  $ \Delta ^i (J/S)R^*$ is the ideal generated by elements of $I$ and their partial derivatives of order $\leq i$ (see \cite{NE}, Prop. 11.6). 
 
   If $C$ is an irreducible closed subscheme of $W$, in \ref{V:ne1.1} we introduced the integer $\nu (I,C)$.  In 
 3.9 and 3.10 of \cite{NE} it is proved that  $\nu (I,C) \ge b$ if and only if 
$\Delta ^{b-1}(I/S)_w \subseteq J_w$ for $w$ in a dense open subset of $C$.  
 Also, if $w$ a is closed point of $ C$, one has:  ${\Delta}^{b-1}(I/S)_w \subseteq J_w$ if and only if for every $f \in I_w$ the corresponding power series in 
 $I_wR^* \subset A'[[x_1,\ldots,x_n]]$ has order $\ge b$.
\end{voi}

\begin{voi}
\label{V:bdp}
 Here we assume $S=\Spec (A)$, $A$ a local ring. If (1) of \ref{V:ne8} is a $t$-permissible sequence of $A$-basic objects 
 one  defines, locally at each point 
$x \in \ma (t_r) \subset W^{(r)}= us({\cB}_r)$, 
 the associated nice $A$-basic object  ${{\cB}_r}''$ using essentially the same procedure as in \ref{V:a12.0} ($\beta$). Here are some properties of ${{\cB}_r}''$ (assuming, to simplify the notation, that it is globally defined): 
 (a) $({\cB}_r'')^{(0)} = ({\cB}_r^{(0)})''$. 
 (b)  A center $C$ for ${\cB}_r$ is $t$-permissible if and only if it is ${\cB}_r''$-permissible. 
 (c) Let $C \subset W_r$ be a ${\cB}_r$-center that is $t$-permissible, consider the transformations ${\cB}_r'' \leftarrow ({\cB}_r'')_1$ and 
${\cB}_r \leftarrow {\cB}_{r+1}$ with center $C$, and the  object ${\cB}_{r+1}''$ associated to ${\cB}_{r+1}$,  assume 
$\mathrm{max}(t_r) = \mathrm{max}(t_{r+1})$. Then, 
$({\cB}_r'')_1=({\cB}s_{r+1})''$. 
 (d) If $A_n=A/{r(A)}^{n+1}$ (a ring in the class $\cA$), $S_n = \Spec (A_n)$, 
${\cB}^{(n)}$ is the $A_n$ basic object induced by $\cB$, then ${{\cB}}''$ induces the associated $A_n$-basic object 
${{\cB}^{(n)}}''$.
 We leave it to the reader to precisely state and verify the analogous statement for the $S$-basic object ${{\cB}_r}''$ that appears in the sequence (1) of \ref{V:ne8}.

Statements (a), (b), (c) 
 are discussed in detail in \cite{EN}, section 8,  in case $A \in \cA$, but the extension to the situation where $A$ is just local is  straightforward. Statement (d) follows from the definitons.
\end{voi}

\smallskip
Condition ($E$) involves the use of certain auxiliary conditions ${\cE}_j$, introduced in \cite{EN}, that we review next.

\smallskip

\begin{voi} 
\label{V:ne9}  
{\it Conditions $\cE _j$}. Let $\cB$ be an $A$-basic object,
  $A \in \cA$
$$(1) \quad B^{(0)}:=(B^{(0)})_0 \leftarrow (B^{(0)})_1  \leftarrow \cdots  \leftarrow (B^{(0)})_r  $$ 
the algorithmic resolution of its fiber $B^{(0)}$, with resolution functions $g ^{(0)} _j$, determining algorithmic centers 
$C^{(0)}_{j}= \ma (g^{(0)}_j), \, j=0, \ldots, r-1$.   We let 
$\o ^{(0)}_0, \o^{(0)} _1, \cdots$ and $t ^{(0)}_0, t^{(0)} _1, \cdots$ denote the $\o$- and $t$-functions of (1), respectively. Now, we introduce certain conditions 
${\cE}_{j}$ ($j=0,1, \ldots, r-1$) on $\cB$, which may be valid or not.   These will satisfy the following properties: 
\begin{itemize}
\item[(a)] 
If $\cE _j$ is valid, then ${\cE}_i$ is valid, for $i<j$. 
\item[(b)] 
If $\cE _j$ is valid, then it is defined a permissible sequence of $A$-basic objects  
$$ (2)_j \quad {\cB}_{0} \leftarrow \cdots \leftarrow {\cB}_{j}\leftarrow {\cB}_{j+1}$$
with centers $C_i \subset us({\cB}_i)$, $i=0, \ldots, j$, inducing on fibers the $j+1$-truncation of (1), in such a way that the sequence $(2)_i$ associated to condition $\cE_i$,  valid by (a), is the truncation of $(2)_j$. We shall say that $C_i$, $j=0,1, \ldots, j$, is the $i$-th associated center of $\cB$ (determined by the validity of $\cE _j$) 
\item[(c)]
These conditions are stable with respect to etale morphisms $W'\to W$ and change of the base artinian ring, $A' \to A$ (see \cite{NE}, 9.1).
 \end{itemize} 
 
 The sequence $(2)_j$ is called the $j$-th partial algorithmic equiresolution determined by condition $\cE_j$. 
  If conditions $\cE _0, \ldots, \cE _{r-1}$ are valid, we say that $\cB$ is {\it algorithmically equisolvable} and call the resulting sequence $(2)_{r-1}$ {\it the algorithmic equiresolution of} $\cB$.
 The intuitive meaning of these conditions is as follows: if $\cE _j$ holds, then the algorithmic resolution process of the fiber, up to level $j$, nicely spreads over the parameter space $\Spec (A)$.
 
 These conditions, satisfying the mentioned properties, are defined inductively on the dimension of $\cB$, as explained next. We just state the main facts, the proofs can be seen in \cite{NE}, Section 8. Notice that if, in (1) $r=0$, then
  $\sg (B^{(0)})=\emptyset$ and $\cB$ is vacuously equisolvable. So, in the sequel we may assume $r > 0$.
\end{voi}

\begin{voi}
\label{V:ne10}
{\it The case where dim ${\cB}=1$.}  We shall define, for $0 \leq j < r$ (with $r$ as in (1) of \ref{V:ne9}), conditions ${\cE}_{j}$, in such a way that if ${\cE}_{j}$ is valid, then the resulting sequence $(2)_j$ of \ref{V:ne9} is $\rho$-permissible (see \ref{V:ne8} (c)). 

Start with ${\cE}_0$. Then necessarily max$(\o ^{(0)} _0)>0$. Consider an open cover 
$\{U_i\}$ of Max($t_{0}$) such that each $U_{i}$ the family ${\cB}''_i=(U_i \to S, (I_i/S)'',b'', E''_i) $ is defined (see \ref{V:ne8.1} (c)). Then we require that, for all $i$,  
$\Delta^{b''-1}(I''_i/S)$ define a ${\cB}''_i$-permissible center $C_i$.  The different $C_i$ glue together to yield a well-defined center $C$ which is $\cB$-permissible. Moreover, $C$ is $t$-permissible. This is the center that $\cE _0$ attaches to $\cB$. 

Now, assuming ${\cE}_s$ defined for $s < j$, we introduce condition ${\cE}_j$. There are two cases: (a) max$(\o ^{(0) }_j)>0$, (b) 
max$(\o ^{(0)} _j)=0$. 

In both cases,  first we require that conditions ${\cE}_s$ be valid for $s < j$. Hence, we have a permissible sequence ${\cB}_0 \leftarrow \cdots \leftarrow {\cB}_{j}$, with centers $C_i$, $0 \le i < j$. These will be the centers $\cE _j$ associates to $\cB$, for $i < j$.

Assume we are in  case (a). Then the sequence above is $t$-permissible. Apply to ${\cB}_{j}$ the technique used in the case $j=0$. Namely, cover Max($t ^{(0)}_{j}$) by  open sets, so that   nice $A$-basic objects 
$({\cB}''_{j})_i= (U_i \to S,I''_{i},b''_i,E''_i)$  are defined. We require that the subscheme defined by  
$\Delta ^{{b''_i} -1}(I''_i/S)$ be a $({\cB}_{j}'')_i$-center $C_{ji}$, for all $i$. As above, this is independent of the chosen cover, and these centers patch together to produce a $t$-permissible center $C_{j}$ for $B_j$. This is the $j$-th associated center, and it satisfies all the requirements.  
 
Now consider case (b). Here, the object ${\cB}_j$ is pre-monomial. To have condition  ${\cE}_j$ satisfied we require that ${\cB}_j$ be monomial, and we take as the $j$-th associated center the canonical monomial center. 
 \end{voi}
\begin{voi}
\label{V:ne11}
Now we study conditions ${\cE}_j$ when ${\dim \, {\cB}}=d$, an arbitrary positive integer. The case $d=1$ being established, we shall proceed by induction on $d$. 

So, assuming condition $\cE _j$ (having properties (a), (b), and (c) of \ref{V:ne9}) known when the dimension of the $A$-basic object is less than $d$, we introduce conditions $\cE _j$ for ${\cB}$ $d$-dimensional. This is done again recursively  on the length $j$ of the algoritmic sequence $(2)_{j-1}$ of \ref{V:ne9} that we get when condition ${\cE}_{j-1}$ is valid: 
$$ (1) \quad \cB={\cB}_0 \leftarrow \cdots \leftarrow {\cB}_j$$
This sequence is permisible,  with centers $C_i \subset us({\cB}_i)$, $i=0, \ldots j-1$, 
inducing on fibers a sequence 
$$ (2) \quad B^{(0)}=B_0^{(0)} \leftarrow \cdots \leftarrow B_j^{(0)}$$
(which is the $j$-truncation of the algorithmic resolution of $B^{(0)}$).  Looking at the functions $\o^{(0)}_p$ corresponding to the sequence (2), we distinguish two cases: 
(a) max $(\o^{(0)} _{j})=0$, (b) max $(\o^{(0)} _{j})>0$. In case (a), ${\cB}_j$ is pre-monomial. We declare 
 condition ${\cE} _j$ valid if ${\cB}_j$ is monomial, with canonical center  $C_j$. We take $C_0, \ldots , C_j$ as the centers that condition $\cE _j$ associates to $\cB$.  
 
 In situation (b), looking at the $t$-functions $t^{(0)}_p$ corresponding to the sequence (2), and letting $M:={\mathrm {Max}}(t^{(0)}_j)$ and $d=\dim(\cB)$, we distinguish the following two possibilities: 
 
 ($\alpha$) $\dim M <d-1$, 
 
 ($\beta$) $\dim M =d-1$. 
 
In case ($\alpha$) consider the index $q$ such that $\max (t^{(0)}_{q-1})>
\max (t^{(0)}_{q})$ but 
$\max (t^{(0)}_{q})=\max(t^{(0)}_{q+1})= \ldots = \max (t^{(0)}_{j})$.

Let  $M_q$ be the image in $W_q$ of $C^{(0)}_j$ (the $j$-th algorithmic center of $B^{(0)}$)   and take an open cover $\{V_{iq}\}$ of $M_q$ such that on each $V_{iq}$ we have a nice basic object (over $A$) 
${\cB}_{iq}''=(U_i \to S,({( {{I_i}_{|U_i}})/S} )'',b_i '',E_{i})$  
with inductive hypersurface $Z_{iq} \subset V_{iq}$. Next,  for each $i$, take the homogeneous $A$-basic object $(\cH{\cB}_{iq}'')$, again nice and admitting $Z_{iq}$ as inductive hypersurface. Then,  for each index $i$, consider the inductive object 
${\cB}^* _{iq}:=(\cH{\cB}_{iq}'')_{Z_i}$, 
of dimension one less than that of ${\cB}_j$. So, 
by induction 
the notion of algorithmic equiresolution for ${\cB}^* _{iq}$ if defined by the inductive hypothesis. To declare
 ${\cE}_j$ valid for $\cB$ we require: 

(a) for each index $i$ the $A$-basic object 
${\cB}^* _{iq}$ is algorithmically equisolvable, let 
$$ (3) \quad {\cB}^*_{iq} \leftarrow {\cB}^*_{i q+1} \leftarrow {\cB}^*_{ij} \leftarrow \cdots $$
be its algorithmic equiresolution, with centers 
$ C_{ip} \subset us({\cB}^{*} _{ip}) , \, q \le p$,

(b) after identifications, $C_{ip}$ is a permissible center of ${\cB}_{ip}$ (restricted to a suitable open set), $q \le p \le j$.

Let us explain (b) more carefully. First, letting 
$\widetilde{{\cB}_{iq} }$  denote the restriction
 of ${\cB}_q$ to $U_{iq}$, we require that $C_{iq}$ be 
$({\cH}    \widetilde{{\cB}_{iq} } )$-permissible and hence, by the equivalence mentioned in \ref{V:ne8.1} (a),  $\widetilde{{\cB}_{iq} }$-permissible. Let $U_{i,q+1}$ be the inverse image of $U_{iq}$ via the blowing-up morphism $W_{q+1} \to W_{q}$ induced by (1) and 
 $\widetilde{{\cB}_{i,q+1} } := {{\cB}_{i,q+1}}_{|U_{i,q+1}}$ We require that (after obvious identifications and the use of \ref{V:ne8.1}  (a) again), $C_{i,q+1}$ be $\widetilde{{\cB}_{i,q+1} }$-permissible; and so on. Thus we obtain 
an induced  
 $t$-permissible sequence
$$(4) \quad \widetilde{{\cB}_{iq} } \leftarrow \widetilde{{\cB}_{i,q+1} } \leftarrow \cdots \leftarrow \widetilde{{\cB}_{ij} }$$

Finally, we require that $C_{ij}$ be  $\widetilde{{\cB}_{i j} }$-permissible.
 In \cite{NE}, it is proved that for each $i$ the locally defined subschemes $C_{ip}$ of $us(\cB)_i=W_i$ agree on intersections, thus determining a closed subscheme $C_i$ of $W_i$.  

The schemes  $C_0, \ldots , C_j$ thus obtained are the centers  that $\cE _j$ attaches to ${\cB}_j$.

To finish, consider the case  ($\beta$). We proceed as in the one-dimensional situation. Namely, let $M(1)$ be the union of the components of ${\mathrm {Max}}(t_j)$ of codimension one. Consider locally defined nice objects 
${{\cB}_j}''_i = (U_i \to S, ({I_i /S})'', b''_i, E''_i)$ 
 as in 
\ref{V:bdp}, where the open sets $U_i$
cover $M(1)$ and we declare condition $\cE _{j}$ valid if for each $i$ the $W_j$-ideal $\Delta ^{b''_i - 1}(I''_i/S)$ defines a permissible center $C_{ij}$. Then  the different centers $C_{ij}$ define a $\cB _j$-center $C_j$. We take $C_0, \ldots , C_j$ as the centers that 
 $\cE _j$ associates to $\cB$. This is independent of the choice of the open cover $\{U_i\}$. 
 Conditions (a), (b) and (c) of \ref{V:ne9} are satisfied.

 So, we have defined conditions $\cE _j$ in all cases.  
 \end{voi}

 \begin{voi}
 \label{V:ne12}
 Now we introduce our last equiresolution condition. We shall use the following notation. If 
 $\cB=(W \to T,I,b,E)$ is a family of basic objects in $\cS$ and $t \in T$, then 
 $A_{t,n}:={\cO}_{T,t}/({M}_{T,t})^{n+1}$ (where $M_{T,t}=r({\cO}_{T,t})$), $S_{t,n}:=\Spec (A_{t,n})$ and ${\cB}_{t,n}$ denotes the pull-back of $\cB$ to $S_{t,n}$ via the natural morphism $S_{t,n}\to T$. 
 \end{voi}
 \begin{Def}
\label{D:e}
 {\it Condition ($E$)}. A family $\cB$ (as in \ref{V:ne12}) satisfies condition ($E$) if  the induced family 
  $\cB _{t,n}$ is algorithmically equisolvable, for all $t \in T$, $n$ any non-negative integer (i.e., ${\cB}_{t,n}$ satisfies conditions  ${\cE}_j $, for all possible values of $j$). 
 \end{Def}

 In section  \ref{S:E} we shall prove that, for a $T$-basic object, with $T$ regular, conditions ($A$) and ($E$) are equivalent.

\section{Some consequences of the equiresolution conditions}
\label{S:S} 

In this section we gather some  material that we need in the proof of  the equivalence of conditions ($A$) and ($E$), which will be discussed  in the next section.

\begin{voi}
 \label{V:c1} Let
$\cB = (\pi: W \to T,J,b,E)$ be a nice basic object over a smooth scheme $T \in \cS$, admitting a $T$-inductive hypersurface $Z$;   
$B=(W,J,b,E)$ the  basic object associated to $\cB$. Note that $B$ is again nice, with inductive hypersurface $Z$. 

We have the induced inductive objects $ {\cB}_Z = (Z \to T, {\cC}(J/T,Z),b!,E_Z)$ (a $T$-basic object) and $B_Z= (Z, {\cC}(J,Z),b!,E_Z)   $. Similarly, we may consider the ``homogenized'' situation: 
 $({\cH}{\cB})_{Z}=(Z \to T, {\cC}({\cH}J/T,Z),b!,E_Z)$ 
$(HB)_{Z}= (Z , {\cC}({H}J,Z),b!,E_Z)$, which is relevant when we use the W-algorithm.

Notice that $B_Z \not= aso\,({\cB}_{Z})$. 
 Rather, if 
${\widetilde {\cB}}_Z:=(Z \to T, {\cC}(J,Z),b!,E_Z) $, then $B_Z=aso({\widetilde {\cB}}_Z)$. 
Similarly, $(HB)_Z=aso({{\widetilde{\cH \cB}}}_Z)$,  where
${{\widetilde{\cH \cB}}}_Z:=  (Z \to T, {\cC}({\cH}J,Z),b!,E_Z)$. We shall write 
$   B_{Z/T}:=aso({\cB}_Z)   $ and $  {(HB)}_{Z/T}:=aso({ {\cH}{\cB} }_Z)  $.
     But we'll see that, under suitable assumptions, there are useful connections between $B_Z$ and $B_{Z/T}$ (or between 
$(HB)_Z$ and $(HB)_{Z/T}$).
\end{voi}

\begin{pro}
\label{P:BZ} Let $\cB$ be a nice $T$-basic object, as in \ref{V:c1}, which satisfies condition($A$). Then the inductive $T$-basic objects  (a) 
${\widetilde {\cB}}_Z$ and  (b) ${\widetilde { {\cH}{\cB}}}_Z$ also satisfy condition ($A$).
\end{pro}

\begin{proof}   
Statement (a) follows from the equivalence of the basic objects $B$ and $B_Z$ (see 
\cite{BEV}, 12.9). Statement (b) is a consequence of (a) and the equivalence of $B$ and $HB$ (\ref{V:a10.0}).
\end{proof}   

\begin{voi}
\label{V:samce}
Consider basic objects $B=(W,I,b,E) $ and $B'=(W,J,d,E)$, let 
$B=B_0 \leftarrow B_1 \leftarrow \cdots \leftarrow  B_r$
and 
$B'=B'_0 \leftarrow B'_1 \leftarrow \cdots  \leftarrow B'_s$ 
be the algotithmic resolutions of $B$ and $B'$, obtained by using centers $C_0, \ldots , C_{r-1}$ and $C'_0, \ldots, C'_{s-1}$, 
respectively. Suppose $C_0=C'_0$. Then $B_1$ and $B'_1$ have the same underlying scheme, namely the blowing-up of $W$ with center $C_0=C'_0$. Assume $C_1=C'_1$. Then again $us(B_2)=us(B'_2)$, assume $C_2=C'_2$. and so on. In this way we may compare the centers $C_j$ and $C'_j$. If $r=s$ and $C_j=C'_j,~ j=0, \ldots, r-1$ we say that {\it $B$ and $B'$ have the same algorithmic centers.} 
\end{voi}

\begin{voi}
\label{V:not}
In the next two propositions we use the notation of \ref{V:c1} and, in addition, the following one.
$$(1) \quad B=B_0 \leftarrow \cdots \leftarrow B_r$$  
denotes the algorithmic resolution of $B$, $B_i = (W_i , J_i, b, E_i)$, with algorithmic resolution functions $g_0,\ldots, g_{r-1}$; for each $i$ we have  (by composition) induced morphisms  
$$f_i:W_i \to W ~{\mathrm  {and}}~ \pi \, f_i : W_i \to T.$$
Let $Z_i$ be the strict transform of $Z$ to $W_i$ via the morphism $f_i$. By means of $\pi f_i$, $W_i$ is a scheme over $T$, for all $i$. Let
$$(2) \quad B_{Z}={(B_Z)}_{0} \leftarrow \cdots \leftarrow {(B_{Z})}_{r}        $$ 
be the algorithmic resolution of $B_Z$, ${(B_Z)}_i = (Z_i , {\cC}_i( J,Z), b!, {(E_Z)}_i)$,  ${\tilde g}_i$, $0 \le i < r$ its $i$-th algorithmic resolution function, and  
 $$(3) \quad    
  B_{Z/T}={(B_{Z/T})}_{0} \leftarrow \cdots \leftarrow {(B_{Z/T})}_{s}      $$
the algorithmic resolution of $B_{Z/T}=aso \, ({\cB}_Z)$; the corresponding algorithmic resolution functions are denoted by $h_i$, 
$0\le i < s$.

Finally, the algorithmic resolution functions of the fibers $B^{(t)}$, ${(B_Z)}^{(t)}$ and ${(B_{Z/T})}^{(t)}$ (for $t \in T$, obtained via either the projection $W \to T$ or  
 $Z \to T$, induced by $\pi$)  are denoted by 
$ g_i^{(t)}$, ${\tilde g}_i^{(t)}$ and ${h_i}^{(t)}$ respectively. Note that when 
  this makes sense (see Remark \ref{R:C}), 
${g_i}^{(t)}={\tilde g}_i^{(t)}$.

When, discussing an algorithmic resolution process, we say that we are in the {\it inductive situation}, we mean case $(b_2)$ in \ref{V:a11.0}, i.e., that where we use the associated inductive object and induction on the dimension.
\end{voi}

\begin{pro}
 \label{P:imp1} We use the notation and assumptions of \ref{V:c1} and \ref{V:not}. Assume  $\cB$ satisfies condition ($A$). Then: 
(a)  ${B}_Z$ and $B_{Z/T}$ have the same algorithmic  centers;  
(b)   ${(HB)}_Z$ and ${(HB)}_{Z/T}$ have the same algorithmic centers.  
\end{pro}

\begin{proof}
 We'll check (a) only. The proof of (b) is entirely similar.  
 Let $C_0, \ldots, C_{s-1}$ be the algorithmic resolution centers of  $B_{Z/T}=aso({\cB}_Z)$, 
${\widetilde C}_0, {\widetilde C}_1, \ldots {\widetilde C}_{r-1}$ those of $B_Z =aso({\widetilde {\cB}}_Z)$. Initially we shall show that $C_0={\widetilde C}_0$.

Choose a dense open set $U \subseteq T$ such that ${{\cB}_Z}_{|U}$ satisfies condition ($A$) (use \ref{R:us1} (a)). We know that ${\widetilde {\cB}}_Z $ satisfies condition ($A$) 
(\ref{P:BZ}).    Hence, its restriction   
${({\widetilde{\cB}_Z)}}_{|U}$ also satisfies condition ($A$).

Now, we claim that for all $z \in Z$ such that $t=\pi (z) \in U$, ${\tilde g}_0 (z) = h_0(z)$. Indeed, we have:
$${\tilde g}_0(z) = g_0(z) = g^{(t)}_0(z)={{\tilde g}^{(t)}}_0(z)=h_0(z)$$
where for the second and fourth equality we use the fact that both ${\widetilde {\cB}}_{|U}$ and ${\cB}_{|U}$ satisfy ($A$), hence condition ($F$) (\ref{P:AFC}). 

 Let $\tilde m = \max ({\tilde {g_0}})$, $m= \max(h_0)$ (in both cases the maximum over all points of $Z$, not just those lying over $U$).  

We claim that ${\tilde m}=m$. We check first that ${\tilde m} \le { m}$.  

Since the induced projection ${\widetilde C}_0 \to T$ is smooth and hence (assuming ${\tilde C}_0$ non-empty) dominant, ${\tilde C}_0 \cap {\pi}^{-1} (U) \not= \emptyset$. Hence,  
${\tilde m} = \max \{{\tilde g}_0 (z) : z \in {\pi}^{-1}(U)\} =   
\max \{h_0(z):z \in {\pi}^{-1}(U) \}
 \le 
\max ({h}_0) = {m}$,  as claimed. 

 Now we check that ${ m} \le {\tilde m}$.  Let 
$z \in Z$ such that $h_0(z)=m$. Then:
$$ m=h_0(z) \le h_0^{(t)}(z)={\tilde g}_0^{(t)}(z)=g_0^{(t)}(z)=g_0(z)={\tilde g}_0^(z) \le {\tilde m}         $$
where in the third and fifth equalities we use the fact that we are in the inductive situation and in the fourth one the fact that $\cB$ satisfies condition ($A$), hence ($F$).

So we have proved that $m={\tilde m}$ and moreover, because of this equality, that if $z \in C_0$, i.e., $h_0(z)=m$, then ${\tilde g}_0(z)={\tilde m}=m$, i.e., $z \in {\widetilde C}_0$. That is, $C_0\subseteq {\widetilde C}_0$. But since ${\widetilde C}_0$ is smooth (hence flat) over $T$, it easily follows that 
 (letting $cl_Z$ denote scheme-theoretic closure in $Z$)
  ${\widetilde C}=cl _Z (C_0 \cap \pi ^{-1}(U))\subseteq C_0$. 
 Thus, $C_0={\widetilde C}_0$, as claimed. 

So, $C_0={\widetilde C}_0$ is a permissible center for both $B_Z$ and $B_{Z/T}$. By \ref{V:ne6.1}, $C_0$ is also a permissible center for the $T$-basic objects ${\widetilde  {\cB}}_Z$ and ${\cB}_Z$ respectively. Transform these $T$-basic objects with center $C_0$. We get:
$$ {\widetilde  {\cB}}_Z  \leftarrow {({\widetilde  {\cB}}_Z)}_1 ~ {\mathrm {and}} ~   
{ {\cB}}_Z   \leftarrow   {({  {\cB}}_Z)}_1            $$ 
respectively. Similarly, if we transform $B_Z$ and $B_{Z/T}$ with center $C_0$, we obtain:
$$B_Z \leftarrow {(B_Z)}_1  ~ {\mathrm {and}} ~ 
B_{Z/T}  \leftarrow {(B_{Z/T})}_1 $$
respectively. Note that ${(B_Z)}_1 = aso ({({\widetilde  {\cB}}_Z)}_1)$ and 
${(B_{Z/T})}_1 = aso ({({  {\cB}}_Z)}_1)$. All these objects have $Z_1$, the blowing-up of $Z$ with center $C_0={\widetilde C}_0$, as their  underlying scheme. Thus  $C_1$ (resp.  ${\widetilde C}_1$), 
  the algorithmic 1-center of $B_{Z/T}$ (resp. $B_Z$) is a closed subscheme of $Z_1$.  We want to show: $C_1={\widetilde C}_1$.  Note that the fiber of 
${(\cB)}_1$ at a point $t \in T$ is the $k(t)$-basic object 
${[B^{(t)} _{Z^{(t)}}]}_1$ (the transform of  
$  B^{(t)} _{Z^{(t)}}     $ with center $C_0 ^{(t)}$) while, with similar notation, that of ${\cB}_1$ is $B_1^{(t)}$. Since we are in the inductive situation, 
$g_1(z) = {\tilde g}_1(z)$ for all $z \in Z_1$. Proceeding as in the previous case, we verify: 
($i$) with the open set $U \subset T$ of before,  $h_1(z)={\tilde g}_1(z) $ for all $Z \in Z_1$ lying over $U$, ($ii$)  if $m_1= \max (h_1)$ and ${\tilde m}_1 = \max ({\tilde g}_1)$, then 
$m_1={\tilde m}_1$. Indeed, all what we need to apply the former arguments  is the smoothness of ${\widetilde C}_1$ over $T$, which is warranted because $\cB$ satisfies ($A$). Hence, as above, 
${\widetilde C}_1 =cl _{Z_1}( {{\widetilde C}_1}) \cap {{{\pi _1 }^{-1}}(U)} = 
cl _{Z_1}( {{C}_1}) \cap {{{\pi _1 }^{-1}}(U)} 
 \subseteq C_1$. Moreover they are equal, because if 
$z \in C_1$, i.e., $z \in Z_1$ and $h_1(z) =m_1$, then 
$$    m_1 =h_1(z) \le h_1^{(t)}(z) ={\tilde g}_1^{(t)}(z)=g_1^{(t)}(z)=g_1(z)={\tilde g}_1 (z) \le {\tilde m}_1         $$
 (the third and fifth equalities because we are in the inductive situation, the fourth because $\cB$ satisfies condition ($A$), hence ($F$))

Now, by \ref{V:prep}, $C_1={\widetilde C}_1$  is a permisible center for the $T$-basic objects 
${({\cB}_Z)}_1 $ and ${({{\widetilde B}_Z})}_1$. We may repeat the procedure. Iterating, we  get: for all $i$, $C_i={\widetilde C}_i$, $i=0, \ldots s-1$. 
 By condition ($A$), $\sg (B^{(t)})_r = \emptyset$ for all $t \in T$. This implies: $r=s$; that is the basic objects $B$ and $B'$ have the same algorithmic centers.
 \end{proof}

\begin{pro}
 \label{P:imp2} We use the notation and assumptions of \ref{V:c1}. Assume ${\cB}_Z$ satisfies condition ($A$). Then:   
(a)  ${B}_Z$ and $B_{Z/T}$ have the same algorithmic  centers, 
(b)   ${(HB)}_Z$ and ${(HB)}_{Z/T}$ have the same algorithmic centers.
\end{pro}

\begin{proof} Again we check case (a), the verification of (b) being almost identical.  

The proof is similar to that of Proposition \ref{P:imp1}. With the same notation, we prove the equality $C_i={\widetilde C}_i$, for all possible center ${\widetilde C}_i$. Check first  
$C_0={\widetilde C}_0$
In this situation $C_0 \to T$ is smooth. Let $V \subseteq T$ be a dense open set such that the restriction of $  {\widetilde{\cB}}$ to $V$    satisfies condition ($A$). Of course, ${{\cB}_Z}_{|V}$ also satisfies condition ($A$)).  
  As in the proof of \ref{P:imp1}, we check that ${\tilde g}_0=h_0$ over $V$.  Also,  letting 
 $m$ and $\tilde m$ be as previously, we see  that $m = \tilde m$. Indeed, as above, from the smoothness of $C_0$ over $T$, 
$C_0=cl _Z ({C_0} \cap {{\pi}_{-1}(V)})$. So, there is a point $x \in Z$ such that $\pi (x) \in V$ and  
$h_0(x)=m$. So, $m=h_0(x)={\tilde g}_0(x) \le \max ({\tilde g}_0)={\tilde m}$. On the other hand, if $x \in Z$ such that ${\tilde g}_0(x)= {\tilde m}$ (i.e., $x \in {\widetilde C}_0$), then 
$$   {\tilde m}\le {\tilde g}_0(x) \le {\tilde g}_0^{(t)}(x)=h_0^{(t)}(x) =h_0(x)\le m       $$
(as in the proof of \ref{P:imp1}).
Then,  $m={\tilde m}$ and all the inequalities above are equalities. Thus, if $x \in {\widetilde C}_0$ , then $h_0(x)=m={\tilde m}$, i.e., $x \in C_0$. Thus, ${\widetilde C}_0 \subseteq C_0$. As before, since $h_0$ and ${\tilde g}_0$ agree over $V$,  
${ C} = cl_Z ( {{ C}_0}) \cap {{\pi}^{-1}(V)}) = 
cl_Z ({ {\widetilde  C}_0})\cap {{\pi}^{-1}(V))} 
 \subseteq {\widetilde C}_0$. So, $C_0={\widetilde C}_0$, as claimed.

Now, we discuss the transition form level zero to one. The scheme $C_0=C$ is a permissible center for both $B_Z$ and $B_{Z/T}$. By \ref{P:AnFn}, $C_0$ is also a permissible center for the $T$-basic objects ${\widetilde  {\cB}}_Z$ and ${\cB}_Z$ respectively. Transform these $T$-basic objects with center $C_0$, as in \ref{P:imp1}. We get:
$$ {\widetilde  {\cB}}_Z  \leftarrow {({\widetilde  {\cB}}_Z)}_1 ~ {\mathrm {and}}   
{ {\cB}}_Z   \leftarrow   {({  {\cB}}_Z)}_1            $$

Similarly, if we transform $BZ$ and $B_{Z/T}$ with center $C_0$, we obtain:
$$B_Z \leftarrow {(B_Z)}_1  ~ {\mathrm {and}}
B_{Z/T}  \leftarrow {(B_{Z/T})}_1 $$
respectively. Note that ${(B_Z)}_1 = aso ({({\widetilde  {\cB}}_Z)}_1)$ and 
${(B_{Z/T})}_1 = aso ({({  {\cB}}_Z)}_1)$. All these objects have $Z_1$, the blowing-up of $Z$ with center $C_0={\widetilde C}_0$, as underlying scheme. Thus  $C_1$ (resp.  ${\widetilde C}_1$), 
  the algorithmic 1-center of $B_{Z/T}$ (resp. $B_Z$) is a closed subscheme of $Z_1$.  We want to show: $C_1={\widetilde C}_1$. 
 Note that the fiber of 
${(\cB)}_1$ at a point $t \in T$ is the $k(t)$-basic object 
${[B^{(t)} _{Z^{(t)}}]}_1$ (the transform of  
$  B^{(t)} _{Z^{(t)}}     $ with center $C_0 ^{(t)}$) while, with similar notation, that of ${\cB}_1$ is $B_1^{(t)}$. Since we are in the inductive situation, 
$g_1(z) = {\tilde g}_1(z)$ for all $z \in Z_1$. Proceeding as in case $j=0$, we verify: ($i$) 
 with the open set $V \subset T$ of before, $ h_1(z)={\tilde g}_1(z) $ for all $Z \in Z_1$ lying over $V$, ($ii$)  if $m_1= \max (h_1)$ and ${\tilde m}_1 = \max ({\tilde g}_1)$, then 
$m_1={\tilde m}_1$. Indeed, all what we need to apply the arguments of  the case $i=0$ is the smoothness of ${\widetilde C}_1$ over $T$, which is warranted because $\cB$ satisfies ($A$). Hence, as above, 
${\widetilde C}_1 =cl _{Z_1}( {{\widetilde C}_1}) \cap {{{\pi _1 }^{-1}}(V)}=
cl _{Z_1}( {{ C}_1}) \cap {{{\pi _1 }^{-1}}(V)} 
 \subseteq C_1$. Moreover, they are equal, because if 
$z \in C_1$, i.e., $z \in Z_1$ and $h_1(z) =m_1$ then, 
$$    m_1 =h_1(z) \le h_1^{(t)}(z) ={\tilde g}_1^{(t)}(z)=g_1^{(t)}(z)=g_1(z)={\tilde g}_1^(z) \le  {\tilde m}_1   $$
as in the proof of \ref{P:imp1}.

So,  we have that $C_1={\widetilde C}_1$. If they are non-empty, we may repeat.  
 Iterating, we  get: for all $i$, $C_i={\widetilde C}_i$, $i=0, \ldots s-1$. As in 
Proposition \ref{P:imp1} we see that $r=s$, finishing the proof.
\end{proof}

We conclude this section with some other results, to be used  in the proof of the equivalence of conditions ($A$) and ($E$).
\begin{voi}
\label{V:e6}  
(i) If  $B =(W ,I,b,E)$ is a {\it good} basic object (over a characteristic zero field $k$) and 
  $B=B_0 \leftarrow \cdots \leftarrow B_r$ is the algorithmic resolution of $B$, then $B_i$ is good, for all $i$.

(ii) Let $\cB =(W \stackrel{\pi}{\to}T,I,b,E)$ be  a family of basic objects in $\cS$ (with $T$ integral and regular)  satisfying condition ($A$), such that its associated basic object $B$ is {\it good}. Let 
$B = B _0 \leftarrow \cdots \leftarrow  B _r $ be the algorithmic equiresolution of $B$,
$B = \cB _0 \leftarrow \cdots \leftarrow  \cB _r $ its associated $T$-sequence (see \ref{V:ne6.1}). 
 Then, for $i=0, \ldots, r$ and all $t \in T$, the fiber $B_i^{(t)}$ is a good basic object. 

\smallskip

In both cases,          by induction on $j$, the result follows from the  inequality $g_{j-1}(x) \ge g_j(x')$
 (when $x' \in C_{j-1}$, the $(j-1)$-th algorithmic center) and the definition of good basic object.
\end{voi}

\smallskip

 In the next lemma we use the following  notation.  If $\cB =
 {(W \stackrel{\pi}{\to}T},I,b,E)$ is a family of basic objects  satisfying condition ($A$), 
$\cB _0 \leftarrow \cdots \leftarrow \cB_r$
 (where we write $\cB _i =(W_i \to T, I_i,b,E_i)$)  is the $T$-sequence associated to the algorithmic resolution sequence of $B=(W,I,b,E)$, we let  
  $[{\Delta}^{b-1}(I_i)]_{1,x}$ (resp. $[{\Delta}^{b-1}(I_i/T)]_{1,x}$)  denote the set of elements in the stalk $[{\Delta}^{b-1}(I_i)]_x \subset {\cO}_{W,x}$ (respectively $[{\Delta}^{b-1}(I_i/T)]_x \subset {\cO}_{W,x}$) of order $1$ (order with respect to the maximal ideal of ${\cO}_{W,x}$). 
  
 \begin{lem}
 \label{L:orduno} 
 Let $\cB$ be a family of basic objects  satisfying condition ($A$), whose associated basic object is good. 
  Then, in the notation just introduced, for every closed point  $x \in W_i$, 
 $$[{\Delta}^{b-1}(I_i)]_{1,x}=[{\Delta}^{b-1}(I_i/T)]_{1,x} $$
 \end{lem}
\begin{proof}
One of the inclusions is clear. To show the inclusion `` $\subseteq$ '', it suffices to work in the completion ${\widehat{\cO}}_{W_i,x}$ (with respect to the maximal ideal). Note that 
$R:={\widehat{\cO}}_{W,x}=k'[[x_1, \ldots,x_p, t_1, \ldots, t_q]]$ (a power series ring), where 
$t_1, \ldots, t_q$ are induced by a system of parameters of ${\cO}_{T,t}$ ($t$ the image of $x$ in $T$) and $k'$ is a suitable field. To show the desired inclusion, it suffices to show that if $f \in IR$, 
$D_1, \ldots, D_{b-1}$ are derivations of $R$, each $D_i$ of the form $\partial /\partial{x_j}$ or 
$\partial /\partial{t_s}$, such that $h=D_1 \cdots  D_{b-1}f \in R$ has order one, then each $D_i$ must be of the form $\partial /\partial{x_j}$, for a suitable $j$. This means that $h \in [{\Delta}^{b-1}(I/T)]_{1,x}R$. Since each element of 
$[{\Delta}^{b-1}(I)]_{1,x}R$ is a linear combination of elements of this form, this proves the required inclusion. 

To check the assertion on the derivations, note that if we take $f \in IR$ as above and we write 
 $ f= M_b + M_{b+1} + \cdots  $  
 (sum of homogeneous parts in $R$) then, the form $M_b$ 
involves the variables $x_1, \ldots, x_p$ only. Indeed, were this false, the element induced by $f$ in 
${\cO}_{{W_i ^{(t)}},P}$, $P$ the generic point of the component of $W_i ^{(t)}$ containing $x$, would have order $<b$, contradicting the fact that $C_i$ is $\cB _i$-permissible. Now, it is clear that if $h$ has order one, all the derivatives used in its definiton from $f$ must be with respect to a variable $x_i$.
\end{proof}

\section{The implications involving Condition ($E$)}
\label{S:E} 
We know that conditions ($A$), ($F$) and ($C$) are equivalent. Now, we shall prove that condition ($E$) is equivalent to any of them. Precisely, first we prove:
 
\begin{voi}
\label{V:e1}

 {\it Condition  ($E$) implies condition ($C$)}. We begin the proof, which ends in \ref{V:e5}. Let $\cB =(W \stackrel{\pi}{\rightarrow}T,I,b,E)$, $T$ irreducible and regular, be an $\cS$-family of basic objects satisfying condition ($E$), 
 $B = (W,I,b,E)$ its associated basic object, with 
  algorithmic resolution 
 $$(1) \quad B=B_0 \leftarrow \cdots \leftarrow B_r \, ,$$  
 $B_i = (W_i,I_i,b,E_i)$, $i=0, \ldots, r$, given by resolution functions $g_0, \ldots, g_{r-1}$, which determine algorithmic centers $C_i=\ma (g_i), ~ i=0,  \ldots, {r-1}$. Let 
 ${\pi}_i :W_i \to T$ denote the composition of the natural morphism $W_i \to W$, $i=0, \ldots, r$, and $\pi$.

 If $t \in T$, ${A_t}$ denotes the completion of ${\cO}_{T,t}$ with respect to the maximal ideal, $ { S}_t=\Spec ({A_t})$, 
 $\cB _t$ the family parametrized by $S_{t}$ induced by $\cB$ by pull-back via the natural morphism $f_t:S_{t} \to T$ sending the closed point of $S_t$ into $t$, and  
  $B_t:=aso(\cB _t)$.
 
  According to a property of the V-algorithm  that we are using, the algorithmic resolution of $B _t $:
 $$ (2) \quad B_t=(B_t)_0 \leftarrow \cdots \leftarrow (B_t)_r $$
is obtained by pull-back of the resolution (1), via the  morphism $f_t$, and the algorithmic centers    
 ${\widehat C}_{t,0}, \ldots,{\widehat C}_{t,r-1}$  used in (2) are induced by 
  the centers $C_0, \ldots, C_{r-1}$ respectively.  Note that, for  all $t \in T$,  the closed fiber of the projection $us(B_t) \to S_t$
 can be identified to the fiber $W^{(t)}$ of ${\pi}: W \to T$ at $t$. Similarly, the closed fiber of the projection 
 ${\widehat C}_{t,i} \to S_t$ can be identified to ${p_i}^{-1}(t)$, where $p_i$ denotes 
   the projection $C_i \to T$ induced by $\pi _i: W_i \to T$.
  %$C_i^{(t)}$. 
 
 To show that Part (i) in   condition ($C$) holds, i.e., that $C_i \cap W^{(t)}_i=C^{(t)}_i$, for all $t\in T$, $ i=0, \ldots, r-1$, it suffices to show that for such $t$ and  $i$, $ {\widehat C}_{t,i} \cap W^{(t)}_i=C^{(t)}_i$ (we are using the notation of Definiton \ref{D:c}, in particular $C^{(t)}_i$ denotes the $i$-th center in the algorithmic resolution of the fiber $B^{(t)}$). 
 
 Thus, 
 without loss of generality, from now on we shall assume  that  our original family $\cB$ is as follows:
    
\smallskip
 
\emph{ $\cB =(W \stackrel{\pi}{\rightarrow}T,I,b,E)$ ,   $T=S=\Spec (A)$, with $(A,\cM)$ a complete $k$-algebra, $k=A/{\cM}$ (a field of characteristic zero), and it satisfies condition ($E$).}

\smallskip

 Let $A_{n}=A/({\cM})^{n+1}$ (with $n \ge 0$ an integer), $S_{n}=\Spec (A_{n})$, and  denote by ${\cB}^{(n)}=({W^{(n)} \to S_n},I^{(n)},b,E^{(n)})$ 
  the family parametrized by $S_{n}$ induced by $\cB$ by pull-back via the natural morphism $S_{n} \to T$.

By our assumption on the validity of condition ($E$), for all $n$ the $A_{n}$-basic object ${\cB}^{(n)}$ admits an algorithmic equiresolution, with centers 
 $C^{(n)}_0, \ldots, C^{(n)}_{r(n)-1}$. Note that since the equiresolution of ${\cB}^{(n')}$ induces that of ${\cB}^{(n)}$ for $n' > n$,  the length $r(n)$ is independent of $n$, say equal to $r_0$, and necessarily $r_0 \ge r$. Let 
 $$(3) \quad  
  {\cB}^{(n)}=({\cB}^{(n)})_0 \rightarrow \cdots \rightarrow ({\cB}^{(n)})_{r_0}   $$
  be the algorithmic equiresolution of 
  ${\cB}^{(n)}$, 
  $({\cB}^{(n)})_{i}=(W^{(n)}_{i} \to S_n , {I}^{(n)}_{i},b,{E}^{(n)}_{i})$. 
  So, in this situation what we need to prove is: 
    $$(4) \quad {C}_{i} \cap {W_i}^{(0)}={C_i}^{(0)}, ~ i=0, \ldots, r-1$$
  (using the identifications discussed in \ref{R:C}, and with $0$ denoting the closed point of $S_n$).
  
Now suppose that for every natural number $n$, $C_0 \cap W^{(n)}={C_0}^{(n)}$ (note that $W^{(n)}$ is a closed subscheme of $W=W_0$). In particular, $C_0 \cap W^{(0)}=C^{(0)}_0$ (the zero-th  center of the algorithmic resolution of the fiber $B^{(0)}={\cB}^{(0)}$). 
    Hence, according to \ref{R:C}, $C_0 $ is flat over $T$. Since $C_0$ and $W$ are regular, ${I(C_0)}_w$ is generated by a regular sequence of ${\cO}_{W,w}$ for all $w \in W$, hence by  \ref{P:bu}  
  there is an identification of $(\cB^{(n)})_1$ and the pull-back of $\cB _1$, via the natural morphism $S_t \to T$. In particular, $W^{(n)}_1$ may be regarded as  a closed subscheme of $W_1$. Assume 
  $C_1 \cap W^{(n)} _1 = {C_1}^{(n)}$.  
   Reasoning as above, we conclude that $C_1 \to T$  satisfies the hypotheses of Proposition \ref{P:bu}.
  So, if $\cB _2 = \cT(\cB _1, C_1)$, then $({\cB}^{(n)})_2$ can be identified to the pull back of $\cB _2$ via the morphism $S_2 \to T$, and so on. 
  
  In view of this observation, the following statements makes sense: 
  
  \begin{itemize}
  \item[$(\alpha_i)$] 
  For all $j \le i$, for any natural number $n$, with the identifications just mentioned, $C_j \cap W^{(n)}_j=C^{(n)}_j$ (and $p_j:C_j \to T$ is flat, $j=0, \ldots  , i$).
  \end{itemize}
  
We say that the $T$-basic object $\cB$ satisfies condition ($\alpha$) if  $(\alpha _i)$  is true for all $i$. We shall prove that if $\cB$ satisfies condition ($E$), then it satisfies $(\alpha)$. In particular, $(\alpha_i)$ will hold for $n=0$ and any $i$, i.e., (4) is true. 
  
  First, we shall verify that if  $(\alpha _j)$ is valid for $j<i$ then  ($\alpha _i$) also holds (in particular, ($\alpha _0)$ will be valid without any hypothesis). 
  
 So, assume that $(\alpha)_j$ is valid for $j < i$. Then, by remark \ref{R:C} if $\pi _i : W_i \to T$ is the naturally induced projection and 
 $\cB _i = (W_i \to T, I_i,b,E_i)$, $C_0$ is a $\cB _0$-permissible center; if $\cB _1 = \cT (\cB, C_0)$, $C_1$ is $\cB _1$-permissible, and so on. Eventually we get a sequence 
 $$ (5) \quad \cB _0 \leftarrow \cB _1 \leftarrow \cdots \leftarrow \cB _i$$
 of $T$-basic objects that, by taking closed fibers,  induces the truncation 
   $$(6) \quad B^{(0)}=(B^{(0)})_0 \leftarrow \cdots \leftarrow (B^{(0)})_{i}$$
of  the algorithmic resolution 
$$   (7) \quad (B^{(0)})_0 \leftarrow (B^{(0)})_1 \leftarrow \cdots    (B^{(0)})_{r_0}   $$
of the fiber $(B^{(0)})_0 := B^{(0)} $, with resolution functions $g^{(0)}_0, g^{(0)}_1, \ldots  , g^{(0)}_{r_0 - 1}$.
 
 In particular, $W^{(0)}_i$ is (or can be identified to) a closed subscheme of $W_i$. 

In proving inductively the validity of $(\alpha _i)$ we shall consider different cases, depending on the $\o$-functions 
${\o}^{(0)}_0, {\o}^{(0)}_1, \ldots $ of the (partial) resolution (6) of the fiber $B^{(0)}$. 
 \end{voi}
 \begin{voi}
 \label{V:e2} 
 {\it Case (a)}: max$({\o}^{(0)}_i)=0$.  This assumption means that the closed fiber $(B^{(0)})_i$ is a monomial basic 
 object. Let $(C^{(0)})_i=H^{(0)}_1 \cap \ldots H^{(0)}_p$, for suitable  hypersurfaces in 
 $E^{(0)}_i$, be the canonical monomial center (re-number if necessary).  Then, for any integer $n > 0$, ${\cB}^{(n)}_i$ will also be monomial (since condition $\cE_i$ is valid for $\cB ^{(n)}$, and its $\o$-function depends on the closed fiber only), and the canonical monomial center for $(\cB ^{(n)})_i$ will be $H_1^{(n)} \cap \cdots H_p ^{(n)}$, where $H_j^{(n)}$induces $H^{(0)}_j$, for all $j$. Note that also, by \ref{V:ct} (b), the basic object $B_i$ is monomial. Since the sequence (6) of \ref{V:e1}
  is induced by the $i$-th truncation of (5) in \ref{V:e1}, there are hypersurfaces $H_1 , \ldots H_p$ in $E_i$ inducing $H^{(0)}_1 , \ldots H^{(0)}_p$ on the closed fiber respectively. We claim that 
  $H_1 \cap \cdots \cap H_p$ is the canonical monomial center of $B_i$, i.e., $C_i=
 H_1 \cap \cdots \cap H_p$. If so, it is clear that 
 $C_i \cap {W^{(n)}}_i=C^{(n)}_i$, finishing this case.
 
By the inequality $g_i(x) \le g^{(0)}_i(x)$ of \ref{V:ct} (b), the algorithmic center $C_i$ must be the intersection of some of the hypersurfaces $H_1, \ldots, H_p$. We shall see that this intersection involves $all$ these hypersurfaces.
 
 Assume, by contradiction, that $C_i$ is the intersection of fewer hypersurfaces. Working at a point  $w \in W^{(0)}_i$ (which may be assumed closed) and the completion $R$ of the local ring  ${\cO}_{W_i,w}$, then we would have the following situation. For suitable parameters $x_1, \ldots, x_p,t_1, \ldots t_q$ (with $t_i$ in $A$) and a suitable field $k'$, 
 $R=k'[[x_1, \ldots, x_p, t_1, \ldots, t_q]]$, $H_j$ is defined by $h_j \in R$, where for some $j$ we have $h_j(0, \ldots, 0,t_1, \ldots, t_q) \not= 0$. Then, $h_j$ contains a term of the form 
 ${\lambda}M(t_1, \ldots t_q)  $, with ${\lambda}$ a non-zero element of $k'$ and  $M$ a monomial, say of total degree $e$. By looking at the induced situation in ${\cB}^{(n)}$, $n$ large enough, we see that $\cB^{(n)}_i$  would  not be a
  monomial object, 
  a contradiction. This shows that 
 $C_i = H_1 \cap \ldots  \cap H_p$, as claimed.
 \end{voi}

\begin{voi} 
\label{V:e3} 
 {\it Case (b)}: max$(\o ^{(0)}_i)>0$ $and$ Max($t ^{(0)}_i$) {\it has components of codimension 1}. Here, $t^{(0)}_i$ denotes the $i$-th $t$-function of the fiber $B^{(0)}$. Each such component is known to be a regular subscheme of $W^{(0)}_i$ (see \cite{BEV}). Let $M_i$ be the union of these one-codimensional components. Then, we know that $C^{(0)}_i = M_i$. To prove the equality 
 $C_i \cap W^{(n)}_i=C^{(n)}_i$, it suffices to consider an open cover $\{U_v\}$ of  
$M_i$ and work with the restriction to each open $U_v$. We choose these opens in such a way that the object $({\cB _i}_{|{U_v}})''$ (cf. \ref{V:bdp} ) is defined. 

We work with a fixed index $v$. To simplify the notation we may assume, without loss of generality, that $U_v=W_i$, i.e. that the $S$-basic object  
 $\cB _i ''= ({W_i \to S}, I_i'',b'',E_i'')$ 
  is globally defined. Consider the relative sheaf 
  $\Delta ^{b''-1}(I_i''/S):=\overline{\Delta}$. By construction, if 
  $({\cB} ^{(n)})_i= ({{W_i} ^{(n)}\to S_n}, {I_i}^{(n)},b,{E_i}^{(n)})$ (see (3) of \ref{V:e1}),  then 
${\cB}_i ''$ induces $({\cB}^{(n)}_i)''$ and 
$\overline{\Delta}$ 
  induces the ${W_i}^{(n)}$-ideal 
  ${\Delta} ^{b''-1}({{I_i}^{(n)}}''/S_n)$. By our assumption that 
 all conditions $\cE _j$ are valid for ${\cB _i}^{(n)}$, the union of the 1-codimensional components of ${\mathcal V}(\Delta ^{b''-1}({{I_i}^{(n)}}''/S_n))$ is an $A_n$-permissible center for $\cB ^{(n)} _i$, namely the algorithmic center $C^{(n)}_j$. In particular, $C^{(n)}_i$ is smooth over $S_n$. Now let $\overline{C}$ be the union of one-codimensional components of ${\mathcal{V}}({\overline{\Delta}})$. By the remarks just made, $\overline{C}$
 induces on $\cB _i ^{(n)}$ the center $C^{(n)}_i$. In particular, 
 $\overline{C} \cap W^{(n)}_i=C^{(n)}_i$. We shall see that 
 $$(1) \quad \overline{C}=C_i$$
 proving $(\alpha _i)$ in this case ($b$).
  
 To see (1), first note that $\overline C$ is smooth over $S$, hence regular. This is a consequence of  
the criterion for smoothness of \cite{An}, Thm. 18, p. 224, and the following facts. 

\smallskip 

(i) $\overline C$ is flat over $S$. This is true  because $\overline C$
 induces on $(\cB ^{(n)})_i$ the algorithmic center $C^{(n)}_i$, which is smooth, in particular flat, over $S$. Then one uses the local criterion for flatness (\cite{M}, Theorem 21.3 (5)). 

\smallskip

(ii) The closed fiber is the center $C^{(0)}_i$, which is smooth. 

\smallskip 

So, $\overline C$ is regular, of codimension one, hence it is defined locally, at each $w$, by a single element 
 $\phi \in {\overline {\Delta}}_w \subset {\cO}_{W_i,w}$ of order one. Let $\Delta:=
 {\Delta} ^{b''-1}(I''_i)$. In general, we have ${\overline {\Delta}} \subseteq {\Delta}$, hence $\sg (B''_i) = V(\Delta) \subseteq V(\overline{\Delta})$. But we claim that if 
 $x$ is a closed point of $ \in V(\overline {\Delta})$ and the codimension of $V(\overline {\Delta})$ equals one, then 
  ${\overline {\Delta}}_x = {\Delta}_x$. This will imply that 
  $\sg (B''_i)={\rm {Max}}(t_i^{(0)})$
 has components of codimension one, so the algorithmic center $C:=C_i$ is the union of such components.  Moreover, $C= \overline C$ and $C \cap W^{(0)}_i={\overline C} \cap W^{(0)}_i$, and we know that the right-hand-side is equal to $C^{(0)}_i$, as needed. 
 
 {\noindent To check the claim, note first that condition ($E$) implies:}
 $$(2) \quad \nu ({I_i} '', {\overline C})= \nu ({I^{(0)}_i} '', {\overline C^{(0)}})$$
 Indeed, we always have the inequality $\le$. Were $<$ correct, by reducing to ${{\cB}_i}^{(n)}$, $n$ large enough, as in case (a) we would get 
 $  \nu ({I^{(n)}_i} '', {\overline {C^{(n)}}}) < \nu ({I^{(0)}_i} '', {\overline C^{(0)}})$, contradicting the fact that, by condition $({\cE}_n)$ (implied by condition ($E$)), $C^{(n)}_i$ is an $A_n$-permissible center. Thus, the equality must be true. 
 
Now, consider the completion $R={\widehat{\cO}_{W_i,x}}=k'[[x_1, \ldots, x_p,t_1, \ldots, t_q]]$, with $k'$ a suitable field, $x_1,\ldots,x_p,t_1, \ldots, t_q$ a regular system of parameters, where $t_i, \ldots,t_q$ is induced by a regular system of parameters of $A$. We may assume $x_1$ to be a generator of 
${\overline{\Delta}}_x$ (i.e., $x_1$ is the element $\phi$ considered before). The equality (2) implies that any element $\psi \in I_i''R$ must be of the form 
$\psi=x_1^{b''}.\psi_1$, $\psi_1 \in R$. Indeed, otherwise we would have 
$\nu({I_i}'',{\overline C}) < b'' = \nu({I_i^{(0)}}'',C^{(0)})$, a contradiction. Any element 
$\alpha = D_1 \ldots D_{b''-1}(\psi)$, with each $D_i$ a derivation of $R$, will be of the form $\alpha=x_1 \alpha _1$, $\alpha _1 \in R$. Since  $\alpha$ is a typical generator of $\Delta _x$, it follows that $\Delta _x \subseteq (x_1)R={\overline{\Delta}}_x$, as claimed. 
 
 \end{voi}
  
 \begin{voi}
 \label{V:e4}  
 Notice that if $\dim B =1$  then the only possible cases are those considered so far. So, we have already proved that the assertion ($\alpha_i$) is valid  for every $i$, if $\dim B=1$. So, we may proceed by induction on the dimension of $B$, and this is what we shall do in the next subsection. 
 \end{voi}
 \begin{voi} 
\label{V:e5}
{\it Case (c):  $\max(\o ^{(0)}_i)>0 ~and ~ \ma(t^{(0)}_i$)  has no component of codimension 1 (i.e., 
$\dim {\rm{Max}}(t_i^{(0)}) < d-1 $)}. First,  note that also 
$\dim {\rm{Max}}(t_i) < d-1 $ (where $t_i$ is the $i$-th $t$-function of the resolution (1) of \ref{V:e1}). Indeed, we may cover ${\rm{Max}}(t_i)$ with afine open sets $V_j$
 so that on each one the object 
 $B''_j:= (V_j,{I_i}'',{b''}_j,{E''}_i)$ of \ref{V:a12.0} and an adapted hypersurface $Z_j$ are defined. Here, 
$B_i = aso ({\cB}_i)$ (see (1) in \ref{V:e1}) and, 
to simplify, we wrote $I_i''$ rather than $({I_i}_{|{V_j}})''$ and $E''_i$ rather than ${E'' _i} _{|{V_j}}$. In the sequel, we  use similar simplifying conventions.  Note that $\sg(B''_j)={\rm{Max}}(t_i) \cap V_j$. With the same notation, let 
 $\overline{\Delta}_j:={\Delta}^{b''-1}(I_j''/S)$ and 
 ${\Delta}_j:={\Delta}^{b''-1}(I_j'')$. 
  We have 
 $\overline{\Delta}_j \subseteq {\Delta}_j$, hence (taking associated subschemes) 
 ${\mathcal V} ({\Delta}_j) \subseteq {\mathcal V}(\overline{\Delta}_j)$. So, if ${\mathcal V}({\Delta}_j)$ has components of codimension one, the same holds for ${\mathcal V}(\overline{\Delta}_j)$. But when passing to the closed fiber, 
 ${\mathcal V}(\overline{\Delta}_j)$ induces ${\mathcal V}(\Delta ((I^{(0)}_j)'')$, so 
 ${\mathcal V}(\Delta ((I^{(0)}_j)'')_{red}=V(\Delta ((I^{(0)}_j)'')= {\rm{Max}}({t^{(0)}_i}) \cap V_j  $ has components of codimension one. This implies that $\dim {\rm{Max}}(t^{(0)}_i) = d-1$, a contradiction.
 to define the algorithmic center $C_i$ for $B_i$ we must use induction on the dimension.  To be clearer  we consider separately two cases:

\smallskip
 
 {\it Case (i)}. ${\rm{max}}(t^{(0)}_{i-1}) >  {\rm{max}}(t^{(0)}_i)$ (this will be vacuously true if $i=0$). Notice that then 
 $\max (t_{i-1})=\max(t^{(0)}_{i-1}) > \max(t^{(0)}_i) \ge \max (t_i)$ (the equality by induction), that is also $\max(t_{i-1})>\max(t_i)$. 
  Recall how we get the algorithmic center $C_i$ in this case. We cover ${\rm{Max}}(g_i)$ (where $g_i$ is the $i$-th algorithmic resolution function of $B$) by open sets $V_{\lambda}$, each one admitting a nice basic object 
 $ B_{i{\lambda}}:=  {H}(({B_i}_{|V_{\lambda}})'')= 
(V_{\lambda},H (     ({  I_{i \lambda}      }        )''),b'',E''_{\lambda})$  
  where $I_{i \l}:={I_i}_{| V_{\l}}$, 
 (see \ref{V:ne9}(a), (c)), with an adapted hypersurface $Z_{i{\lambda}} \subset V_{\lambda}$ (see \ref{V:a11.0}($\beta$)). Then we take the induced basic object 
 ${B^*}_{i{\lambda}}:=(B_{i {\lambda} })_{Z_{i{\lambda}}}=(Z_{i\lambda},I_{i\lambda}^{*},b''!,E_{i \lambda}'')$, where 
 $I_{i\lambda}^{*}:={{\mathcal C}(   {H} (( I_{i \l})'')}_{Z_{i \l}})$. By induction on the dimension, when we consider the algorithmic resolution of ${B^*}_{j{\lambda}}$ there is a center 
 $C_{i{\lambda}}\subset Z_{i{\lambda}}$, which is a locally closed subscheme of $W_i=us(B_i)$. These glue together to give us the algorithmic center $C_i \subset W_i$. This construction is independent of the choice  of the open sets $V_{\lambda}$ or the hypersurfaces $Z_{i{\lambda}}$. Now, we claim:
 
 \smallskip

 ($\star$) \emph {We may cover $\ma (g_i ^{(0)})\subset W_i$ by open sets $V_{\lambda}$ as above such that moreover each  hypersurface 
 $ Z_{i{\lambda}} \subset V_{\lambda}$  is defined by a section of 
 ${\Delta}^{b''-1}(         {\cH}(({I_i}/S)'', b'')      )$ defined over $V_{\l}$. }

\smallskip
 
 This statement can be seen as follows. It suffices to work in a neighborhood of a closed point $x \in \ma (g_i^{(0)} )$. Let $x $  be such a point. Take a neighborhhod $V'$ of $x$ such that the ideals 
  ${H}({(I_i}_{|{V'}})'',b'') \supseteq {\cH}(  {{I_i}_{| V'}}/S)'' ,b'')$ are defined (for the second, use $\cB _i$ in (5) of \ref{V:e1}). To simplify, we just write 
  $H (I''_i)$ and $\cH (I_i/S '')$ respectively. Note that $R:={\widehat{\cO}_{W_i,x}}=k'[[x_1, \ldots, x_p,t_1, \ldots,t_q]]$, for a suitable field $k'$ and a regular system of parameters $x_1, \ldots, x_p,t_1, \ldots,t_q$, where $t_1, \ldots,t_q$ are induced by parameters of $A$. 
  Since     
${\Delta}^{b''-1} ({\cH}( {I_i }/S '')){\cO}_{{{ {W_i}^{(0)}} _{x}}}               
={{{\Delta}^{b''-1}}(\cH ( {I^{(0)}}_i '')})_x:=J$,               
 there is an element $ {\bar{\phi _1}}$ in $J$ such that 
   (re-numbering if necessary) for suitable derivatives ${\bar D}_1, \ldots, {\bar D}_{b''-1}$, each one with respect to a variable $x_s$, we have that the element 
 ${\bar D}_1 \ldots {\bar D}_{b''-1}{\bar{\phi _1}}$ defines a hypersurface adapted for $H {B^{(0)}_i}$. If ${\phi _1} \in \cH({I_i/S}''){\cO}_{W_i,x}$  induces ${\bar{\phi _1}}$, and $D_j$ is the derivative of 
 ${\cO}_{W_i,x}$ corresponding to ${\bar D}_{j}$, for all $j$, then 
 $\beta={ D}_1 \ldots { D}_{b''-1}{\phi _1}$ is an element of order one in 
 ${\Delta}^{b''-1}(\cH ({I_i}''/S))_x$ (and hence in ${\Delta}^{b''-1}({H}(I_i''))$), defining on  a suitable neighborhood  $V \subseteq V'$ of $x$ an adapted hypersurface, relative to $S$. This shows 
 $(\star)$.
 
 From now on we assume that our opens $V_{\lambda}$ satisfy this extra property. We shall write
 $(\cH {\cB}''_{i \lambda}):= \cH ((({{{\cB}_i})}_{|V_{\lambda}})''):= 
   (V_{\l} \to S, {\cH}( ((I_i/S)_{|{V_{\l}}}''), b'',E_{i \l}'' )  )$ (using sequence (5) of \ref{V:e1}), 
    ${\cB ^{*}}_{i \lambda}:=
  (\cH {\cB}''_{i \lambda})_{Z_{i \lambda}}
  =
  (Z_{i \lambda} \to S,           {\mathcal C}(\cH ({(I_i/S})_{|V_{\lambda}})''),Z_{i \l})             ,b'' !,E^*_{i \lambda})$.
   When we reduce over $A_n$ (or $S_n$) we shall use the superscript $(n)$ to indicate the induced object. For instance,  
  ${Z_{i \lambda}}^{(n)}$ is the hypersurface induced by ${Z_{i \lambda}}$, which is adapted for  ${\cB ^{*}_{i \lambda}}^{(n)}=
      ({Z_{i \lambda}}^{(n)} 
      \to S_n, I_{i\lambda}^{(n)},b'',(E''_{i \lambda})^{(n)})$, 
   the $A_n$-basic object induced by 
  $\cB ^{*}_{i \lambda}$. 
 
 We claim that for all possible $\lambda$, the family 
  ${\cB ^{*}}_{i \lambda}$ satisfies  condition ($E$). Indeed, since by assumption $\cB$ satisfies 
     condition ($E$), 
  $(\cH {\cB}''_{i \lambda})$ satisfies condition $(\cE _j)$ for all $j$. Since 
  $\dim \, {\rm{Max}}(t^{(0)}_i) < d-1 $, this means that the inductive $A_n$-basic object 
   ${(\cH {\cB}''_{i \lambda})}^{(n)}$ satisfies conditions ($\cE _j$), for all $j$. But there is an identification 
   ${\cB ^{*}_{i \lambda}}^{(n)}= [(\cH {\cB}''_{i \lambda})^{(n)}]_{{Z_{i{\lambda}}}^{(n)}}$. This shows that 
   ${\cB ^{*}_{i \lambda}}^{(n)}$  satisfies $(\cE)_j$ for all $j$, i.e., that 
   ${\cB ^{*}_{i \lambda}}$
 satisfies condition ($E$). 
 
 Actually, in the presentation of \cite{NE}, condition $(\cE_j)$ for 
 $(\cH {\cB}''_{i \lambda})^{(n)}$ requires that for $some$ $A_n$-adapted hypersurface $Z'$ the inductive object 
 $((\cH {\cB}''_{i \lambda})^{(n)})_{Z'}$ satisfies condition $(\cE _j)$. But it is not hard to see, with the aid of Theorem 7.13 in \cite{NE}, that if condition $(\cE _j)$ holds for $({(\cH {\cB}''_{i \lambda}})^{(n)})_{Z'}$, then it will be also true for $any$ $A_n$ adapted hypersurface $Z$, in particular our ${Z_{i \lambda}}^{(n)}$.
 
 As a consequence of the claim just proved,  
  the zeroth center ${C^*}_{i \lambda}$ of the algorithmic resolution of its associated basic objet 
  $aso(({\cB ^{*}}_{i \l})_{Z_{i \l}})=(Z_{i \l},  {{\mathcal C}( {\cH}( (I_{i \l}/S)'') , Z_{i \l}), b''!, {E_{i \l}}^{*}  })$ satisfies: 
    ${C^*}_{i \lambda} \cap Z_{i \lambda}^{(n)}= {{C^*}_{i \lambda}}^{(n)}$ (an equality of subschemes of $Z_{i \l}^{(n)}$). Since 
${{\cB}^*}_{i \l}$ satisfies condition ($E$), by induction on the dimension it satisfies ($A$). So, we may use Proposition 
 \ref{P:imp2} to conclude that 
 ${{C^*}_{i \lambda}}$ is also the zeroth center of the algorithmic resolution of the basic object 
  $  (B_i)_{Z_{i \l}} :=     (Z_{i \l}, {{\mathcal C}( {H}( (I_{i \l})'') , Z_{i \l}), b''!, {E_{i \l}}^{*}  })={{H}(({B_i}_{V_{\l}})'')}_{Z_{i \l}}$. Since 
    ${\rm{max}}(t_{i-1})<{\rm{max}}(t_{i})$, according to our definitions,  the center $C_i \subset W_i$ in the algoritmic resolution sequence \ref{V:e1} (1) satisfies: 
  $C_i \cap V_{\lambda}^{(n)}={C^*}_{i \lambda}$. Since 
  ${Z_{i \lambda}}^{(n)} \subset {V_{\lambda}}^{(n)}$ (the open of $W_i ^{(n)}$ induced by $V_{\l}$), by letting $\l$ vary we obtain 
   $C_{i \lambda}\cap {W_i}^{(n)}={C_{i \lambda}}^{(n)}$, as needed. This concludes case ($i$).

 \smallskip
 
  {\emph Case (ii)}.  ${\rm{max}}(t^{(0)}_{i-1}) =  {\rm{max}}(t^{(0)}_i)$. Let $s$ be the index such that 
 $$ (1) \quad   {\rm{max}}(t^{(0)}_{s-1})<{\rm{max}}(t^{(0)}_{s}) = \cdots = {\rm{max}}(t^{(0)}_{i}). $$
 We apply the procedure of part (i) to $M_s$, the image of the center $C^{(0)}_i$ in 
  $B_s$. We use the same notation, substituting the index $i$ by $s$. Thus we get open sets $V_{s{\lambda}}$, hypersurfaces $Z_{s{\lambda}}$ (satisfying the extra condition ($\star$) above), nice $S$-basic objects 
 $({\cH}{\cB}''_{s{\lambda}})$, and the corresponding inductive $S$-basic objects 
 ${\cB}^{*}_{s{\lambda}}=( {\cH}{\cB}''_{s{\lambda}} )_{Z_{s{\lambda}}})$. As in case (i), we see that the validity of condition ($E$) for $\cB$ implies that ${\cB}^{*}_{s{\lambda}}$ satisfies condition ($E$) for all $\lambda$. 
 
Now consider the basic object $B_{s \l}^{\vee}:=  aso ({\cB}^{*}_{s{\lambda}})=( {\cH}{\cB}''_{s{\lambda}} )_{Z_{s{\lambda}}}$ and its algorithmic equiresolution 
 $$ (2)  \quad {B}^{\vee}_{s{\lambda}} := [{B}^{\vee}_{s{\lambda}}]_0  \leftarrow [{B}^{\vee}_{s{\lambda}}]_1 \leftarrow \cdots \leftarrow [{B}^{\vee}_{s{\lambda}}]_{i-s}\leftarrow \cdots $$
 with centers $C_{j \l}^{*} \subset Z_{j \l}^{\vee}=us ([{B}^{\vee}_{s{\lambda}}]_{j-s}) $. 
 By Proposition \ref{P:imp2}, the algorithmic resolution (2) induces the algorithmic resolution of $B_{s \l}^{*}$ (we use the notation of part ($i$)). In particular, by induction on $i$, if 
 $q_{sj}:W_j \to W_s$ is the morphism induced by the sequence (1) of \ref{V:e1}, 
 $V_{j \l}:=q_{js}^{-1}(V_{s \l}) \subset W_j$, $Z_{j \l} =q_{j \l }^{-1}(Z_{s \l})\subset V_{j \l}$, then $Z_{j \l}$ may be identified to $Z_{j \l}^{\vee}$ for all $j \geq s$. As in case ($i$), we have (using the $t$-functions of (1) of \ref{V:e1}):
 $$ (2) \quad   {\rm{max}}(t_{s-1})<{\rm{max}}(t_{s}) = \cdots = {\rm{max}}(t_{i}) $$
 hence, by our definitions, as in case ($i$), with the identifications above we have:
 $$(3) \quad C_i \cap V_{i \l} = C_i ^{*}$$
 and $C_i ^{*}$ is the algorithmic center that (2) attaches to 
 ${\cH}( ({B_i} _{|V_{i \l}})''      )$. So, by (3) and the fact that ${Z_{i \lambda}}^{(n)} \subset {W_{i \lambda}}^{(n)}$, 
  we get 
  ${(C_{i}}\cap{V'}_{i{\lambda})}\cap {W_i}^{(n)}={C_{i \lambda}}^{(n)}$, which implies that 
  $C_i \cap {W_i}^{(n)}={C_i}^{(n)}$, as desired.
  
  Note that the proof just presented shows that in case $i=0$, i.e. where we have  the $S$-object $\cB$ only, the statement ($\alpha_i$) is valid, for every $i$. Since we saw that ``($\alpha_j$) is valid  for $j<i$ implies ($\alpha_i$) is valid'', by induction we have ($\alpha_i$) valid for every $i$, as needed.

\smallskip

Now we return to condition ($ii$) in Definition \ref{D:c}. To show its validity, using the algorithmic resolutions of $B$ and $B^{(0)}$ (cf. $(1)$ and $(7)$ in \ref{V:e1}), it suffices to check that $r=r_0$. Our previous discussion shows that $r \le r_0$. Assuming by contradiction that 
 $r < r_0$ holds, let $C_r^{(0)}$ be the $r$-th algorithmic resolution center of $B_r^{(0)}$. Since $\cB$ satisfies condition ($E$), we find a sequence of schemes $C_r^{(n)}$, $n$ any natural number, where  
$C_r^{(n)} \subset us({\cB}_r ^{(n)}) $ is the $r$-th algorithmic center of ${\cB}^{(n)}$. In this sequence, $C_r^{(n)}$ induces  $C_r^{(m)}$ for $n \ge m$, hence it defines a subscheme $C_r \subset us (B_r)$ smooth over $S$, inducing $C_r^{(n)}$ for all $n$. Let 
$x \in us(B_r^{(0)}) \subset us (B_r) $ be a closed point. We have 
$\nu(I_r ^{(n)},C_r^{(n)}) \ge b$ for all $n$; this implies  $ \nu(I_r,C_r) \ge b$. Since 
$\nu _x (I_r) \ge \nu(I_r,C_r)$, we'd get $\nu _x (I_r) \ge b$. This contradicts the fact that $\sg(B_r) = \emptyset$. Thus, $r=r_0$, as desired. This concludes the proof of the implication $(E) \Longrightarrow (C)$.
\end{voi}  

\smallskip
 
  \begin{voi}
\label{V:e7} 
We start the proof of the implication ($A$) $\Rightarrow $ ($E$).   

Consider a family  $\cB=({W\to T},I,b,E)$ ($T$ irreducible, regular), with associated basic object $B=(W,I,b,E)$, satisfying condition ($A$). We want to see $\cB$ satisfies condition  ($E$). As in the proof of the reverse implication (\ref{V:e1}), one is easily reduced to the case where $T=S=\Spec A$, $A=(A,M)$  a complete $k$-algebra, $k=A/{M}$ (a field of characteristic zero). Moreover, as in \ref{V:e1}, with $A_n=A/{M^{n+1}}$, $S_n:=\Spec (A_n)$, ${\cB}^{(n)}=({W ^{(n)}\to S_n},I^{(n)},b,E^{(n)})$ (the $A_n$-basic object induced by $\cB$, so that $\cB ^{(0)}=B^{(0)}$, the closed fiber of $\cB$) it suffices to show that ${\cB}^{(n)}$ satisfies condition $\cE _j$, for all $j$.

We consider the algorithmic resolution of $B$:
$$(1) \quad B=B_0\leftarrow \cdots \leftarrow B_r$$
$B_i=(W_i,I_i,b,E_i)$, with algorithmic centers $C_i \subset W_i$ and associated $S$-sequence 
$$(2) \quad \cB=\cB_0\leftarrow \cdots \leftarrow \cB_r$$
(see \ref{V:ne6.1}). By the validity of condition ($A$) (or its equivalent ($C$)), the induced sequence of closed fibers of (2):
$$(3) \quad B^{(0)}=B^{(0)}_0\leftarrow \cdots \leftarrow B^{(0)}_r$$ 
is the algorithmic resolution of the fiber $B^{(0)}$.

Now, if $C$ is a closed scheme of $W$ and $x \in C$ is a closed point, let $\nu _x(I,C)$ denote the largest integer $b$ such that $I_x \subseteq [I(C)_x]^b$ (in ${\cO}_{W,x}$). It is known that there is a dense open set $U \subseteq C$ such that 
$\nu (I,C)=\nu _x(I,C)$, for all $x \in U$ (see \ref{V:ne1.1} and \cite{NE}, 3.10). We have, for any integer $n \ge 0$:
$$(4) \quad \nu _x(I,C) \le \nu _x(I^{(n)},C^{(n)}) \le  \nu _x(I^{(0)},C^{(0)})$$
(with $C^{(n)}=C \cap W^{(n)}$). If $C=C_0$ is the $0$-th center of the algorithmic resolution (1), then for $x$ in a dense open set of $C_0$, in (4) we have 
$\nu _x(I,C)=\nu_x(I^{(0)},C^{(0)})$, hence all the inequalities in (4) are equalities. This implies that $C_0^{(n)}$ is a permissible center for ${\cB}_0^{(n)}:={\cB}^{(n)}$.

If $q_1:W_1 \to W_0=W$ is the blowing-up with center $C_0$, then by \ref{P:bu} the transform $\cT ({\cB}_0^{(n)},C_0^{(n)})$ may be identified to $\cB _1^{(n)}$, the object induced by $\cB _1$ on the scheme $W_1^{(n)}:=q_1^{-1} (W_0^{(n)})$ (identifiable to the blowing-up of $W_0^{(n)}$ with center $C_0^{(n)}$). In other words, ${\cB}_1^{(n)}={\cT}({\cB}_0^{(n)},C_0^{(n)})$.  Iterating, if condition ($A$) is valid for $B=B_0$, we obtain from sequence (1) (after identifications as indicated above) the permissible sequence of $A_n$-basic objects 
$$(5) \quad   {\cB}^{(n)} =  {\cB}_0^{(n)} \leftarrow  {\cB}_1^{(n)} \leftarrow \cdots \leftarrow  {\cB}_r^{(n)} $$ 
with centers $C_i^{(n)}=C_i \cap W_i^{(n)}$, $i=0, \ldots , r-1$.    We shall see that:
\begin{itemize}
 \item [($\alpha$)] \emph { for any $n$,  $\cB^{(n)}$ satisfies condition $\cE_i$ for all possible $i$}, 
 \item [($\beta$)] \emph {the sequence (5) is the algorithmic equiresolution of $\cB ^{(n)}$ thus determined.}
\end{itemize}
Statement ($\alpha$) is a consequence  of the  fact that, as we shall see, the following assertions ($\gamma _i$) are valid, for all $i \ge 0$:

\smallskip
\begin{itemize}
\item[ ($\gamma _i$) ]
\emph  {If $\cB ^{(n)}$ satisfies conditions $\cE _j$ for $j<i$ and the $i$-th truncation of (5) is the partial equiresolution thus determined (see $(2)_j$ in \ref{V:ne9}), then $\cE _i$ is also valid and the $i$-th algorithmic center is $C_i^{(n)}:=C_i \cap W_i^{(n)}$  (hence the $(i+1)$-truncation of (5) is the associated partial equiresolution of $\cB ^{(n)}$).}
\end{itemize}
Consider  the $\o$- and $t$-functions of (1) and (3); the $q$-th ones are respectively denoted by 
$\o _q, \o ^{(0)}_q$ and $t_q, t^{(0)}_q$. Note that condition ($A$) implies:
$$ (6)  \quad {\mathrm{max}}({\o}_j) = {\mathrm{max}}({\o}^{(0)}_j) ~ {\mathrm{and}}~{\mathrm{max}}({t}_j) = {\mathrm{max}}({t}^{(0)}_j). $$
To prove ($\gamma _i$),  there are several cases to consider:

(i) ${\mathrm{max}}({\o}_i) = 0$ (monomial case).

(ii) ${\mathrm{max}}({\o}_i) > 0$ and ${\mathrm{Max}}(t_i)$ has components of codimension one (notation of \ref{V:ne11}).

(iii) ${\mathrm{max}}({\o}_i) > 0$ and ${\mathrm{Max}}(t_i)$ has codimension $>1$.
\end{voi}

\begin{voi}
\label{V:e8}
{\it Case (i)}. Note first that by the equalities (6) of \ref{V:e7}, both $B_i$ and $B^{(0)}_i$ are monomial basic objects. Then, for all $n \ge 0$,  the $A_n$-basic object ${\cB}^{(n)}_i$ is pre-monomial (since the fiber ${\cB}^{(0)}_i$ is monomial). We claim that $B^{(n)}_i$ is monomial (hence condition $(\cE _i)$ is valid) and that the canonical center is the restriction of the $i$-th algorithmic center $C_i$ of $B_i$.

Let  $E_i=(H_{i1},\ldots,H_{is})$, $E^{(0)}_{i}=(H^{(0)}_{i1},\ldots,H^{(0)}_{is})$ (where $H^{(0)}_{ij}=H_{ij}\cap W^{(0)}_i \subset W_i$).
 By the equality 
$\Gamma_i(w)=g_i(w)=g^{(0)}_i(w)=\Gamma ^{(0)}_i(w)$ for any $w \in W^{(0)}_i$ (valid because condition ($F$), equivalent to ($A$), holds) if 
$C_i = H_{i{s_1}} \cap \cdots \cap H_{i{s_q}}$, for indices $s_1 < \cdots < s_q$, then also the $i$-th algorithmic center for $B^{(0)}_i$ is $C^{(0)}_i = H^{(0)}_{i{s_1}} \cap \cdots \cap H^{(0)}_{i{s_q}}$. Let $C^{(n)}_i:=C_i \cap W_i^{(n)}$, so $C^{(n)}_i= H^{(n)}_{i{s_1}} \cap \cdots \cap H^{(n)}_{i{s_q}}$, with $H^{(n)}_j$ induced by $H_j$, for all $j$. Since for all irreducible component $C$ of $C_i$, $\nu(I_i,C) \le \nu(I^{(n)}_i,C^{(n)}) \le \nu(I^{(0)}_i,C^{(0)})$ (with $C^{(n)}$ and $C^{(0)}$ induced by $C$ on ${\cB}^{(n)}$ and $B^{(0)}$ respectively) and  $\nu(I_i,C) = \nu(I^{(0)}_i,C^{(0)})$ (because condition ($A$) holds, see \ref{V:ne6.1}), all the inequalities above are equalities. This implies that  $C^{(n)}_i$ is monomial and that  $C^{(n)}_i$ is the center. Hence,   $\cE _i$ holds for ${\cB}^{(n)}_i$ as claimed.
\end{voi}

\begin{voi}
\label{V:e9} 
{\it Case (ii)}. By the validity of condition ($A$), we have $C_i \cap W_i^{(0)}=C_i^{(0)}$. This implies that both  ${\mathrm {Max}}(t_i)$ and ${\mathrm {Max}}(t^{(0)}_i)$ have codimension one. So, to check condition $\cE_i$ for ${\cB}^{(n)}_i$, we are in the situation  ($\beta$) of \ref{V:ne11}. Let $x$ be a closed point of $C_i$, $U$ a neighborhood of $x$ in $W_i$ such that the object $({B_i}_{|U})''=(U, {I_i}_{|U}''),b'',E_i'')$ is defined and, moreover $U \cap L = \emptyset$ for any irreducible component $L$ of  ${\mathrm {Max}}(t_i)$ of codimension $>1$.  Then,  $C_i \cap U = {\mathcal V}(\Delta ^{b''-1}({I_i}_{|U} ''))$.  
 Without loss of generality, we may assume that  $U=W_i$, then we simply write $B_i ''$, $I_i ''$, etc. We always have: 
 $\Delta ^{b''-1}({I_i} ''/S)  \subseteq   \Delta ^{b''-1}({I_i} '')$ (using (2) of \ref{V:e7} to define  $\Delta ^{b''-1}({I_i} ''/S)$).  From the fact that these ideals are locally principal, generated by an order one element, and Lemma \ref{L:orduno}, it follows that the  inclusion above is an equality. Thus,   $C_i  = {\mathcal V}(\Delta ^{b''-1}({I_i}''/S))$. Now, 
  ${\Delta}^{b''-1}({I_i}''/S)$ 
induces over $\cB ^{(n)}$ the $W^{(n)}_i$-ideal    ${\Delta}^{b''-1}((I^{(n)}_i)''/S)$. Since 
$C_i \cap W^{(0)}_i = C^{(0)}_i$, 
then $C_i \cap W^{(n)}_i= C^{(n)}_i =  {\mathcal{V}}({\Delta}^{b''-1}(({I^{(n)}_i})''/S))$, and this is a permissible $\cB^{(n)}$-center. 
  But this means that 
  ($\gamma _i$) is valid for $\cB ^{(n)}$. This proves case (ii). 
 
 \smallskip

Note that so far we did not use induction on the dimension of $\cB$. If $\dim\,{\cB}=1$, the only possible cases are those discussed so far. So, the case $\dim\,{\cB}=1$ is settled. The base being established, in the remainder of the proof we shall use induction on the dimension.
\end{voi}

\begin{voi}
\label{V:e10} 
 {\it Case (iii)}. The equality (6) of \ref{V:e7} and the fact that $C_i \cap W_i^{(0)}=C^{(0)}_i$ imply that in this case to verify ($\gamma _i$) we are in situation ($\alpha$) of \ref{V:ne11}.
 Consider the smallest index $s$ such that 
${\mathrm{max}}(t_{s-1})>{\mathrm{max}}(t_{s})= \cdots = {\mathrm{max}}(t_{i})$ 
(it could be $s=i$). We shall prove that if we assume that 
 ${\cB}^{(n)}$ satisfies condition $\cE_{s-1}$ (i.e., in case $i=0$ we make no assumption) then, it also satisfies conditions $\cE _j$   and $C_j^{(n)}$ is the corresponding algorithmic equiresolution center, $s \le j \le i$. Firstly, we recall how the algorithmic center $C_i$ of 
\ref{V:e7} (1)   is obtained in this case. Cover 
$M_s := \max (g^{(0)} _s)$ by open sets (in $W_s$) $V_{s, \lambda}$ such that on each 
$V_{s \lambda}$ the nice basic object 
$ (HB_{s \lambda}) := H(({B_s}_{|V_{s \lambda}})'') $ is defined, with inductive hypersurface 
$Z_{s \lambda}$ ($\lambda$ in a suitable set).  Then we consider the inductive objects 
$B^* _{s \lambda}:={(HB_{s \lambda})}_{Z_{s \lambda}}$. Next take the algorithmic resolution of each object $B^* _{s \lambda}$ (defined by induction on the dimension) :
$$  (1) \quad B^*_{s \lambda}:=(B^*)_s \longleftarrow {(B^*_{\lambda})}_{s+1}
\longleftarrow \cdots \longleftarrow {(B^*_{\lambda})}_i \longleftarrow \cdots    $$
obtained by using algorithmic centers 
$C_{j \l} \subset us ((B^*)_j)$, $j=s, s+1, \ldots$. The sequence (1) induces a permissible sequence 
$$(2) \quad  (HB_{s \lambda}):=(HB_{\l} )_s \longleftarrow \cdots \longleftarrow 
(HB_{\l} )_i  \longleftarrow \cdots $$ 
using the same centers $C_{j \l}$ as before. Moreover, if $Z_{j \l}$ is the strict transform of $Z_{s \l}$ to $V_{j \l} \subset W_j$ (the inverse image of $V_{s \l}$ via the morphism 
$W_{s} \leftarrow W_j$ coming from (1) of \ref{V:e7}, we have  identifications 
${(B^* _{\l})}_j = {[{(HB_\l )}_j]}_{Z_{j \l}}, ~j \ge s$.
 So, $C_{j \l}$ is a locally closed subscheme  of 
$us({(HB _{\l})}_j) = V_{j \l}$. Then, the open sets $V_{i j}$ cover $\ma (g^{(0)}_i)$ and the $i$-th algorithmic center $C_i$ in \ref{V:e7} (1)  satisfies 
${C_i}_{|V_{i \l}} = C_{i \l}$. 

The hypersurface $Z_{s \l}$ is defined, say near a closed point $x_s \in M_s \cap W_s ^{(0)}$, by a first order element of ${[\Delta ^{b''-1}(H(I_s''))]}_{x_s}$ (where, to simplify, we wrote $I_s={I_s}_{|V_{s \l}}$, similar conventions will be used in the sequel). By 
Lemma \ref{L:orduno}, 
 using sequence (2) above, $Z_{s \l}$  may be defined, near $x_s$, by an element of 
${[\Delta ^{b''-1}(\cH( (I_s'' / S) )]}_{x_s}$.
 Moreover, this can be done in such a way that this element induces an element of order one on the fibers. With this choice we have 
$$   (\Delta ^{b''-1}(\cH (I_s''/S))) {\cO}_{W_s ^{(n)}}=
\Delta ^{b''-1}(\cH (I_s''/S_n))    $$ 
and $Z_{s \l} \cap W_s ^{(n)}=Z_{s \l} ^{(n)}$ is a 
$(\cH {\cB}_{s \l} ^{(n)})$-adapted hypersurface, where 
$({\cH {\cB}}_{s \l} ^{(n)})$
 denotes the $S_n$-basic object 
${{{\cH}(\cB_s ^{(n)}}'')}_{|V_{s \l}}$.
 This implies that 
$       {\cH} ({\cB _s}_{|V_{s \l}})''  $ induces the $S_n$-basic object   
${\cB ^*}_{s \l}:= 
{{(\cH ({\cB}_{s} ^{(n)})}_{|V_{s \l}})'')}_  {Z_{s \l} ^{(n)}}   $.

Now we assert that condition ($C$) (equivalent to ($A$)) holds for ${\cB ^*}_{s \l}$, for all $\l$. The non-trivial part is to show that  if $C^* _j$ are the centers used in the algorithmic resolution of $aso ({\cB ^*}_{s \l})$, then:

\begin{itemize}
\item[ ($\alpha$) ]
 $  C^*_j \cap Z_{j \l}^{(0)}=C_j^{(0} \cap Z_{j \l}^{(0)}  $, \emph{where $Z_{j \l}$ is the closed fiber of $Z_{j \l} \to S$ and we use the notation of \ref{V:e7} (3)}.
\end{itemize}

Now, by Proposition \ref{P:BZ}, 
$ {\widetilde {\cB}}_{s \l}^* := 
{(Z_{s \l} \to S, \cC(HI_{s \l}'',Z_{s \l}),b''!,E_s)} _{|Z_s \l})$
satisfies condition ($C$), hence the centers used in the algorithmic resolution of  $aso({\widetilde {\cB}}_{s \l}^*)$ satisfy ($\alpha$). But by 
Proposition \ref{P:imp1},
 the algorithmic resolution centers of $aso({\widetilde {\cB}}_{s \l}^*)$ and 
$aso({\cB ^*}_{s \l})$ coincide, hence our claim follows. Consequently, by induction on the dimension, ${\cB ^*}_{s \l}    $ satisfies condition ($E$) and the   algorithmic equiresolution center that is associated to  ${\cB}^*_{s\lambda}$ is precisely  $C^*_{s \l} \cap W^{(n)}_{s} \cap V_{s\lambda}=C^{(n)}_s \cap V_{s\lambda}$, where $C^*_{s\lambda}$  is the algorithmic center for $B^*_{s\lambda}$. But, by construction, $C^*_{s\lambda}=C_s \cap V_{s \lambda}$ (with $C_s$ the $s$-th algorithmic center of $B$).  Hence, 
 $C^{(n)}_s \cap V_{s\lambda}=C_s \cap W^{(n)}_s $.  As explained in \ref{V:e7}, this is  ${\cH}({{\cB}^{(n)}_s})_{ |V_{s\lambda}}$-permissible. Hence, condition $\cE_s$ is valid for $\cB ^{(n)}$ and, moreover, the $s$-th algorithmic center is precisely $C_s \cap W ^{(n)} _s$.  
Now look at sequence (5) in \ref{V:ne12}. We have  
${\cB}_{s+1}^{(n)}={\cT}({\cB}_s^{(n)},C_s^{(n)})$. With the previous notation, let 
${\cH}{\cB}_{s+1 \l}:={\cH}(({\cB}_{s+1 \l})'')$.
 As in  
\cite{NE}, 6.10 (Giraud's Lemma)  for every $\l$ the strict transform $Z_{{s+1}\l}$ of $Z_{s \l}$ to $V_{{s+1} \l}$ is an $S$-adapted hypersurface, locally defined by a section of the sheaf 
 $({\Delta}^{b''-1}(({\cH}(I_{s+1}/S))'')$.  We  may repeat the argument above to conclude that:
 $Z_{{s+1}  \l}$ induces  a  $({H{B}^{(n)}}''_{{s+1} \lambda})$-adapted hypersurface  $Z^{(n)}_{{s+1},\l}$;    $B^*_{{s+1}\lambda}$ induces 
a ${\cH}{\cB}_{s+1 \l}^{(n)}$; and   
${\cH}( ({{\cB}_{s+1 \l}}_{|V_{s+1 \l}}       )''      )$  
induces the $S_n$-basic object 
${{\cB}^*}_{s+1 \l}:={( {\cH}( ({{{\cB}^{((n))}}_{s+1}}_{|{V_{s+1 \l}}})'' ) )}_{Z_{s+1 \l}}$. 

Again by induction on the dimension we see that   
${\cB}^*_{{s+1}\lambda}$ 
 satisfies condition ($E$) and, as before, using Proposition \ref{P:imp2},   the algorithmic equiresolution center that is associated to  ${\cB}^*_{{s+1}\lambda}$ is precisely  $C^*_{s+1} \cap W^{(n)}_{{s+1}} \cap V_{{s+1}\lambda}=C^{(n)}_{s+1} \cap V_{{s+1}\lambda}$, where $C^*_{{s+1}\lambda}$  is the algorithmic center for $B^*_{{s+1}\lambda}$ and 
$C^{(n)}_i \cap V_{i\lambda}=C_i \cap W^{(n)}_i $. According to \ref{V:e7}, this is ${{\cH}(({{\cB}^{(n)}_{s+1}})'')}_{ |V_{{s+1}\lambda}}$-permissible. Hence, condition $\cE_{s+1}$ is valid for $\cB ^{(n)}$ and moreover the ${s+1}$-th algorithmic center is precisely $C_{s+1} \cap W ^{(n)} _{s+1}$.   
 Iterating,  we get that analogous results are valid for $j=s, \ldots, i$. In particular,  condition $\cE _j$ holds for $\cB ^{(n)}$ and the $j$-th algorithmic equiresolution center is precisely $C^{(n)}_j$, as claimed.
\
Finally, statement ($\beta$) in \ref{V:e7} follows from the fact that since $\cB$ satisfies condition ($A$) (or ($C$)), the length $r$ of sequence (1) in \ref{V:e7} is equal to the length $r_0$ of the algorithmic resolution sequence of the closed fiber $B^{(0)}$. Thus, the fact that $(A)$ implies ($E$) is proved.
\end{voi} 

\smallskip
Summarizing, we have proved that Conditions $(A)$ and $(C)$ are equivalent. Hence, in view of our results in Section \ref{S:BE}, we have:
\begin{thm}
 \label{T:mango} Conditions $(A)$, $(F)$, $(C)$ and $(E)$ are equivanent for families of basic objects parametrized by a regular schem $T$ (in the class $\cS$).
\end{thm}

\section{Families of ideals and varieties}
\label{S:IV}

\begin{voi}
\label{V:I1} {\it Idealistic triples}. An {\it idealistic triple} (or {\it id-triple}) over a scheme $T \in \cS$ (or, simply, a $T$-triple) is a $3$-tuple 
${\mathcal T}=(\pi:W \to T,I,E)$, where $\pi$ is a smooth morphism, $I$ is a $W$-ideal such that for all closed points $t \in T$ the $W^{(t)}$-ideal $I{\cO}_{W^{(t)}}$ is never zero (with $W^{(t)}:={\pi}^{-1}(t)$), and $E=(H_1, \ldots,H_n)$ is a collection of $T$-hypersurfaces having normal crossings (see \ref{V:ne0}). \\
If $t$ is a closed point of $T$, there is a naturally defined notion of {\it fiber} of $ \cT$ at $t$ (by reducing modulo the maximal ideal of ${\cO}_{T,t}$). More generally, this notion is valid for any point $t \in T$. Indeed, firstly we consider the naturally induced id-triple over $T_t=\Spec ({\cO}_{T,t}) $ and then the induced id-triple over the closed point of $T_t$ (see \ref{V:ne0.1}). When $T=\Spec k$, $k$ a field clear from the context, we often will simply write ${\mathcal T}=(W,I,E)$.
\end{voi}

\begin{voi}
\label{V:1.1} In a similar way, we can define, as in \cite{BEV}, page 393, the notion of {\it family of ideals}. A  family of ideals, parametrized by a scheme $T$ (in $\cS$) is a pair ${\mathcal I}=(\pi:W \to T,I)$, where $T$ and $I$ are as in \ref{V:I1} We define the notion of fiber as above. Given a family of ideals $(\pi:W \to T,I) $ there is a canonically associated id-triple, namely $(\pi:W \to T,I,\emptyset)$.
\end{voi}

\begin{voi}
\label{V:I2} Given the id-triple $\cT=(\pi:W \to T,I,E)$, $E=(H_1, \ldots,H_n)$, a subscheme $C$ of $W$ having normal crossings with $E$ relative to $T$ (see \ref{V:ne0}) is called a $T$-{\it permissible center} (or a ${\mathcal T}$-center). If $C$ is a $T$-permissible center, we define the transform of $\cT$ with center $C$. This is the id-triple ${\cT}_1=({\pi}_1 :W_1 \to T,I'_1,E_1)$, where ${\pi}={\pi}{q}_1$ (with $q_1:W_1 \to W $ the blowing-up of $W$ with center $C$), $I_1'=I{\cO}_{W_1}$, and $E_1=(H'_1, \ldots, H_{n+1}')$ is defined as in \ref{V:ne6}. The process of replacing an id-triple $\cT$ by a transform (with a $T$-permissible center) is called a {\it permissible transformation} of the triple.

A $T$-principalization of an $id$-triple ${\cT}$ is a sequence of $T$-permissible transformations 
${\cT}={\cT}_0 \leftarrow {\cT}_1 \leftarrow \cdots \leftarrow {\cT}_r$ of id-triples (we write 
${\cT}_i=({\pi}_i:W_i \to T, I'_i,E_i)$), having the following properties: \\
(i) for all $z \in W_r$, the stalk $(I'_r)_z$ is generated by an element 
$f_z \in {\cO}_{W_r,z} $, not in $r({\cO}_{T,t}){\cO}_{W_r,z}$, where $t={\pi}_{r}(t)$, moreover 
${\mathcal V}(I'_r)$ and $E_i$ have normal crossings.
 \\
(ii) The composition morphism $q: W_r \to W$ induces an isomorphism  $q^{-1}(U) \cong U$ where $U=W \setminus V(I)$.

By taking fibers, for each $t \in T$ such an $T$-principalization induces a principalization sequence for ${\mathcal T}^{(t)}$, in the sense of \cite{BEV}, Theorem 2.5.

It is known that when the base is a field $k$ (of characteristic zero), the V-algorithm of resolution for basic objects  induces an algorithm for {\it principalization} of id-triples. Namely, given the id-triple 
 $\mathcal T=(W,I,E)$ over a field $k$, one considers the basic object 
 $B_0= (W,I,1,E)$ and applies the algorithm to $B_0$, getting a resolution 
 $B=B_0 \leftarrow \cdots \leftarrow B_r$, $B_i=(W_i,I_i,1,E_i)$, $i=0, \ldots r$. Then,  by dropping the entry $b=1$ in each basic object $B_i$, for a suitable $s \le r$, we get the desired principalization. Namely, $s$ is the first index such that $\max (\o _s)=0$. See \cite{BEV}, Parts I and II, for details. Henceforth this principalization process will be referred to as {\it the algorithmic principalization  of} $\mathcal T=(W,I,E)$.
 \end{voi}
 
\begin{voi}
\label{V:I3} The different equiresolution conditions for basic objects discussed in  Sections \ref{S:FB} and \ref{S:CE} naturally induce analagous notions for families of id-triples and of ideals. Namely, we say that the id-triple $\cT=(W\to T,I,E)$ (with $T$ smooth) satisfies condition ($A$) (resp. ($F$), ($C$), ($\tau$), ($E$)) if the $T$-basic object 
${\cB}={(W\to T,I,1,E)}$ does so.

From Theorem \ref{T:mango} it follows that conditions $(A)$, $(F)$, $(C)$ and $(F)$ on an id-triple $\cT=(W\to T,I,E)$, $T$ smooth, are equivalent. Moreover, any of these is equivalent to ($\tau$) if all the projections $C_i \to T$ are proper (notation of \ref{V:ne2}, see \ref{P:tau}). However, condition ($E$) makes sense for an arbitrary parameter scheme $T \in \cS$.  
 
We shall say that the $T$-triple $\cT$ ($T$ arbitrary)  is {\it equisolvable} if condition ($E$) holds. If so, $\cT$ induces the algorithmic resolution of each fiber. If $T$ is smooth, for $\cT$ to be equisolvable, it is equivalent to require  the validity of condition ($A$), ($F$) or ($C$) (or ($\tau$), if $W \to T $ is proper).

If $\cT$ is equisolvable and in the algorithmic equiresolution sequence of the associated $T$-basic object $\cB$ we systematically delete the third entry (equal to 1), we obtain an equiresolution of $\cT$, called its {\it algorithmic equiresolution}. 

Finally, given a family of ideals 
$\mathcal I=(W \to T, I)$ (with $T$ in $\cS$), we say that it is equisolvable  if its associated family of id-triples 
$\cT=(W \to T, I, \emptyset)$ is equisolvable.  We define the {\it algorithmic equiresolution} of $\mathcal I$ in an obvious way from that of the associated family of id-triples, by deleting the last entry of each term.
\end{voi}

\begin{voi}
\label{V:I4}
Working over a characteristic zero field $k$, we say that an {\it embedded scheme } 
 is a pair $\mathcal X = (X, W)$ where $W$ is a scheme, smooth over $k$, and $X$ is a equidimensional subscheme of $W$. If, moreover, $X$ is reduced, we shall talk about an {\it embedded variety }.
 
 A {\it resolution} of an embedded variety $\cX$ is a proper, birational morphism ${f:W' \to W}$, with $W'$ smooth, such that: (i) the exceptional locus of $f$ is the  union of regular hypersurfaces 
 $H_1, \ldots, H_n$ with normal crossings, (ii) the strict transform $X'$ of $X$ to $W'$ is regular, and has normal crossings with $H_1, \ldots, H_n$, (iii) $f$ induces an isomorphism $X' \setminus f^{-1}(\Sigma) \to X \setminus \Sigma$, where $\Sigma$ denotes the singular locus of $X$.

 As explained in sections (2.4) and (5.8) of \cite{BEV}, the V-algorithm for resolution of basic objects induces an algorithm for resolution of an embedded variety $\cX =(X,W)$. Let us review this method. Consider the basic object $B=(W,I(X),1,\emptyset)$ and its corresponding algorithmic resolution:
 $$B=B_0 \leftarrow {B}_1 \leftarrow \cdots \leftarrow {B}_r $$
 obtained via resolution functions $g_0, \ldots, g_{r-1}$, taking values in a totally ordered set ${\Lambda}^{(d)}$, $d=\dim(W)$ 
 (which yield algorithmic centers $C_i={\mathrm{Max}}(g_i), \, i=0, 
 \dots r-1$). We write 
 $B_i = (W_i, I_i, 1, E_i)$, for all $i$. 
 Then, $g_0$ is constant, say equal to $a \in {\Lambda}^{(d)}$ on $W \setminus \sg (X)$ and there is a unique index $n$ (depending on $B$, hence on $\cX$) such that max($g_n$)=$a$.  Moreover, the strict transform $X_n$ of $X$ to $W_n$ is a union of components of the algorithmic center $C_n = {\mathrm {Max}}(g_n)$ (i.e., the $n$-th center in the algorithmic resolution process). Hence, $X_n$ is regular, and it has normal crossings with $E_n$.  We shall denote the index $n$ above by $\eta({\cX})$.  

 The sequence
 $$\cX=(W, \emptyset)=(W_0,\emptyset) \leftarrow (W_1,E_1) \leftarrow \cdots \leftarrow (W_{\eta({\cX})},E_{\eta({\cX})})$$
 where $W_i$ and $E_i$ are obtained as above (i.e., for $i=1, \ldots \eta(\cX)$, $W_{i} \to W_{i-1}$ is the blowing up with center $C_{i-1}$ and  $E_i $ is the exceptional divisor of the induced morphism $W_i \to W$), is called the {\it algorithmic resolution of $\cX$}. Then the induced morphism 
 $W_{\eta(\cX)} \to W$ is indeed a resolution of $\cX$.
 \end{voi}

 \begin{voi}
 \label{V:I5}
 {\it Families of embedded schemes}. 
  An $\cS$-family of $d$-dimensional embedded schemes is a pair $\cX = (X, p:W \to T)$,  where $T \in \cS$, $X$ is a closed subscheme of $W$ and   $p$ is smooth, such that: (a) for all $t \in T$ the pair of fibers $\cX ^{(t)} = (X^{(t)}, W^{(t)})$, $t \in T $ 
  is an embedded scheme over $k(t)$ (i.e., $X^{(t)}$ is equidimensional) and (b) the morphism $X \to T$ induced by $p$ is flat. 
 
 We shall say that a family $\cX$ of embedded schemes is {\it equisolvable} if the associated family of ideals $ (p:W \to T, I(X)) $ is equisolvable, i.e., the $T$-basic object $ (p:W \to T, I(X),b,\emptyset) $ satisfies condition ($E$). If 
 $T$ is  a regular scheme this is equivalent to require that 
 any of the equivalent conditions ($A$), ($F$), ($C$) hold.
 
 If the family of embedded schemes $\cX=(X,p:W \to T)$ ($T$ arbitrary in $\cS$) is equisolvable, then the algorithmic equiresolution sequence of the $T$-basic object 
  ${\cB}(\cX):= {(W \to T, I(X),1,\emptyset)} $
 induces, for any $t \in T$, the algorithmic resolution of the fiber    
 $B(\cX)^{(t)}= {(W^{(t)}, I(X^{(t)}),b,\emptyset)} $. In particular, if $W^{(t)}$ is an algebraic variety (i.e., it is reduced), 
 as explained in \ref{V:I4}, the $\eta(\cX)$-truncation of the algorithmic resolution of  $B(\cX)^{(t)}$ yields the algorithmic resolution of the embedded variety ${\cX}^{(t)}$.

 Suppose now that $T$ is smooth and irreducible and that, for $t$ in a non-empty subset $U$ of $T$,  the fiber $X^{(t)}$ is an algebraic varities.  In particular, the generic fiber is an algebraic variety (hence reduced). This implies (e.g., by using condition ($F$)) that for all $t \in U$, 
 $\eta(X^{(t)})=\eta(X)$. That is, the length of the algorithmic resolution sequence of each fiber 
 is the same. Thus, the algorithmic resolution of embedded scheme $\cX$ induces that of each fiber ${\cX}^{(t)}$ (an embedded variety over the field $k(t)$).
 
 As remarked in \cite{BEV} page 393, the advantage of the present approach (using the ideal $I(X)$) is that it  makes sense even in the case where some fibers are not reduced.  
\end{voi}

\providecommand{\bysame}{\leavevmode\hbox to3em{\hrulefill}\thinspace}

\end{document}